\documentclass{article}
\usepackage{amsmath}
\usepackage{amsthm}

\usepackage{textcomp, arcs}
\usepackage{times}
\usepackage{xcolor}
\usepackage[title]{appendix}
\usepackage{enumitem}
\usepackage{xparse,etoolbox}
\usepackage{caption}
\usepackage{chngcntr}
\counterwithin{figure}{section}
\counterwithin{table}{section}
\usepackage{tikz}
\usepackage{graphicx}
\usepackage[mathscr]{eucal}
\usepackage{scrextend}
\usepackage{wasysym}
\usepackage{etoolbox}
\makeatletter
\date{}

\usepackage[all]{xy}

\usepackage{etoolbox}
\makeatletter
\patchcmd{\@makechapterhead}{50\p@}{\chapheadtopskip}{}{}

\patchcmd{\@makeschapterhead}{50\p@}{\chapheadtopskip}{}{}
\makeatother
\newlength{\chapheadtopskip}
\usepackage[parfill]{parskip} 
\usepackage{amssymb}
\usepackage{imakeidx}
\usepackage{enumitem}
\usepackage[mathlines]{lineno}
\DeclareMathAlphabet{\mathpzc}{OT1}{pzc}{m}{it}
\makeatletter
\newcommand{\mylabel}[2]{#2\def\@currentlabel{#2}\label{#1}}
\makeatother
\usepackage{fancyhdr}
\usepackage{enumitem}
\usepackage{xpatch}
\newtheorem{theorem}{Theorem}[section]
\newtheorem{lemma}[theorem]{Lemma}
\newtheorem{obs}[theorem]{Observation}
\newtheorem{defn}[theorem]{Definition}
\newtheorem{prop}[theorem]{Proposition}
\newtheorem{cor}[theorem]{Corollary}
\newtheorem{fact}[theorem]{Fact}

\newtheorem{subclaim}{Subclaim}[theorem]
\newcounter{claimlevel}[theorem]

\ExplSyntaxOn
\NewDocumentEnvironment{claim}{O{=}}
 {
  \str_case:nn { #1 }
   {
    {=}  { }
    {+}  { \stepcounter{claimlevel} }
    {-}  { \addtocounter{claimlevel}{-1} }
   }
  \begin{ Claim \int_to_Roman:n { \value{claimlevel} } }
 }
 {
  \end{ Claim \int_to_Roman:n { \value{claimlevel} } }
 }
\ExplSyntaxOff

\newenvironment{claimproof}[1]{\par\noindent\underline{Proof:}\space#1}{\hspace{1mm}$\blacksquare$}

\usepackage[left=2.54cm,top=2.54cm,right=2.54cm,bottom=2.54cm]{geometry}
\usepackage[nodisplayskipstretch]{setspace}
\newlength\FHoffset
\fancyheadoffset{\FHoffset}
\newlength\FHright
\makeatletter
\usepackage{framed}

 \newtheoremstyle{TheoremNum}
        {\topsep}{\topsep}              
        {\itshape}                      
        {}                              
        {\bfseries}                     
        {.}                             
        { }                             
        {\thmname{#1}\thmnote{ \bfseries #3}}
    \theoremstyle{TheoremNum}
    \newtheorem{thmn}{Theorem}

\newtheoremstyle{PropNum}
        {\topsep}{\topsep}              
        {\itshape}                      
        {}                              
        {\bfseries}                     
        {.}                             
        { }                             
        {\thmname{#1}\thmnote{ \bfseries #3}}
    \theoremstyle{PropNum}

\newtheoremstyle{LemmaNum}
        {\topsep}{\topsep}              
        {\itshape}                      
        {}                              
        {\bfseries}                     
        {.}                             
        { }                             
        {\thmname{#1}\thmnote{ \bfseries #3}}
    \theoremstyle{LemmaNum}
    
\makeatletter
\renewcommand\subitem{\@idxitem\nobreak\hspace*{20\p@}}
\renewcommand\subsubitem{\@idxitem\nobreak\hspace*{20\p@}}
\makeatother

\makeatother
\title{Distant 2-Colored Components on Embeddings: Part II\\ The Short-Inseparable Case}
\author{Joshua Nevin}
\begin{document}
\maketitle

\begin{center}\textbf{Abstract}\end{center} This is the second in a sequence of three papers in which we prove the following generalization of Thomassen's 5-choosability theorem: Let $G$ be a graph embedded on a surface of genus $g$. Then $G$ can be $L$-colored, where $L$ is a list-assignment for $G$ in which every vertex has a 5-list except for a collection of pairwise far-apart components, each precolored with an ordinary 2-coloring, as long as the face-width of $G$ is at least $2^{\Omega(g)}$ and the precolored components are of distance at least $2^{\Omega(g)}$ apart. This provides an affirmative answer to a generalized version of a conjecture of Thomassen and also generalizes a result from 2017 of Dvo\v{r}\'ak, Lidick\'y, Mohar, and Postle about distant precolored vertices. In this paper we prove that the above result holds for a restricted class of embeddings, i.e those embeddings which satisfy certain triangulation conditions and do not have separating cycles of length at most four.

\section{Background}

All graphs in this paper have a finite number of vertices. Given a graph $G$, a \emph{list-assignment} for $G$ is a family of sets $\{L(v): v\in V(G)\}$, where each $L(v)$ is a finite subset of $\mathbb{N}$. The elements of $L(v)$ are called \emph{colors}. A function $\phi:V(G)\rightarrow\bigcup_{v\in V(G)}L(v)$ is called an \emph{$L$-coloring of} $G$ if $\phi(v)\in L(v)$ for each $v\in V(G)$, and $\phi(x)\neq\phi(y)$ for any adjacent vertices $x,y$. Given an $S\subseteq V(G)$ and a function $\phi: S\rightarrow\bigcup_{v\in S}L(v)$, we call $\phi$ an \emph{ $L$-coloring of $S$} if $\phi$ is an $L$-coloring of the induced graph $G[S]$. A \emph{partial} $L$-coloring of $G$ is an $L$-coloring of an induced subgraph of $G$. Likewise, given an $S\subseteq V(G)$, a \emph{partial $L$-coloring} of $S$ is a function $\phi:S'\rightarrow\bigcup_{v\in S'}L(v)$, where $S'\subseteq S$ and $\phi$ is an $L$-coloring of $S'$. Given an integer $k\geq 1$, $G$ is called \emph{$k$-choosable} if it is $L$-colorable for every list-assignment $L$ for $G$ such that $|L(v)|\geq k$ for all $v\in V(G)$. Thomassen demonstrated in \cite{AllPlanar5ThomPap} that all planar graphs are 5-choosable. Actually, Thomassen proved something stronger. 

\begin{theorem}\label{thomassen5ChooseThm}
Let $G$ be a planar graph with facial cycle $C$. Let $xy\in E(C)$ and $L$ be a list assignment for $V(G)$ such that each vertex of $V(G\setminus C)$ has a list of size at least five and each vertex of $V(C)\setminus\{x,y\}$ has a list of size at least three, where $xy$ is $L$-colorable. Then $G$ is $L$-colorable.
\end{theorem}

In this sequence of papers, we also rely on the following very useful result from \cite{2ListSize2PaperSeriesI} which is an analogue of Theorem \ref{thomassen5ChooseThm} where the precolored edge has been replaced by two lists of size two.

\begin{theorem}\label{Two2ListTheorem} 

Let $G$ be a planar graph,  let $F$ be a facial subgraph of $G$, and let $v,w\in V(F)$. Let $L$ be a list-assignment for $V(G)$ where $|L(v)|\geq 2$, $|L(w)|\geq 2$, and furthermore, for each $x\in V(F)\setminus\{v,w\}$, $|L(x)|\geq 3$, and, for each $x\in V(G\setminus F)$, $|L(x)|\geq 5$. Then $G$ is $L$-colorable. \end{theorem}

We now recall some notions from topological graph theory. Given an embedding $G$ on surface $\Sigma$, the deletion of $G$ partitions $\Sigma$ into a collection of disjoint, open connected components called the \emph{faces} of $G$. Our main objects of study are the subgraphs of $G$ bounding the faces of $G$. Given a subgraph $H$ of $G$, we call $H$ a \emph{facial subgraph} of $G$ if there exists a connected component $U$ of $\Sigma\setminus G$ such that $H=\partial(U)$. We call $H$ is called a \emph{cyclic facial subgraph} (or, more simply, a \emph{facial cycle}) if $H$ is both a facial subgraph of $G$ and a cycle. Given  a cycle $C\subseteq G$, we say that $C$ is \emph{contractible} if it can be contracted on $\Sigma$ to a point, otherwise we say it is \emph{noncontractible}. We now introduce two standard paramaters that measure the extent to which an embedding deviates from planarity. 

\begin{defn}\label{EWandFWDefn} \emph{Let $\Sigma$ be a surface and let $G$ be an embedding on $\Sigma$. The \emph{edge-width} of $G$, denoted by $\textnormal{ew}(G)$, is the length of the shortest noncontractible cycle in $G$. The \emph{face-width} of $G$, denoted by $\textnormal{fw}(G)$, is the smallest integer $k$ such that there exists a noncontractible closed curve of $\Sigma$ which intersects with $G$ on $k$ points. If $G$ has no noncontractible cycles, then we define $\textnormal{ew}(G)=\infty$, and if $g(\Sigma)=0$, then we define $\textnormal{fw}(G)=\infty$. The face-width of $G$ is also sometimes called the \emph{representativity} of $G$. Some authors consider the face-width to be undefined if $\Sigma=\mathbb{S}^2$, but, for our purposes, adopting the convention that $\textnormal{fw}(G)=\infty$ in this case is much more convenient. The notion of face-width was introduced by Robertson and Seymour in their work on graph minors and has been studied extensively.} \end{defn}

In this paper, our primary interest lies in embeddings without short separating cycles, so we define the following

\begin{defn} \emph{Let $\Sigma$ be a surface and let $G$ be an embedding on $\Sigma$. A \emph{separating cycle} in $G$ is a contractible cycle $C$ in $G$ such that each of the two connected components of $\Sigma\setminus C$ has nonempty intersection with $V(G)$. We call $G$ \emph{short-inseparable} if $\textnormal{ew}(G)>4$ and $G$ does not contain any separating cycle of length 3 or 4.}\end{defn}

We now define the following. 

\begin{defn} \emph{Let $\Sigma$ be a surface and let $G$ be an embedding on $\Sigma$. We say that $G$ is a \emph{2-cell embedding} if each component of $\Sigma\setminus G$ is homeomorphic to an open disc. We say that $G$ is a \emph{closed 2-cell embedding} ifthe closure of each component of $\Sigma\setminus G$ is homeomorphic to a closed disc.}\end{defn}

If $G$ is a 2-cell embedding, then $\textnormal{fw}(G)$ can be alternatively regarded as the smallest integer $k$ such that there exists a collection of $k$ facial subgraphs of $G$ whose union contains a noncontractible cycle of $G$. In practice, we are usually working with a 2-cell embedding, and in that case, we mostly use the above definition of $\textnormal{fw}(G)$ rather than that of Definition \ref{EWandFWDefn}, as it is usually easier to work with for our purposes. Lastly, in this paper, we rely on the following simple but useful result, which is a consequence of a characterization in \cite{lKthForGoBoHm6} of obstructions to extending a precoloring of a cycle of length at most six in a planar graph. 

\begin{theorem}\label{BohmePaper5CycleCorList} Let $G$ be a short-inseparable planar embedding with facial cycle $C$, where $5\leq |V(C)|\leq 6$ . Let $G'=G\setminus C$ and $L$ be a list-assignment for $G$, where $|L(v)|\geq 5$ for all $v\in V(G')$. Let $\phi$ be an $L$-coloring of $V(C)$ which does not extend to $L$-color $G$. If $|V(C)|=5$, then $G'$ consists of a lone vertex $v$ adjacent to all five vertices of $C$, where $L_{\phi}(v)=\varnothing$. On the other hand, if $|V(C)|=6$, then $G'$ consists of one of the following.
\begin{enumerate}[label=\roman*)]
\itemsep-0.1em
\item A lone vertex $v$ adjacent to at least five vertices of $C$, where $L_{\phi}(v)=\varnothing$; OR
\item An edge where, for each $v\in V(G')$, $L_{\phi}(v)$ is the same 1-list and $G[N(v)\cap V(C)]$ is a length-three path; OR
\item A triangle where, for each $v\in V(G')$, $L_{\phi}(v)$ is the same 2-list and $G[N(v)\cap V(C)]$ is a length-two path.
\end{enumerate}
\end{theorem}

Our main result for this sequence of three papers is Theorem \ref{5ListHighRepFacesFarMainRes} below. Note that the pairwise-distance lower bound does not depend on the number of special faces or their sizes. 

\begin{theorem}\label{5ListHighRepFacesFarMainRes} Let $\Sigma$ be a surface, $G$ be an embedding on $\Sigma$ of face-width at least $2^{\Omega(g(\Sigma))}$, and $F_1, \ldots, F_m$ be a collection of facial subgraphs of $G$ which are pairwise of distance at least $2^{\Omega(g(\Sigma))}$ apart. Let $x_1y_1, \ldots, x_my_m$ be a collection of edges in $G$, where $x_iy_i\in E(F_i)$ for each $i=1,\ldots, m$. Let $L$ be a list-assignment for $G$ such that 
\begin{enumerate}[label=\arabic*)]
\itemsep-0.1em
\item for each $v\in V(G)\setminus\left(\bigcup_{i=1}^mV(C_i)\right)$, $|L(v)|\geq 5$; AND
\item For each $i=1,\ldots, m$, $x_iy_i$ is $L$-colorable, and, for each $v\in V(F_i)\setminus\{x_i, y_i\}$, $|L(v)|\geq 3$.
\end{enumerate}
Then $G$ is $L$-colorable.  \end{theorem}

An immediate consequence of Theorem \ref{5ListHighRepFacesFarMainRes} is the following slightly weaker result about pairwise far-apart components with ordinary 2-colorings. 

\begin{theorem}\label{WVersionThmPrecCompFW}  Let $\Sigma$ be a surface, $G$ be an embedding on $\Sigma$ of face-width at least $2^{\Omega(g(\Sigma))}$, and $L$-be a list-assignment for $V(G)$ in which every vertex has a list of size at least five, except for the vertices of some connected subgraphs $K_1, \cdots, K_m$ of $G$ which are pairwise of distance at least $2^{\Omega(g(\Sigma))}$ apart, where, for each $i=1, \cdots, m$, there is an $L$-coloring of $K_i$ which is an ordinary 2-coloring. Then $G$ is $L$-colorable. 
 \end{theorem}
The above result can also be viewed as a generalization of the main result of \cite{DistPrecVertChoosePap}, which in turn provided an affirmative answer to a conjecture of Albertson. We now introduce some topological notions used throughout this paper. Unless otherwise specified, all graphs are regarded as embeddings on a previously specified surface, and all surface are compact, connected, and have zero boundary. If we want to talk about a graph $G$ as an abstract collection of vertices and edges, without reference to sets of points and arcs on a surface then we call $G$ an \emph{abstract graph}. 

\begin{defn}\label{ContractNatCPartDefn}\emph{Let $\Sigma$ be a surface, let $G$ be an embedding on $\Sigma$, and let $C$ be a contractible cycle in $G$. Let $U_0, U_1$ be the two open connected components of $\Sigma\setminus C$. The unique \emph{natural $C$-partition} of $G$ is the pair $\{G_0, G_1\}$ of subgraphs of $G$ where, for each $i\in\{0,1\}$, $G_i=G\cap\textnormal{Cl}(U_i)$.} \end{defn}

\begin{defn}
\emph{Given a graph $G$, a subgraph $H$ of $G$, a subgraph $P$ of $G$, and an integer $k\geq 1$, we call $P$ a \emph{$k$-chord} of $H$ if $|E(P)|=k$ and $P$ is of the following form.}
\begin{enumerate}[label=\emph{\arabic*)}]
\itemsep-0.1em
\item \emph{$P:=v_1\cdots v_kv_1$ is a cycle with $v_1\in V(H)$ and $v_2, \cdots, v_k\not\in V(H)$}; OR
\item \emph{$P:=v_1\cdots v_{k+1}$, and $P$ is a path with distinct endpoints, where $v_1, v_{k+1}\in V(H)$ and $v_2,\cdots, v_k\not\in V(H)$.}
\end{enumerate}

\end{defn}

$P$ is called \emph{proper} if it is not a cycle, i.e $P$ intersects $H$ on two distinct vertices. Otherwise it is called \emph{improper}. Note that, for $1\leq k\leq 2$, any $k$-chord of $H$ is proper, as $G$ has no loops or duplicated edges. A 1-chord of $H$ is simply referred to as a \emph{chord} of $H$. In some cases, we are interested in analyzing $k$-chords of $H$ where the precise value of $k$ is not important. We call $P$ a \emph{generalized chord} of $H$ if there exists an integer $k\geq 1$ such that $P$ is a $k$-chord of $H$. We call $P$ a \emph{proper} generalized chord of $H$ if there is an integer $k\geq 1$ such that $P$ is a proper $k$-chord of $H$. (A proper generalized chord of $H$ is also called an \emph{$H$-path}). We define \emph{improper} generalized chords of $H$ analogously. For any $A, B\subseteq V(G)$, an \emph{$(A,B)$-path} is a path $P=x_0\cdots x_k$ with $V(P)\cap A=\{x_0\}$ and $V(P)\cap B=\{x_k\}$. Given a surface $\Sigma$, an embedding $G$ on $\Sigma$, a cyclic facial subgraph $C$ of $G$, and a proper generalized chord $Q$ of $C$, there is, under certain circumstances, a natural way to partition of $G$ specified by $Q$. 

\begin{defn}\label{ContractNatCQPartChordDefn} \emph{Let $\Sigma$ be a surface, $G$ be an embedding on $\Sigma$,  $C$ be a cyclic facial subgraph of $G$, and $Q$ be a generalized chord of $C$, where each cycle in $C\cup Q$ is contractible. The unique \emph{natural $(C,Q)$-partition} of $G$ is the pair $\{G_0, G_1\}$ of subgraphs of $G$ such that $G=G_0\cup G_1$ and $G_0\cap G_1=Q$, where, for each $i\in\{0,1\}$, there is a unique open connected region $U$ of $\Sigma\setminus (C\cup Q)$ such that $G_i$ consists of all the edges and vertices of $G$ in $\textnormal{Cl}(U)$.}

\emph{If the facial cycle $C$ is clear from the context then we usually just refer to $\{G_0, G_1\}$ as the \emph{natural $Q$-partition} of $G$. Note that this is consistent with Definition \ref{ContractNatCPartDefn} in the sense that, if $Q$ is not a proper generalized chord of $C$ (i.e $Q$ is a cycle) then the natural $Q$-partition of $G$ is the same as the natural $(C,Q)$-partition of $G$. If $\Sigma$ is the sphere (or plane) then the natural $(C, Q)$-partition of $G$ is always defined for any $C,Q$.}
\end{defn}

\begin{defn}
\emph{For any graph $G$, vertex set $X\subseteq V(G)$, integer $j\geq 0$, and real number $r\geq 0$, we set $D_j(X, G):=\{v\in V(G): d(v, X)=j\}$ and $B_r(X, G):=\{v\in V(G): d(v, X)\leq r\}$. Given a subgraph $H$ of $G$, we usually just write $D_j(H, G)$ to mean $D_j(V(H), G)$, and likewise, we usually write $B_r(H, G)$ to mean $B_r(V(H), G)$.}
\end{defn}

If $G$ is clear from the context, then we drop the second coordinate from the above notation to avoid clutter. We now introduce some additional notation related to list-assignments. We frequently analyze the situation where we begin with a partial $L$-coloring $\phi$ of a graph $G$, and then delete some or all of the vertices of $\textnormal{dom}(\phi)$ and remove the colors of the deleted vertices from the lists of their neighbors in $G\setminus\textnormal{dom}(\phi)$. We thus define the following.  

\begin{defn}\emph{Let $G$ be a graph, with list-assignment $L$. Let $\phi$ be a partial $L$-coloring of $G$ and $S\subseteq V(G)$. We define a list-assignment $L^S_{\phi}$ for $G\setminus (\textnormal{dom}(\phi)\setminus S)$ as follows.}
$$L^S_{\phi}(v):=\begin{cases} \{\phi(v)\}\ \textnormal{if}\ v\in\textnormal{dom}(\phi)\cap S\\ L(v)\setminus\{\phi(w): w\in N(v)\cap (\textnormal{dom}(\phi)\setminus S)\}\ \textnormal{if}\ v\in V(G)\setminus \textnormal{dom}(\phi) \end{cases}$$ \end{defn}

If $S=\varnothing$, then $L^{\varnothing}_{\phi}$ is a list-assignment for $G\setminus\textnormal{dom}(\phi)$ in which the colors of the vertices in $\textnormal{dom}(\phi)$ have been deleted from the lists of their neighbors in $G\setminus\textnormal{dom}(\phi)$. The situation where $S=\varnothing$ arises so frequently that, in this case, we simply drop the superscript and let $L_{\phi}$ denote the list-assignment $L^{\varnothing}_{\phi}$ for $G\setminus\textnormal{dom}(\phi)$. In some cases, we specify a subgraph $H$ of $G$ rather than a vertex-set $S$. In this case, to avoid clutter, we write $L^H_{\phi}$ to mean $L^{V(H)}_{\phi}$. 

We frequently deal with situations where we have a set $S$ of vertices that we want to delete, and it is desirable to color as few of them as possible in such a way that we can safely delete the remaining vertices of $S$ without coloring them. We thus introduce the following definition.

\begin{defn}\emph{Let $G$ be a graph with a list-assignment $L$. Given an $S\subseteq V(G)$ and a partial $L$-coloring $\phi$ of $V(G)$, we say that $S$ is \emph{$(L, \phi)$-inert in $G$} if, for every extension of $\phi$ to a partial $L$-coloring $\psi$ of $G\setminus (S\setminus\textnormal{dom}(\phi))$ whose domain contains $D_1(S\setminus\textnormal{dom}(\phi))$, we get that $\psi$ extends to an $L$-coloring of $\textnormal{dom}(\psi)\cup V(S)$.}\end{defn}

We repeatedly implicitly rely on the following observation. 

\begin{obs} Let $G$ be a graph with a list-assignment $L$. Let $S, S'\subseteq V(G)$ and $\phi, \phi'$ be partial $L$-colorings of $G$, where $\phi\cup\phi'$ is well-defined and the sets $S\setminus\textnormal{dom}(\phi)$ and $S'\setminus\textnormal{dom}(\phi')$ are of distance at least two apart. If $S$ is $(L, \phi')$-inert in $G$ and $S'$ is $(L, \phi')$-inert in $G$, then $S\cup S'$ is $(L, \phi\cup\phi')$-inert in $G$. \end{obs}

\section{Two Black Boxes From \cite{JNevinThesisManyFacePartIPaperRevised}}\label{BlackBoxPIISec}

To prove the main result of this paper, we rely on two results from \cite{JNevinThesisManyFacePartIPaperRevised}, Paper I of this sequence, as black boxes. To state these results, we introduce the following definitions. 

\begin{defn}\label{UniqueKLSpecifiedDefn} \emph{Let $G$ be an embedding on a surface $\Sigma$. Let $C$ be a contractible facial cycle of $G$, and $L$ is a list-assignment for $V(G)$. Given an integer $k\geq 2$, we say that $C$ is \emph{uniquely $k$-determined in $G$} (with respect to $L$) if, for any generalized chord $Q$ of $C$ (proper or improper) of length at most $k$, each cycle in $C\cup Q$ is contractible, and, letting $G=G_0\cup G_1$ be the natural $(C,Q)$-partition of $G$, there is a (necessarily unique) $i\in\{0,1\}$ such that}
\begin{enumerate}[label=\emph{\alph*)}]
\itemsep-0.1em
\item\emph{Every vertex of $G_{1-i}\setminus C$ has an $L$-list of size at least five and the open component of $\Sigma\setminus (C\cup Q)$ whose closure contains $G_{1-i}$ is a disc}; AND
\item\emph{Either $G_i\setminus C$ contains a vertex $v$ with $|L(v)|<5$ or the open component of $\Sigma\setminus (C\cup Q)$ whose closure contains $G_{i}$ is not a disc}.
\end{enumerate}
\emph{If $\Sigma, G, L$ are all clear, then we just say that $C$ is uniquely $k$-determined. In the setting above, we denote $G_{1-i}$ by $G^{\textnormal{small}}_Q=G_{1-i}$ and $G_i$ by $G^{\textnormal{large}}_Q$. If $k\geq 3$ and $Q'$ is a generalized chord of $C$ (of arbitrary length) with the additional property that each vertex of $Q'$ is of distance at most one from $C$, then $Q'$ also satisfies all the partitioning properties above and we define $G^{\textnormal{small}}_{Q'}$ and $G^{\textnormal{large}}_{Q'}$ analogously in this case.} 
 \end{defn}

\begin{defn}\label{DefnSmallSideLargeSideUniqueKLD} \emph{Let $k\geq 1$ be an integer and $\Sigma, G, C, L$ be as in Definition \ref{UniqueKLSpecifiedDefn}, where $C$ is uniquely $k$-determined. We define $\textnormal{Sh}^k_L(C, G)$ to be the union of all sets of the form $V(G^{\textnormal{small}}_Q\setminus Q)$, where $Q$ runs over all generalized chords $Q$ of $C$ with $|E(Q)|\leq k$. If $G, L$ are clear from the context, then we just write $\textnormal{Sh}^k(C)$.}
 \end{defn}

Our two results from \cite{JNevinThesisManyFacePartIPaperRevised} that we rely on as black boxes are Theorems \ref{FaceConnectionMainResult} and \ref{SingleFaceConnRes} below. 

\begin{theorem}\label{FaceConnectionMainResult} There exists a constant $d\geq 0$ such that the following holds: Let $G$ be a short-inseparable embedding on a surface $\Sigma$ with $\textnormal{fw}(G)\geq 6$, and let $\mathcal{C}$ be a family of facial cycles of $G$, where the elements of $\mathcal{C}$ are pairwise of distance $\geq d$ apart and every facial subgraph of $G$ not lying in $\mathcal{C}$ is a triangle. Let $L$ be a list-assignment for $V(G)$, where each vertex of $V(G)\setminus\left(\bigcup_{C\in\mathcal{C}}V(C)\right)$ has a list of size at least five. Let $\mathcal{D}\subseteq\mathcal{C}$, where each element of $\mathcal{D}$ is uniquely $4$-determined and each vertex of $\bigcup_{C\in\mathcal{D}}V(C)$ has a list of size three. Let $F\in\mathcal{D}$ and $\{P_C: C\in\mathcal{D}\setminus\{F\}\}$ be a family of paths, pairwise of distance $\geq d$ apart, where each $P_C$ is a shortest $(C,F)$-path. Then there exist an $A\subseteq V(G)$ and partial $L$-coloring $\phi$ of $A$ such that: 
\begin{enumerate}[label=\arabic*)]
\item $A$ is $(L, \phi)$-inert in $G$, and each vertex of $D_1(A)$ has an $L_{\phi}$-list of size at least three; AND
\item $G[A]$ is connected and has nonempty intersection with each element of $\mathcal{D}$; AND
\item $A\subseteq\left(B_2(F)\cup\textnormal{Sh}^4(F)\right)\cup\bigcup_{C\in\mathcal{D}\setminus\{F\}}\left(\textnormal{Sh}^4(C)\cup B_2(C\cup P_C)\right)$.
\end{enumerate}
In particular, any $d\geq 34$ suffices. 
 \end{theorem}

Informally, Theorem \ref{FaceConnectionMainResult} states that, starting with many pairwise far-apart faces with 3-lists, we can, under the specified conditions, color and delete a subgraph of $G$ which is far from all the elements of $\mathcal{C}\setminus\mathcal{D}$ and connects the elements of $\mathcal{D}$ to a single face of $G\setminus A$ (where $G\setminus A$ regarded as an embedding on $\Sigma$) where all the vertices of this new lone face all have 3-lists. We specify that the elements of $\mathcal{D}$ have vertices with lists of size precisely three, rather than at least three, so that, for each $C\in\mathcal{C}$, the requirement that $C$ is uniquely 4-determined enforces the desired property that, for any generalized chord $Q$ of $C$ of length at most four, all the elements of $\mathcal{D}\setminus\{C\}$ lie in the same component of $\Sigma\setminus (C\cup Q)$. Our second main result, Theorem \ref{SingleFaceConnRes},  is a ``single face" version of Theorem \ref{FaceConnectionMainResult}. 

\begin{theorem}\label{SingleFaceConnRes} Let $\Sigma, G, \mathcal{C}, L, d$ be as in Theorem \ref{FaceConnectionMainResult}. Let $F\in\mathcal{C}$ (possibly $\mathcal{C}=\{F\}$), where each vertex of $F$ has a list of size at least three and $F$ uniquely 4-determined. Let $P:=v_0\cdots v_n$ be a path of length $\geq d$ with both endpoints in $F$, where $P$ is a shortest path between its endpoints and $V(P)\cap D_k(F)=\{v_k, v_{n-k}\}$ for each $0\leq k\leq 3$.  Suppose further that there is a noncontractible closed curve of $\Sigma$ contained in $F\cup P$. Then there exist an $A\subseteq V(G)$ and a partial $L$-coloring $\phi$ of $A$ such that 
\begin{enumerate}[label=\arabic*)]
\itemsep-0.1em
\item $A$ is $(L, \phi)$-inert in $G$ and each vertex of $D_1(A)$ has an $L_{\phi}$-list of size at least three.
\item $G[A]$ is connected and $V(F)\cup V(v_3Pv_{n-3})\subseteq A\subseteq B_2(F\cup P)\cup\textnormal{Sh}^4(F)$.
\end{enumerate}
 \end{theorem}

\section{The Structure and Main Result of this Paper}\label{OverviewMainPfSection}

In this paper, we primarily deal with the following structures. 

\begin{defn}\label{MainChartDefn}\emph{Let $k, \alpha\geq 1$ be integers. A tuple $\mathcal{T}=(\Sigma, G, \mathcal{C}, L, C_*)$ is called an $(\alpha, k)$-\emph{chart} if $\Sigma$ is a surface, $G$ is an embedding on $\Sigma$ with list-assignment $L$, and $\mathcal{C}$ is a family of facial subgraphs of $G$ such that}

\begin{enumerate}[label=\emph{\arabic*)}]
\itemsep-0.1em 
\item \emph{$C_*\in\mathcal{C}$ and, for any distinct $H_1, H_2\in\mathcal{C}$, we have $d(H_1, H_2)\geq\alpha$}; AND
\item \emph{$|L(v)|\geq 5$ for all $v\in V(G)\setminus \left(\bigcup_{H\in\mathcal{C}}V(H)\right)$}; AND
\item \emph{For each $H\in\mathcal{C}$, there exists a connected subgraph $\mathbf{P}_{\mathcal{T},H}$ of $H$, called the \emph{precolored subgraph} of $H$, where}
\begin{enumerate}[label=\emph{\roman*)}]
\itemsep-0.1em 
\item \emph{$|E(\mathbf{P}_{\mathcal{T}, H})|\leq k$ and $V(\mathbf{P}_{\mathcal{T},H})$ is $L$-colorable}; AND
\item \emph{$|L(v)|\geq 3$ for all $v\in V(H)\setminus V(\mathbf{P}_{\mathcal{T},H})$.}
\end{enumerate}
\end{enumerate}

\end{defn}

We now introduce some more natural terminology. 

\begin{defn}\label{ChartMoreTerms}
\emph{A tuple $\mathcal{T}$ is called a \emph{chart} if there exist integers $k,\alpha\geq 1$ such that $\mathcal{T}$ is an $(\alpha, k)$-chart. A chart $\mathcal{T}=(\Sigma, G, \mathcal{C}, L, C_*)$ is called \emph{colorable} if $G$ is $L$-colorable. We call $G$ the \emph{underlying graph} of the chart and $\Sigma$ the \emph{underlying surface} of the chart. The elements of $\mathcal{C}$ are called the \emph{rings} of the chart. In particular, the elements of $\mathcal{C}\setminus\{C_*\}$ are called the \emph{inner} rings of the chart and $C_*$ is called the \emph{outer} ring of the chart.} 
\end{defn}

The terms ``inner rings" and ``outer ring" are suggestive in the sense that, when proving results of planar graphs by minimal-counterexample arguments, it is a common proof technique (such as in \cite{AllPlanar5ThomPap}) to allow the outer face to have some special properties. An embedding on a compact surface does not have an outer face, but it is sometimes helpful to think of the special face $C_*$ in a similar way. The terminology in \ref{ChartMoreTerms}, together with the notation $\mathcal{C}$, is suggestive of the fact that our primary interest is the case where the rings of the chart are cyclic facial subgraphs of $G$, but in general, the definition of facial subgraphs does not require the elements of $\mathcal{C}$ to be cycles or even connected subgraphs of $G$. If the underlying chart is clear from the context, we usually drop the $\mathcal{T}$ from the notation $\mathbf{P}_{\mathcal{T}, H}$. 

\begin{defn}
\emph{Let $\mathcal{T}=(\Sigma, G, \mathcal{C}, L, C_*),$ be a chart and let $H\in\mathcal{C}$. We say $H$ is a \emph{closed} $\mathcal{T}$-ring if $V(\mathbf{P}_H)=V(H)$. Otherwise, we say $C$ is an \emph{open} $\mathcal{T}$-ring. If $\mathcal{T}$ is clear from the context then we just call $C$ a closed (resp. open) ring.}
\end{defn}

Of particular importance to us over the course of the proof of Theorem \ref{5ListHighRepFacesFarMainRes} are embeddings which do not have separating cycles of length 3 or 4, so we give them a special name. 

\begin{defn} \emph{Let $\Sigma$ be a surface and let $G$ be an embedding on $\Sigma$. A \emph{separating cycle} in $G$ is a contractible cycle $C$ in $G$ such that each of the two connected components of $\Sigma\setminus C$ has nonempty intersection with $V(G)$. We call $G$ \emph{short-inseparable} if $\textnormal{ew}(G)>4$ and $G$ does not contain any separating cycle of length 3 or 4. Likewise, given a chart $\mathcal{T}$, we call $\mathcal{T}$ a \emph{short-inseparable chart} if the underlying graph of $\mathcal{T}$ is short-inseparable.}
\end{defn}

In this paper, we complete the most difficult step in the proof of Theorem \ref{5ListHighRepFacesFarMainRes}, which is proving that Theorem \ref{5ListHighRepFacesFarMainRes} holds for a restricted class of embeddings defined below.

\begin{defn}\label{TessTriangleDefnTerm} \emph{Let $\mathcal{T}=(\Sigma, G, \mathcal{C}, L, C_*)$ be a chart. We say $\mathcal{T}$ is \emph{chord-triangulated} if, for every facial subgraph $H$ of $G$ with $H\not\in\mathcal{C}$, every induced cycle of $G[V(H)]$ is a triangle. We call $\mathcal{T}$ a \emph{tessellation} if it is short-inseparable and chord-triangulated.}
 \end{defn}

In this paper, we show that Theorem \ref{5ListHighRepFacesFarMainRes} holds for tessellations. Actually, we prove something stronger by defining a structure called a \emph{mosaic} and showing that all mosaics are colorable, where a mosaic is a special kind of a tessellation satisfying some additional properties. The remainder of this paper consists entirely of the proof of this result, which is Theorem \ref{AllMosaicsColIntermRes1-4}. The key to the proof of Theorem \ref{AllMosaicsColIntermRes1-4}  is to choose the right definition of mosaics, i.e the right induction hypothesis. In particular, in Sections \ref{MosaicsAndPropertiesSec}-\ref{BandOpenRingCritMosaicProperSec}, we show that a minimal counterexample to Theorem \ref{AllMosaicsColIntermRes1-4} satisfies some properties which allow us to apply Theorem \ref{FaceConnectionMainResult} from \cite{JNevinThesisManyFacePartIPaperRevised} to create a smaller counterexample. The key result we need from these sections is Theorem \ref{NOver4GoodSideCorCritMos}. We have to be careful when we produce a smaller counterexample to Theorem \ref{AllMosaicsColIntermRes1-4} by applying Theorem \ref{FaceConnectionMainResult} to a minimal counterexample, because when we delete a subset of the vertices of an embedding in a way specified in Theorem \ref{FaceConnectionMainResult}, the face-width might decrease. The resulting smaller embedding still needs to have high face-width in order to be the underlying graph of a smaller counterexample. While edge-width is monotone with respect to subgraphs, this is not true of face-width. By choosing our induction hypothesis correctly, we ensure that the subgraph obtained from a minimal counterexample in the manner described above has high-enough face-width to be a smaller counterexample. Having high edge-width is a weaker condition than having high face-width, and the discussion above suggests that it might be easier to work with edge-width than face-width in a minimal counterexample argument, but the conditions on the face-width cannot be dropped in the statement of Theorem \ref{5ListHighRepFacesFarMainRes} and replaced with lower bounds on just the edge-width, or else there is an infinite family of counterexamples (see \cite{JNevinThesisManyFacePartIPaperRevised}). One of the intermediate steps in the proof of Theorem \ref{AllMosaicsColIntermRes1-4} is to show that, given a minimal counterexample to our induction hypothesis, the underlying embedding has sufficiently high face-width that the embedding obtained by the deletion described above is still the underlying embedding of a mosaic. We prove this in Section \ref{CritMosFaceWidthSec}, by using Theorem \ref{SingleFaceConnRes} as a black box and using some intermediate results that we prove in Sections \ref{RainbowsandConsistPathSec}-\ref{RedForOpRings}. Finally, in Section \ref{CompleteProofMosaic}, we combine Theorem \ref{FaceConnectionMainResult} with the work of the preceding sections to prove Theorem \ref{AllMosaicsColIntermRes1-4}. 

\section{Mosaics and Their Properties}\label{MosaicsAndPropertiesSec}

We begin by fixing constants $N_{\textnormal{mo}}, \beta$, where $N_{\textnormal{mo}}\geq 200$ and $\beta:=100N_{\textnormal{mo}}^2$. It is convenient to also require that $N_{\textnormal{mo}}$ is a multiple of 3. The subscript of the notation $N_{\textnormal{mo}}$ refers to mosaics, which is the term we use for our special charts. In particular, we define a mosaic to be a tessellation which satisfies some additional properties satisfied below, where a mosaic is allowed to contain some precolored faces (i.e closed rings) of length at most $N_{\textnormal{mo}}$. In order to state our stronger induction hypothesis, we begin with the following definitions.

\begin{defn}\label{SigmaOrLPredictableDefn}
\emph{Let $G$ be a graph and $H$ be a subgraph of $G$. We say that $H$ is \emph{shortcut-free} (in $G$) if, for any $x,y\in V(H)$, $d_G(x,y)=d_H(x,y)$. We say that $H$ is \emph{semi-shortcut-free} (in $G$) if it satisfies the weaker condition that, for any $x, y\in V(H)$ with $d_G(x,y)\leq 2$, we have $d_G(x,y)=d_H(x,y)$. Furthermore, given an $S\subseteq V(H)$,}
\begin{enumerate}[label=\emph{\arabic*)}]
\itemsep-0.1em
\item\emph{we say that $H$ is \emph{$(L, S)$-predictable} if there is a $v\in D_1(S)\setminus V(H)$ such that, for any proper $L$-coloring $\phi$ of $V(S)$, every vertex of $D_1(S)\setminus (V(H)\cup\{v\})$ has an $L_{\phi}$-list of size at least three, and $|L_{\phi}(v)|\geq 2$.}
\item\emph{We say that $H$ is \emph{highly $(L, S)$-predictable} if, for any proper $L$-coloring $\phi$ of $V(S)$, every vertex of $D_1(S)\setminus V(H)$ has an $L_{\phi}$-list of size at least three.}
\end{enumerate}
\end{defn}

In order to state the distance conditions we impose on our tessellations, we introduce the following notation.

\begin{defn}\label{RkAndWNotation}

\emph{Let $\mathcal{T}=(G, \mathcal{C}, L, C_*)$ be a chart. We define a rank function $\textnormal{Rk}(\mathcal{T}| \cdot):\mathcal{C}\rightarrow\mathbb{R}$, and for each $H\in\mathcal{C}$, a subset $\mathpzc{w}_{\mathcal{T}}(H)$ of $V(C)$ as follows.}
\begin{center}
\begin{tabular}{ cc } 
$\mathpzc{w}_{\mathcal{T}}(H):=\begin{cases} V(H)\ \textnormal{if $H$ is a closed $\mathcal{T}$-ring}\\ V(H\setminus\mathbf{P}_H)\ \textnormal{if $H$ is an open $\mathcal{T}$-ring}\end{cases}$ & $\textnormal{Rk}(\mathcal{T}| H):=\begin{cases} N_{\textnormal{mo}}\cdot |E(H)|\ \textnormal{if $H$ is a closed $\mathcal{T}$-ring}\\ 2N_{\textnormal{mo}}^2\ \textnormal{if $H$ is an open $\mathcal{T}$-ring} \end{cases}$ 
\end{tabular}
\end{center}
\end{defn}
\bigskip
If the underlying chart $\mathcal{T}$ is clear from the context then we drop the symbol $\mathcal{T}$ from the notation above. We also introduce the following notation, which is a generalization of standard notation for paths.

\begin{defn} \emph{Given a graph $H$, we let $\mathring{H}$ be the graph obtained from $H$ by deleting the vertices of degree at most one. In particular, if $H$ is a path of length at least two, then $\mathring{H}$ is the path obtained by deleting the endpoints of $H$, and if $H$ is a cycle, then $\mathring{H}=H$.} \end{defn}

\begin{defn}\label{MainMosaicAxioms}

\emph{A chart $\mathcal{T}:=(\Sigma, G, \mathcal{C}, L, C_*)$ is called a \emph{mosaic} if $\mathcal{T}$ is a tessellation such that, letting $g$ be the genus of $\Sigma$, the following additional conditions hold.}
\begin{enumerate}
\itemsep-0.1em
\item [\mylabel{}{\textnormal{M0)}}] \emph{For each $H\in\mathcal{C}$, if $H$ is open, then $|E(\mathbf{P}_H)|\leq\frac{2N_{\textnormal{mo}}}{3}$, and if $H$ is closed, then $|E(\mathbf{P}_H)|\leq N_{\textnormal{mo}}$}; AND
\item [\mylabel{}{\textnormal{M1)}}] \emph{For each open ring $H\in\mathcal{C}$, $\mathbf{P}_H$ is a path, where no chord of $H$ has an endpoint in $\mathring{\mathbf{P}}_H$. Furthermore, $H$ is highly $(L, V(\mathbf{P}_H))$-predictable and $\mathbf{P}_H$ is semi-shortcut-free}; AND
\item [\mylabel{}{\textnormal{M2)}}] \emph{For each closed ring $H\in\mathcal{C}$, $H$ is  $(L, V(\mathbf{P}_H))$-predictable and $\mathbf{P}_H$ is semi-shortcut-free}; AND
\item [\mylabel{}{\textnormal{M3)}}] \emph{For each $C\in\mathcal{C}\setminus\{C_*\}$, we have $d(\mathpzc{w}(C_*), \mathpzc{w}(C))\geq 2.9\beta\cdot 6^{g-1}+\textnormal{Rk}(C)+\textnormal{Rk}(C_*)$}; AND
\item [\mylabel{}{\textnormal{M4)}}] \emph{For any distinct $C_1, C_2\in\mathcal{C}\setminus\{C_*\}$, we have $d(\mathpzc{w}(C_1), \mathpzc{w}(C_2))\geq 2\beta\cdot 6^g+\textnormal{Rk}(C_1)+\textnormal{Rk}(C_2)$;} AND
\item [\mylabel{}{\textnormal{M5)}}] \emph{$\textnormal{fw}(G)\geq 2.1\beta\cdot 6^{g-1}$ and $\textnormal{ew}(G)\geq 2.1\beta\cdot 6^{g}$.}
\end{enumerate}

\end{defn}

Note that M5) is weaker than the condition that $\textnormal{fw}(G)\geq 2.1\beta\cdot 6^g$. The remainder of this paper consists of the proof of the following result: 

\begin{theorem}\label{AllMosaicsColIntermRes1-4} All mosaics are colorable. \end{theorem}

Note that, given a chart $\mathcal{T}=(\Sigma, G, \mathcal{C}, L, C_*)$ and an $H\in\mathcal{C}$, if $\mathbf{P}_H$ is a path of length at most one, then $H$ automatically satisfies condition M1). Furthermore, if $G$ is short-inseparable and $H\in\mathcal{C}$ is a cycle of length at most four with $V(H)=V(\mathbf{P}_H)$, then $H$ automatically satisfies condition M2), unless $V(G)=V(H)$ and $G$ is a 4-cycle with a lone chord. To prove Theorem \ref{AllMosaicsColIntermRes1-4}, we begin by introducing the following termonology. 

\begin{defn}
\emph{Let $\mathcal{T}=(\Sigma, G, \mathcal{C}, L, C_*)$ be a mosaic. We say that $\mathcal{T}$ is \emph{critical} if the following hold.}

\begin{enumerate}[label=\emph{\arabic*)}]
\itemsep-0.1em 
\item\emph{$G$ is not $L$-colorable}; AND
\item\emph{For any mosaic $(\Sigma', G', \mathcal{C}', L', D)$, if either of the following conditions hold, then $G'$ is $L$-colorable:}
\begin{enumerate}[label=\emph{\alph*)}]
\item $|V(G')|<|V(G)|$; OR
\item\emph{$|V(G')|=|V(G)|$, but either $\sum_{v\in V(G')}|L'(v)|<\sum_{v\in V(G)}|L(v)|$ or $g(\Sigma')<g(\Sigma)$.} 
\end{enumerate}
\end{enumerate}
\end{defn}

\begin{defn}\emph{Given a chart $\mathcal{T}=(\Sigma, G, \mathcal{C}, L, C_*)$ and a subgraph $H$ of $G$, we let $\mathcal{C}^{\subseteq H}$ denote $\{C\in\mathcal{C}: C\subseteq H\}$.} \end{defn}

The main result of Section \ref{MosaicsAndPropertiesSec} is Proposition \ref{BasicPropertiesCricMosProp}, which we prove below. 

\begin{prop}\label{BasicPropertiesCricMosProp}
Let $\mathcal{T}=(\Sigma, G, \mathcal{C}, L, C_*)$ be a critical mosaic. Then the following hold: 
\begin{enumerate}[label=\arabic*)]
\itemsep-0.1em
\item For each $H\in\mathcal{C}$, we have $|L(v)|=3$ for each $v\in V(C\setminus \mathbf{P}_H)$ and $|L(p)|=1$ for each $p\in V(\mathbf{P}_H)$; AND
\item For each $v\in V(G)\setminus\left(\bigcup_{H\in\mathcal{C}}V(H)\right)$, $|L(v)|=5$.
\item $G$ is a 2-connected closed 2-cell embedding, and every facial subgraph of $G$ not in $\mathcal{C}$ is a triangle; AND
\item  For each open ring $H\in\mathcal{C}$, $|V(H)|\geq 5$ and $|E(\mathbf{P}_H)|\geq 1$.
\end{enumerate}
\end{prop}

\begin{proof} Both 1) and of 2) both follow directly from the minimality of $\sum_{v\in V(G)}|L(v)|$. We now prove 3). We first prove that every element of $\mathcal{C}$ is a cycle. Suppose toward a contradiction that there is an $H\in\mathcal{C}$ which is not a cycle. As $H$ is a facial subgraph of $G$, it follows that there is an $S\subseteq V(H)$ with $|S|\leq 1$ such that $G\setminus S$ has more than one connected component. 

\begin{claim}\label{ForEachCutClaim1}
Let $K$ be a connected component of $G\setminus S$ and let $L'$ be a list-assignment for $V(K+S)$ where each $x\in S$ is precolored with a color of $L(x)$, and otherwise $L'=L$. Then the following hold.  
\begin{enumerate}[label=\roman*)]
\itemsep-0.1em
\item $K+S$ is $L$-colorable; AND
\item If $V(K\cap \mathbf{P}_H)=\varnothing$ then $K+S$ is $L'$-colorable.  
\end{enumerate}
\end{claim}

\begin{claimproof} Let $H^K:=H\cap (K+S)$ and $\mathcal{C}^K:=\{H'\in\mathcal{C}\setminus\{H\}: H'\subseteq K+S\}$. We now define a facial subgraph $C_*^K$ of $K+S$ in the following way. If $C_*\subseteq H^K$, then we set $C_*^K=C_*$. Otherwise, we let $C_*=H^K$. We first prove ii). Suppose that $V(K\cap \mathbf{P}_H)=\varnothing$ and let $\mathcal{T}':=(\Sigma, K+S, \mathcal{C}^K\cup\{H^K\}, L', C_*^K)$. We claim that $\mathcal{T}'$ is a mosaic. Firstly, since $V(K\cap\mathbf{P}_H)=\varnothing$, $H$ is an open $\mathcal{T}$-ring. Since $H$ is an open $\mathcal{T}$-ring and $H^K$ has a precolored path in $\mathcal{T}'$ consisting of one vertex, it follows that $H^K$ is an open $\mathcal{T}'$-ring and that $\mathcal{T}'$ satisfies all of M0)-M4). Note that $\textnormal{fw}(K)\geq\textnormal{fw}(G)$ and $\textnormal{ew}(K)\geq\textnormal{ew}(G)$, so M5) holds for $\mathcal{T}'$ as well. Thus, $\mathcal{T}'$ is a mosaic, and since $|V(K+S)|<|V(G)|$, it follows that $K+S$ is $L'$-colorable. This proves ii). 

Now we prove i). If $V(K\cap \mathbf{P}_H)=\varnothing$ then, by ii), $K+S$ is $L'$-colorable and thus $L$-colorable, so we are done in that case, so now suppose that $V(K\cap \mathbf{P}_H)\neq\varnothing$. Let $\mathcal{T}^{\dagger}:=(\Sigma, K+S, \mathcal{C}^K\cup\{H^K\}, L, C_*^K)$.  Suppose toward a contradiction that $K+S$ is not $L$-colorable. Since $|V(K+S)|<|V(G)|$, it follows from the criticality of $\mathcal{T}$ that $\mathcal{T}^{\dagger}$ is not a mosaic. Each component of $G\setminus S$ has edge-width at least $\textnormal{ew}(G)$ and, since $|S|\leq 1$, each component of $G\setminus S$ also has face-width at least $\textnormal{fw}(G)$, so $\mathcal{T}^{\dagger}$ satisfies M5). If $H$ is an open $\mathcal{T}$-ring, then $|E(\mathbf{P}_H\cap H^K)|\leq |E(\mathbf{P}_H)|\leq\frac{2N_{\textnormal{mo}}}{3}$, and if $H$ is a closed $\mathcal{T}$-ring, then $H^K$ is a closed $\mathcal{T}^{\dagger}$-ring as well, and $|E(\mathbf{P}_H\cap H^K)|\leq |E(\mathbf{P}_H)|\leq N_{\textnormal{mo}}$. Thus, $\mathcal{T}^{\dagger}$ satisfies M0) in any case, so it violates one of M1)-M4).  Now consider the following cases.

\textbf{Case 1:} $V(H^K)\not\subseteq V(\mathbf{P}_H)$,

In this case, $H$ is an open $\mathcal{T}$-ring and $H^K$ is an open $\mathcal{T}^{\dagger}$-ring. In particular, $\mathbf{P}_H$ is a path, and $\mathbf{P}_H\cap H^K$ is a (possibly empty) path, and furthermore, we have $\mathpzc{w}_{\mathcal{T}^{\dagger}}(H^K)=V(H^K)\setminus V(\mathbf{P}_H)$. Since $\mathpzc{w}_{\mathcal{T}}(H)=V(H)\setminus V(\mathbf{P}_H)$, it follows that $\mathcal{T}^{\dagger}$ satisfies M3)-M4), where, possibly, $H^K$ is the outer ring of $\mathcal{T}^{\dagger}$. Furthermore, M1)-M2) are immediate, so we contradict the fact that one of M1)-M4) is violated. 

\textbf{Case 2:} $V(H^K)\subseteq V(\mathbf{P}_H)$

In this case, if $H^K$ is a path, then we have $V(H^K)=V(K+S)$, since $H^K$ is a facial subgraph of $K+S$. If that holds, then $K+S$ is $L$-colorable, contradicting our assumption. Thus, $H^K$ is not a path, and it follows from M1) applied to $\mathcal{T}$ that $H$ is not an open $\mathcal{T}$-ring. Thus, $H$ is a closed $\mathcal{T}$-ring, and $H^K$ is a closed $\mathcal{T}^{\dagger}$-ring. It is again immediate that $\mathcal{T}^{\dagger}$ is a mosaic. In particular, the rank of $H^K$ is less than that of $H$, so $\mathcal{T}^{\dagger}$ satisfies M3) and M4) where, possibly, $H^K$ is the outer ring of $\mathcal{T}^{\dagger}$. Furthermore, M1)-M2) are immediate, and we contradic the fact that one of M1)-M4) is violated. This completes the proof of Claim \ref{ForEachCutClaim1}. \end{claimproof}

It immediately follows from i) of Claim \ref{ForEachCutClaim1} that $S\neq\varnothing$, or else we $L$-color each component of $G\setminus S$, and the union of these $L$-colorings is an $L$-coloring of $G$, contradicting the fact that $\mathcal{T}$ is critical. Thus, we have $|S|=1$, so let $S=\{x\}$ for some $x\in V(H)$. If $x\in V(\mathbf{P}_H)$, then $x$ is already precolored by $L$, and in that case, we again $L$-color $K+x$ for each connected component $K$ of $G-x$, and the union of these $L$-colorings is an $L$-coloring of $G$, contradicting the fact that $\mathcal{T}$ is critical. Thus, we have $x\not\in V(\mathbf{P}_H)$. Since $\mathbf{P}_H$ is connected, there is a unique connected component $K$ of $G-x$ with $\mathbf{P}_H\subseteq K$. By i) of Claim \ref{ForEachCutClaim1}, $K+x$ admits an $L$-coloring $\phi$. By 2) of Claim \ref{ForEachCutClaim1} applied to each of the remaining conected components of $G-x$, $\phi$ extends to an $L$-coloring of $G$,  contradicting the fact that $\mathcal{T}$ is critical. We conclude that every element of $\mathcal{C}$ is a cycle.

\begin{claim}\label{EvFSGnotCTriang} Every facial subgraph of $G$ not lying in $\mathcal{C}$ is a triangle. \end{claim}

\begin{claimproof} Suppose there is a facial subgraph $H$ of $G$ with $H\not\in\mathcal{C}$, where $H$ is not a triangle. If no subgraph of $H$ is a cycle, then $G=H$ and $G$ is $L$-colorable, which is false, so at least one subgraph of $H$ is a cycle. If there is a cycle $H'\subseteq H$ which is not a triangle, then, since $\mathcal{T}$ is chord-triangulated and $G$ is short-inseparable, it follows that $V(H')=V(G)$ and again, $G$ is $L$-colorable, which is false. Thus, every cycle which is a subgraph of $H$ is a triangle. As $H$ is not a triangle and $G$ is short-inseparable, we again have $G=H$, so $G$ is $L$-colorable, which is false. \end{claimproof}

As every element of $\mathcal{C}$ is a cycle, it follows from Claim \ref{EvFSGnotCTriang} that $G$ is 2-connected. Now we finish the proof of 3). It suffices to prove that $G$ is a 2-cell embedding. If that holds, then, since every facial subgraph of $G$ is a contractible cycle, it follows that $G$ is a closed 2-cell-embedding. Suppose 3) does not hold. Thus, there is a facial cycle $H$ of $G$ and a component $U$ of $\Sigma\setminus G$ such that $H=\partial(U)$, but $U$ is not a disc. As $\textnormal{fw}(G)>1$, $G$ admits an embedding $G'$ on $\mathbb{S}^2$ obtained by replacing $U$ with a disc, and $(\mathbb{S}^2, G', \mathcal{C}, L, C_*)$ satisfies all of M0)-M5), as the genus has only decreased. In particular, $\textnormal{fw}(G')=\textnormal{ew}(G')=\infty$. Since $\mathcal{T}$ is critical, $G$ is $L$-colorable, which is false. This proves 3). 

We now prove 4). Suppose toward a contradiction that $|V(H)|\leq 4$. By 3), $H$ is a cycle.  Since $G$ is short-inseparable and $H$ is a cycle of length at most four, $G$ contains no chord of $H$, so $V(H)$ admits an $L$-coloring $\psi$. Let $L'$ be a list-assignment for $V(G)$ where the vertices of $H$ are precolored by $\psi$ and otherwise $L'=L$. By 3), $H$ is a cycle, and $\mathcal{T}':=(\Sigma, G, \mathcal{C}, L', C_*)$ is a tessellation in which $H$ is a closed ring. We claim that $\mathcal{T}'$ is a mosaic. M5) is immediate, since the embedding is the same as in $\mathcal{T}$. Since $H$ is a chordless cycle of length at most four and $G$ is short-inseparable, it follows that that M2) is satisfied. Furthermore, $\mathcal{T}'$ still satifies the distance conditions of Definition \ref{MainMosaicAxioms}, as the rank of $H$ has only decreased. It is immediate that the other conditions are satisfied. Thus, $\mathcal{T}'$ is a mosaic, and, by the minimality of $\sum_{v\in V(G)}|L(v)|$, it follows that $G$ is $L'$-colorable and thus $L$-colorable, contradicting the fact that $\mathcal{T}$ is critical. Now suppose toward a contradiction that $|E(\mathbf{P}_C)|<1$. Thus, $\mathbf{P}_C$ is either empty or consists of a lone vertex, so let $L'$ be a list-assignment for $V(G)$ obtained from $L$ by properly $L$-precoloring a path $P$ of length one in $H$, where this path includes the lone vertex of $\mathbf{P}_C$, if it exists. As $H$ is a cycle, we have $V(H)\neq V(P)$, and it is immediate that $\mathcal{T}':=(\Sigma, G, \mathcal{C}, L', C_*)$ is a tessellation where $H$ is an open ring with a precolored path $P$. As this path has length one, it is immediate that $\mathcal{T}'$ satisfies M1) and is thus a mosaic. It follows from the minimality of $\sum_{v\in V(G)}|L(v)|$ that $G$ is $L'$-colorable and thus $L$-colorable, a contradiction. \end{proof}

The usefulness of imposing both face-width conditions and stronger edge-width conditions in M5) of Definition \ref{MainMosaicAxioms} lies in the fact that, if an embedding has high edge-width and consists mostly of triangles, then it also has high face-width. In the remainder of this paper, we repeatedly make use of the following facts about edge-width and face-width. The following is straightforward to check. 

\begin{prop}\label{HighEw+Triangles=HighFw} Let $\Sigma$ be a surface and let $G$ be a 2-cell embedding on $\Sigma$. Then the following hold.
\begin{enumerate}[label=\emph{\arabic*)}]
\itemsep-0.1em
\item Let $F$ be a noncontractible cycle and let $\mathcal{F}$ be a family of facial subgraphs of $G$, where $F$ is contained in $\bigcup\mathcal{F}$
\begin{enumerate}[label=\alph*)]
\itemsep-0.1em
\item If every element of $\mathcal{F}$ is a triangle, then $\textnormal{ew}(G)\leq |\mathcal{F}|+2$; AND 
\item If there is a $D\in\mathcal{F}$ such that each $D'\in\mathcal{F}\setminus\{D\}$ is a triangle, then $|E(F)\setminus E(D\cap F)|\leq |\mathcal{F}|+2$; AND
\end{enumerate}
\item Let $\alpha\geq 2$ be a constant and let $\mathcal{C}$ be a family of facial subgraphs of $G$, where each facial subgraph of $G$, other than those of $\mathcal{C}$, is a triangle. Let $D\in\mathcal{C}$ be of distance at least $\alpha$ from each element of $\mathcal{C}\setminus\{D\}$. Let $F'\subseteq G$ be a cycle with $V(F'\cap D)\neq\varnothing$. Then either $F'$ contains no vertices of any element of $\mathcal{C}\setminus\{D\}$ or $F'$ cannot be contained in the union of fewer than $2(\alpha-1)$ facial walks of $G$. 
\end{enumerate}
\end{prop}

Many of the arguments in Section \ref{WebInwardContractCycleSec}-\ref{BandOpenRingCritMosaicProperSec} involve starting with a critical mosaic $(\Sigma, G, \mathcal{C}, L, C_*)$, and then constructing a smaller mosaic either from one side of a separating cycle in $G$ or one side of a generalized chord of an element of $\mathcal{C}$. When we do this, we need to ensure the resulting smaller embedding has high enough face-width to satisfy M5). This the purpose of the two facts below, which are straightforward consequences of Proposition \ref{HighEw+Triangles=HighFw}. 

\begin{fact}\label{HighEwTriangleFwF1Cycle} Let $\Sigma$ be a surface and let $G$ be a 2-cell embedding on $\Sigma$. Let $D\subseteq G$ be contractible cycle and let $\alpha\geq 2$ be a constant. Let $U, U'$ be the components of $\Sigma\setminus D$, where $U$ is an open disc. Suppose further that $D$ is of distance at least $\alpha$ from each facial subgraph of $G$ in $\textnormal{Cl}(U')$ which is distinct from $D$ and not a triangle. Then, for any 2-cell embedding $H$ on $\Sigma$ obtained from $G\cap\textnormal{Cl}(U')$ by adding some edges and vertices to $\textnormal{Cl}(U)$, we have $\textnormal{fw}(H)\geq\min\{\textnormal{fw}(G), 2(\alpha-1), \textnormal{ew}(G)-|E(D)|-1\}$. \end{fact}

Our second fact is the following. 

\begin{fact}\label{HighEwFwF2Chord}  Let $\Sigma$ be a surface and let $G$ be a 2-cell embedding on $\Sigma$. Let $C\subseteq G$ be cyclic facial subgraph of $G$ and let $Q$ be a proper generalized chord of $C$, where each of the three cycles of $C\cup Q$ is contractible. Let $\alpha\geq 2$ be a constant and let $G=G_0\cup G_1$ be the natural $Q$-partition of $G$, and, for each $j=0,1$, let $U_j$ be the unique component of $\Sigma\setminus (C\cup Q)$ such that $G_j=G\cap\textnormal{Cl}(U_j)$, and suppose that $U_1$ is an open disc. Suppose further that $C\cup Q$ is of distance at least $\alpha$ from each facial subgraph of $G$ in $\textnormal{Cl}(U_0)$ which is distinct from $C$ and not a triangle. Let $H$ be a 2-cell embedding on $\Sigma$ obtained from $G_0$ by adding some edges and vertices to $\textnormal{Cl}(U_1)$. Then $\textnormal{fw}(H)\geq\min\{\textnormal{fw}(G), 2(\alpha-1), \textnormal{ew}(G)-2|E(Q)|-2\}$. \end{fact}

\section{Webs of Contractible Cycles}\label{WebInwardContractCycleSec}

The purpose of Section \ref{ObstructingCycleSec} is to show that, in a critical mosaic, one side of a separating cycle of length at most $N_{\textnormal{mo}}$ is ``obstructed" in a sense that one of the rings of the chart has to be close enough to the separating cycle to prevent us from coloring one side and then extending the precoloring by constructing of a smaller mosaic from the other side. This is made precise in the statement of Theorem \ref{MainRes2CloseCycleBounds}. To prove this, we first show that, under certain conditions, one side of a separating cycle in a critical mosaic is colorable. In this section, we describe a general procedure for constructing a smaller mosaic from a critical mosaic that we use in Section \ref{ObstructingCycleSec}. To avoid repetition, it is helpful to introduce some additional language which is a natural extension of our terminology of inner and outer rings for charts. 

\begin{defn} \emph{A \emph{pre-chart} is a 3-tuple $\mathcal{P}=(\Sigma, G, C_*)$, where $\Sigma$ is a surface, $G$ is an embedding on $\Sigma$, and $C_*$ is a contractible facial subgraph of $G$. As with Definition \ref{ChartMoreTerms}, we call $\Sigma, G$, and $C_*$ the respective \emph{underlying surface}, \emph{underlying graph}, and \emph{outer ring} of $\mathcal{P}$.} \end{defn}

\begin{defn}\label{Def14InOutContract} \emph{Let $\mathcal{T}$ be a tuple which is either a chart or a pre-chart, where the respective underlying surface, underlying graph, and outer ring of $\mathcal{T}$ are $\Sigma, G$, and $C_*$, and furthermore, $C_*$ is a proper subgraph of $G$. Let $D\subseteq G$ be a contractible cycle, and let $U_0, U_1$ be the components of $\Sigma\setminus D$. Let $i\in\{0,1\}$ be the unique index such that both of the following hold: $C_*\subseteq\textnormal{Cl}(U_i)$ and, if $C_*=D$, then $G\cap\textnormal{Cl}(U_i)=C_*$. }
\begin{enumerate}[label=\emph{\arabic*)}]
\itemsep-0.1em
\item\emph{We let $\textnormal{Ext}_{\mathcal{T}}(D)=G\cap\textnormal{Cl}(U_i)$ and $\textnormal{Int}_{\mathcal{T}}(D)=G\cap\textnormal{Cl}(U_{1-i})$.}
\item\emph{We define $\Sigma^D_{\mathcal{T}}$ to be the surface obtained from $\Sigma$ by replacing $\textnormal{Cl}(U_{1-i})$ with a closed disc bounded by $D$.}
\item \emph{We say that $D$ is $\mathcal{T}$-\emph{outward contractible} if $U_i$ is a homeomorphic to a disc. We say that $\mathcal{T}$-\emph{inward contractible} if $U_{1-i}$ is homeomorphic to a disc.}
\item\emph{We say that $U_i$ and $\textnormal{Cl}(U_i)$ are $\mathcal{T}$-\emph{internally bounded} by $D$. Likewise, we say that $U_{1-i}$ and $\textnormal{Cl}(U_{1-i})$ are $\mathcal{T}$-\emph{externally bounded} by $D$.}
\end{enumerate}
 \end{defn}

When using the notation of Definition \ref{Def14InOutContract}, if the underlying chart or pre-chart $\mathcal{T}$ is clear from the context, then we drop the subscript $\mathcal{T}$ from 1)-2) above and the prefix $\mathcal{T}$ from 3)-4) above. 

\begin{defn}\label{InAndOutWebDefn}\emph{Let $\mathcal{T}$ be a tuple which is either a chart or a pre-chart, where the respective underlying surface, underlying graph, and outer ring of $\mathcal{T}$ are $\Sigma, G$, and $C_*$ , and furthermore, $C_*$ is a proper subgraph of $G$. Let $D^0\subseteq G$ be an inward contractible cycle and let $U$ be the open disc externally bounded by $D^0$. Let $n:=\left\lceil\frac{|E(D^0)|}{4}\right\rceil$. A $\mathcal{T}$-\emph{web} of $D^0$ is an embedding on $\Sigma$ obtained from $D^0$ by adding to $U$ a sequence of $n$ concentric cyles $D^1, \ldots, D^n$, an additional vertex $x$, where $x$ is adjacent to all the vertices of $D^n$, and, for each $i=1,\ldots, n$, adding some edges between $D^i$ and $D^{i-1}$, such that the following hold.}
\begin{enumerate}[label=\emph{\arabic*)}]
\item \emph{$D^1$ is a cycle $v_1\ldots v_{2n}$ of length $2n$, where, for each $v_i$ of odd index, the neighborhood of $v_i$ on $D^0$ is a lone vertex, and, for each $v_i$ of even index, the neighborhood of $v_i$ on $D^0$ is an edge of $D^0$, and furthermore, the $D^0$-neighborhoods of any two consecutive $D^1$-vertices have nonempty intersection}; AND
\item\emph{$D^2$ is a cycle of length $n$ lying in the open disc externally bounded by $D^1$, where, for each $v\in V(D^2)$, the neighborhood of $v$ on $D^1$ is a path of length two whose midpoint has even index in $v_1\ldots v_{2n}$. Furthermore, the $D^1$-neighborhoods of consecutive $D^2$-vertices intersect on precisely a common endpoint of the two paths}; AND
\item \emph{For each $i=3, \ldots, n$, both of the following hold.} 
\begin{enumerate}[label=\emph{\alph*)}]
\itemsep-0.1em
\item\emph{$D^i$ separates $x$ from $D^{i-1}$ and, for each $w\in V(D^i)$, the neighborhood of $w$ in $D^{i-1}$ is an edge of $D^{i-1}$.}
\item\emph{For any distinct $w, w'\in V(D^i)$, the neighborhoods of $w, w'$ in $D^{i-1}$ are distinct edges, and these edges share an endpoint of $D^{i-1}$ if and only if $ww'$ is an edge of $D^i$.}
\end{enumerate}
\end{enumerate}
\end{defn}

\begin{prop}\label{InAndOutWebFactsProp} Let $\mathcal{T}$ be either a chart or a pre-chart, where the respective underlying surface, underlying graph, and outer ring of $\mathcal{T}$ are $\Sigma, G$, and $C_*$, and furthermore, $G$ is short-inseparable and $C_*$ is a proper subgraph of $G$. Let $D\subseteq G$ be an inward contractible cycle of length $n\geq 5$. Let $\mathbf{D}$ be the set of induced cycles of $G[V(D)]\cap\textnormal{Int}(D)$ and let $G'$ be an embedding on $\Sigma$ obtained from $G[V(\textnormal{Ext}(D))]$, where, for each $F\in\mathbf{D}$ of length at least five, we add a $\mathcal{T}$-web of $F$ to the open disc externally bounded by $F$. Then the following hold.
\begin{enumerate}[label=\arabic*)]
\itemsep-0.1em
\item $G'$ is short-inseparable; AND
\item For any $x,y\in V(G)$, we have $d_G(x,y)\leq d_{G'}(x,y)$; AND
\item $|V(G'\setminus G)|+|V(D)|\leq 9n^2$. 
\end{enumerate}
\end{prop}

\begin{proof} Let $\mathbf{D}_{\geq 5}$ be the set of $F\in\mathbf{D}$ of length at least five. For each $F\in\mathbf{D}_{\geq 5}$, since $\left\lceil\frac{|E(F)|}{4}\right\rceil\geq 2$, we are adding at least two cycles to the open disc externally bounded by $F$ before we add a lone neighbor to all the vertices of the innermost cycle, and the innermost cycle has length $|E(F)|$. Possibly $G$ has some chords of $F$ in $\textnormal{Ext}(F)$, or there is a possibly a vertex of $V(\textnormal{Ext}(F))\setminus V(F)$ with two nonadjacent neighbors in $D$, but in any case, since $G$ is short-inseparable, $G'$ is also short-inseparable. Furthermore, for each $F\in\mathbf{D}_{\geq 5}$, since we have added $\left\lceil\frac{|E(F)|}{4}\right\rceil$ concentric cycles to the open disc bounded by $F$, where the first one has length $2|E(F)|$ and the others all have length $|E(F)|$, it is clear from our construction that, for any $x,y\in V(F)$, there is no $(x,y)$-path in $G'$ of length strictly less than $d_G(x,y)$, so 2) follows. Now we prove 3). Let $\mathcal{P}$ be be the pre-chart $(\Sigma, G', C_*)$. We first note the following.

\begin{claim}\label{WebSideBoundClaim} For each $F\in\mathbf{D}$, we have $|V(\textnormal{Int}_{\mathcal{P}}(F))|\leq |E(F)|^2$. \end{claim}

\begin{claimproof} This is immediate if $|E(F)|\leq 4$ so let $|E(F)|\geq 5$. By Definition \ref{InAndOutWebDefn}, $|V(\textnormal{Int}_{\mathcal{P}}(F))|\leq \frac{|E(F)|^2}{4}+3|E(F)|+1$, where the $3|E(F)|$-term is the contribution from the edges of the outermost cycle in the open disc externally bounded by $F$ and the edges of $F$ itself. Since $|E(F)|\geq 5$, we have $3|E(F)|+1\leq\frac{3|E(F)|^2}{4}$, so $|V(\textnormal{Int}_{\mathcal{P}}(F))|\leq |E(F)|^2$.  \end{claimproof}

Note that the subgraph of $G'$ consisting of $D$ and all of the chords of $D$ in the open disc externally bounded by $D$ can be regarded as an outerplanar embedding, since $D$ is inward contractible. In particular, we have $|\mathbf{D}|\leq n-4$, and furthermore, $\sum_{F\in\mathbf{D}}|E(F)|=n+2(|\mathbf{D}|-1)\leq 3n$, since, in the sum on the left, each chord of $D$ in the open disc bounded by $D$ is counted twice and each edge of $D$ is counted once. Applying Claim \ref{WebSideBoundClaim}, we have the following. 
$$|V(\textnormal{Int}_{\mathcal{P}}(D))|\leq\sum_{F\in\mathbf{D}}|V(\textnormal{Int}_{\mathcal{P}}(F))|\leq\sum_{F\in\mathbf{D}}|E(F)|^2\leq\left(\sum_{F\in\mathbf{D}}|E(F)|\right)^2$$
Thus, we get $|V(\textnormal{Int}_{\mathcal{P}}(D))|\leq 9n^2$, so $|V(G'\setminus G)|+|V(D)|\leq 9n^2$. \end{proof}

\section{Obstructing Cycles}\label{ObstructingCycleSec}

\begin{prop}\label{CritMosaicExtSideColor} Let $\mathcal{T}=(\Sigma, G, \mathcal{C}, L, C_*)$ be a critical mosaic, and let $D\subseteq G$ be a cycle, where $|V(D)|\leq N_{\textnormal{mo}}$ and at least one $C\in\mathcal{C}$ lies in $\textnormal{Int}(D)$. Suppose further that each such $C$ is of distance at least $10N_{\textnormal{mo}}^2$ from $D$. Then $V(\textnormal{Ext}(D))$ is $L$-colorable. \end{prop}

\begin{proof} Note that, since $\textnormal{ew}(G)\geq\textnormal{fw}(G)>N_{\textnormal{mo}}$, $D$ is contractible, so the above is well-defined. It suffices to show that $G[V(\textnormal{Ext}(D))]$ is $L$-colorable. Let $\mathbf{D}$ be the set of the induced cycles of $G[V(D)]\cap\textnormal{Int}(D)$. Possibly $\mathbf{D}=\{D\}$. It follows from M5) applied to $\mathcal{T}$ that there is an $F\in\mathbf{D}$ such that each cycle of $\mathbf{D}\setminus\{F\}$ is inward contractible (possibly $F$ is inward contractible as well). Now, $G[V(\textnormal{Ext}(D))]$ can be regarded as an embedding $H$ on the surface $\Sigma^F$ in the natural way, so we let $\mathcal{P}$ denote the pre-chart $(\Sigma^F, H, D)$. Let $G'$ be an embedding on $\Sigma^F$ obtained from $H$, where, for each $D'\in\mathbf{D}$ of length at least five, we add a $\mathcal{P}$-web of $D'$ to the open disc which is $\mathcal{P}$-externally bounded by $D'$. Note that this is well-defined, since it follows from our assumption on $D$ that that $D$ is a proper subgraph of $H$. 

\begin{claim}\label{GPrimeSmallerClaim} $|V(G')|<|V(G)|$. \end{claim}

\begin{claimproof} Note that $D$ is $\mathcal{P}$-inward contractible, so it follows from 3) of Proposition \ref{InAndOutWebFactsProp} that $|V(\textnormal{Int}_{\mathcal{P}}(D))|\leq 9|E(D)|^2$. By assumption, there is a $C\in\mathcal{C}$ with $C\subseteq\textnormal{Int}_{\mathcal{T}}(D)$, and $d(C, D)\geq 10N_{\textnormal{mo}}^2$, so $|V(\textnormal{Int}_{\mathcal{T}}(D))|\geq 10N_{\textnormal{mo}}^2$. Since $D$ has length at most $N_{\textnormal{mo}}$, we have $|V(\textnormal{Int}_{\mathcal{P}}(D))|<|V(\textnormal{Int}_{\mathcal{T}}(D))|$, and thus $|V(G')|<|V(G)|$. \end{claimproof}

Let $\mathcal{C}':=\{C\in\mathcal{C}: C\subseteq\textnormal{Ext}_{\mathcal{T}}(D)\}$. Let $L'$ be a list-assignment for $V(G')$ where each vertex of $G'\setminus G$ is given an arbitrary 5-list and otherwise $L'=L$. By assumption, each $C\in\mathcal{C}$ lying in the closed region externally bounded by $D$ is of distance at least $10N_{\textnormal{mo}}^2$ from $D$, i.e there is no $C\in\mathcal{C}\setminus\mathcal{C}'$ which has nonempty intersection with $D$, so every vertex of $G'$ with an $L'$-list of size less than five lies in $\bigcup_{C\in\mathcal{C}'}V(C)$. Let $\mathcal{T}':=(\Sigma^F, G', \mathcal{C}', L', C_*)$. It follows from 1) of Proposition \ref{InAndOutWebFactsProp} that $G'$ is short-inseparable. Since $D$ is a separating cycle and $G$ is short-inseparable, there is no induced cycle of $G[V(D)]$ of length precisely four, i.e each element of $\mathbf{D}$ is either a triangle or has length at least five. In particular, it follows from our construction that every facial subgraph of $G'$, except those of $\mathcal{C}'$, is a triangle, so $\mathcal{T}'$ is a tessellation. As the genus has not increased, it also follows from 2) of Proposition \ref{InAndOutWebFactsProp} that $\mathcal{T}'$ satisfies M3), M4) of Definition \ref{MainMosaicAxioms}, and that the precolored subgraph of each ring in $\mathcal{T}'$ is still semi-shortcut-free, so M0)-M2) are satisfied well. Now we just check M5). Note that $g(\Sigma^F)\in\{0, g(\Sigma)\}$. If $g(\Sigma^F)=0$, then M5) is trivially satisfied, so we are done in that case, so now suppose that $g(\Sigma^F)=g(\Sigma)$. It follows from 2) of Proposition \ref{InAndOutWebFactsProp} that $\textnormal{ew}(G')\geq\textnormal{ew}(G)$. Note that, since $G$ is a 2-cell embedding on $\Sigma$, $G'$ is a 2-cell embedding on $\Sigma^F$. Now it follows from our distance conditions, together with Fact \ref{HighEwTriangleFwF1Cycle}, that $\textnormal{fw}(G')\geq 2.1\beta\cdot 6^{g(\Sigma)-1}$. Thus, $\mathcal{T}'$ is a mosaic. Since $\mathcal{T}$ is critical, it follows from Claim \ref{GPrimeSmallerClaim} that $G'$ is $L'$-colorable, so $G[V(\textnormal{Ext}(D))]$ is $L$-colorable, as desired. \end{proof} 

We have an analogous fact for the other side. 

\begin{prop}\label{CritMosaicIntSideColor} Let $\mathcal{T}=(\Sigma, G, \mathcal{C}, L, C_*)$ be a critical mosaic, and let $D\subseteq G$, where $|V(D)|\leq N_{\textnormal{mo}}$ and at least one $C\in\mathcal{C}$ lies in $\textnormal{Int}(D)$. Suppose further that each $C\in\mathcal{C}$ lying in $\textnormal{Ext}(D)$ is of distance at least $10N_{\textnormal{mo}}^2$ from $D$. Then $V(\textnormal{Int}(D))$ is $L$-colorable.  \end{prop}

\begin{proof} Note that $\textnormal{ew}(G)\geq\textnormal{fw}(G)>N_{\textnormal{mo}}$, so $D$ is contractible and thus the above is well-defined. Let $H=G[V(\textnormal{Int}(D))$. It suffices to prove that $H$ is $L$-colorable. By assumption, $d(C_*, D)\geq 10N^2_{\textnormal{mo}}$, and there is a $C^{\dagger}\in\mathcal{C}$ with $C^{\dagger}\subseteq\textnormal{Int}(D)$. Thus, $D$ is a separating cycle. In particular, $D$ separates $C^{\dagger}$ from $C_*$. Since $\textnormal{fw}(G)>1$, $C^{\dagger}$ is contractible. We now construct a tessellation on a new surface obtained from $\Sigma$, in which the underlying graph of this tessellation contains $H$ as a subgraph and has outer ring $C^{\dagger}$ rather than $C_*$. Let $\mathcal{P}$ be the pre-chart $(\Sigma, H, C^{\dagger})$. Let $\mathbf{D}$ be the set of the induced cycles of $G[V(D)]\cap\textnormal{Int}_{\mathcal{P}}(D)$. Note that the chords of $D$ in $\textnormal{Int}_{\mathcal{P}}(D)$ are precisely the chords of $D$ in $\textnormal{Ext}_{\mathcal{T}}(D)$). Since $\textnormal{ew}(H)=\textnormal{ew}(G)>N_{\textnormal{mo}}$, there is an $F\in\mathbf{D}$ such that every element of $\mathbf{D}\setminus\{F\}$ is $\mathcal{P}$-inward contractible. Possibly $F$ is also $\mathcal{P}$-inward contractible. Now, $H$ can be regarded as an embedding on the surface $\Sigma_{\mathcal{P}}^F$, so $\mathcal{P}'=(\Sigma_{\mathcal{P}}^F, H, C^{\dagger})$ is a pre-chart. Each element of $\mathbf{D}$ is $\mathbf{P}"$-inward contractible. As in Proposition \ref{CritMosaicExtSideColor}, we augment $H$ to an embedding $H'$ on $\Sigma^F_{\mathcal{P}}$ in the following way. For each $D'\in\mathbf{D}$ of length at least five, we add a $D'$-web to the open region of $\Sigma^F_{\mathcal{P}}$ externally $\mathcal{P}'$-bounded by $D'$.

\begin{claim}\label{KSmallerThanGClaim} $|V(H')|<|V(G)|$. \end{claim}

\begin{claimproof} It just suffices to show that $|V(\textnormal{Int}_{\mathcal{P}'}(D))|<|V(\textnormal{Ext}_{\mathcal{T}}(D))|$. Since $D$ is $\mathcal{P}'$-inward contractible, it follows from 3) of Proposition \ref{InAndOutWebFactsProp} that $|V(\textnormal{Int}_{\mathcal{P}'}(D))|\leq 9|E(D)|^2$. By assumption, we have $d_G(C_*, D)\geq 10N_{\textnormal{mo}}$, so $|V(\textnormal{Ext}_{\mathcal{T}}(D))|\geq 10N_{\textnormal{mo}}^2$ and thus $|V(\textnormal{Int}_{\mathcal{P}'}(D))|<|V(\textnormal{Ext}_{\mathcal{T}}(D))|$, as desired. \end{claimproof}

Let $\mathcal{C}':=\{C\in\mathcal{C}: C\subseteq K\}$. Let $L'$ be a list-assignment for $V(H')$ where each vertex of $H'\setminus G$ is given an arbitrary 5-list and otherwise $L'=L$. By assumption, each $C\in\mathcal{C}$ lying in $\textnormal{Int}_{\mathcal{T}}(D)$ is of distance at least $10N_{\textnormal{mo}}^2$ from $D$, so every vertex of $H'$ with an $L'$-list of size less than five lies in $\bigcup_{C\in\mathcal{C}'}V(C)$. Let $\mathcal{T}':=(\Sigma^F_{\mathcal{P}}, H', \mathcal{C}', L', C^{\dagger})$. By 1) of Proposition \ref{InAndOutWebFactsProp}, $H'$ is short-inseparable. As $G$ is short-inseparable and $D$ is a separating cycle of $G$, there is no cycle of $\mathbf{D}$ of length precisely four, i.e each element of $\mathbf{D}$ is either a triangle or has length at least five. In particular, every facial subgraph of $H'$, except those of $\mathcal{C}'$, is a triangle, so $\mathcal{T}'$ is a tessellation. Since $C^{\dagger}$ is an inner ring of $\mathcal{C}$ but the outer ring of $\mathcal{T}'$, and the genus has not increased, the distance conditions that $C^{\dagger}$ needs to satisfy have only weakened, so it also follows from 2) of Proposition \ref{InAndOutWebFactsProp} that $\mathcal{T}'$ satisfies M1)-M4) of Definition \ref{MainMosaicAxioms}. In particular, the precolored subgraph of each ring of $\mathcal{T}'$ is still semi-shortcut-free, and it is immediate that M0) is satisfied. We just need to check M5). We have $g(\Sigma^F_{\mathcal{P}})\in\{0, g(\Sigma)\}$, and if $g(\Sigma^F_{\mathcal{P}})=0$, then M5) is trivially satisfied, so suppose that $g(\Sigma^F_{\mathcal{P}})=g(\Sigma)$. It follows from 2) of Proposition \ref{InAndOutWebFactsProp} that $\textnormal{ew}(H')\geq\textnormal{ew}(G)$. Since $G$ is a 2-cell embedding on $\Sigma$, it follows from our construction of $K$ that $H'$ is a 2-cell embedding on $\Sigma^F_{\mathcal{P}}$. It follows from Fact \ref{HighEwTriangleFwF1Cycle}, together with our distance conditions, that $\textnormal{fw}(H')\geq 2.1\beta\cdot 6^{g(\Sigma)-1}$, so $\mathcal{T}'$ satisfies M5). Thus, $\mathcal{T}'$ is a mosaic. As $\mathcal{T}$ is critical, it follows from Claim \ref{KSmallerThanGClaim} that $H'$ is $L'$-colorable, so $G[V(\textnormal{Int}(D))]$ is $L$-colorable, as desired. \end{proof}

Now we prove the two main results of Section \ref{ObstructingCycleSec}, the first of which is stated below in Theorem \ref{FirstMainNoHandleInside}. Note that, in the statement of Theorem \ref{FirstMainNoHandleInside}, $D$ is contractible by our edge-width conditions, so $\textnormal{Int}(D)$ is well-defined. 

\begin{theorem}\label{FirstMainNoHandleInside} Let $\mathcal{T}=(\Sigma, G, \mathcal{C}, L, C_*)$ be a critical mosaic and let $D\subseteq G$ be a separating cycle of length at most $N_{\textnormal{mo}}$. Then either $D$ is inward contractible or $\textnormal{Int}(D)$ contains at least one element of $\mathcal{C}$. \end{theorem}

\begin{proof} Suppose there is a separating cycle $D\subseteq G$ of length at most $N_{\textnormal{mo}}$ for which neither of these hold, and, among all such cycles, we choose $D$ so that $|V(\textnormal{Int}(D))|$ is minimized. In particular, $g>0$. Note that $D$ is contractible by M5). Let $g:=g(\Sigma)$ and let $U$ be the component of $\Sigma\setminus D$ externally bounded by $D$. 

\begin{claim}\label{ogir45traifs} Let $D'$ be a cycle of length at most $N_{\textnormal{mo}}$ with $D'\subseteq\textnormal{Int}(D)$ and let $U'$ be the open component of $\Sigma\setminus D'$ externally bounded by $D'$. If $U'$ contains a noncontractible closed curve of $\Sigma$, then $D'$ is a separating cycle of $G$. \end{claim}

\begin{claimproof} By our edge-width conditions, $D'$ is contractible, so $U'$ is well-defined. As $G$ is a 2-cell embedding, there is a noncontractible cycle in $G\cap\textnormal{Cl}(U')$. We have $\textnormal{ew}(G)\geq 2.1\beta\cdot 6^{g}$, so $|V(\textnormal{Int}(D'))\setminus V(D'))|>0$. Since $D$ is a separating cycle of $G$ and $D'\subseteq\textnormal{Int}(D)$, we also have $|V(\textnormal{Ext}(D'))\setminus V(D')|>0$, so $D'$ is a separating cycle of $G$. \end{claimproof}

$\mathcal{C}^{\subseteq\textnormal{Int}(D)}=\varnothing$ and $D$ has been chosen to minimize $|V(\textnormal{Int}(D))|$, Claim \ref{ogir45traifs} immediately implies the following.

\begin{claim}\label{ClaimDHighPred1} $D$ is a shortcut-free subgraph of $\textnormal{Int}(D)$. Furthermore, for each $v\in\textnormal{Int}(D)$, the graph $G[N(v)\cap V(D)]$ is a path of length at most one. \end{claim} 

Applying Claim \ref{ClaimDHighPred1}, we have the following:

\begin{claim}\label{VExtLColClaim0} $V(\textnormal{Ext}(D))$ is $L$-colorable. \end{claim}

\begin{claimproof} Since $D$ is contractible and all of noncontractible closed curves of $\Sigma$ have nonempty intersection with $\textnormal{Cl}(U)$, it follows that $\textnormal{Ext}(D)$ can be regarded as an embedding on $\mathbb{S}^2$, and $\mathcal{P}=(\mathbb{S}^2, \textnormal{Ext}(D), C_*)$ is a pre-chart. Thus, there is an embedding $G^{\dagger}$ on $\mathbb{S}^2$ which is obtained from $\textnormal{Ext}(D)$ by adding a $D$-web to the open disc of $\mathbb{S}^2$ which is $\mathcal{P}$-externally bounded by $D$. Let $L^{\dagger}$ be a list-assignment for $V(G^{\dagger})$ where each vertex of $G^{\dagger}\setminus\textnormal{Ext}(D)$ is assigned an arbitrary 5-list, and otherwise $L^{\dagger}=L$. Let $\mathcal{T}^{\dagger}:=(\mathbb{S}^2, G^{\dagger}, \mathcal{C}, L^{\dagger}, C_*)$ and note that $\mathcal{T}^{\dagger}$ is a tessellation. We claim now that $\mathcal{T}^{\dagger}$ is a mosaic. M5) is trivially satisfied, since the underlying surface is $\mathbb{S}^2$, and M0) is immediate as well. Since the genus has only decreased, it follows from 2) of Proposition \ref{InAndOutWebFactsProp} that $\mathcal{T}^{\dagger}$ also satisfies M1)-M4), so $\mathcal{T}^{\dagger}$ is indeed a mosaic. By 3) of Proposition \ref{InAndOutWebFactsProp}, we have $|V(\textnormal{Int}_{{\mathcal{T}}^{\dagger}}(D)|\leq 9N_{\textnormal{mo}}^2$. Since $G$ is a 2-cell embedding, $\Sigma$ contains a noncontractible cycle of $G$ in $\textnormal{Cl}(U)$, so $|V(\textnormal{Int}_{\mathcal{T}}(D))|\geq\textnormal{ew}(G)\geq 2.1\beta\cdot 6^{g}$. Thus, we have $|V(G^{\dagger})|<|V(G)|$. Since $\mathcal{T}$ is critical, it follows that $G^{\dagger}$ is $L^{\dagger}$-colorable, so $\textnormal{Ext}(D)$ is $L$-colorable. It follows from Claim \ref{ClaimDHighPred1} there are no chords of $D$ in $\textnormal{Int}(D)$, so $V(\textnormal{Ext}(D))$ is $L$-colorable. \end{claimproof}

By Claim \ref{VExtLColClaim0}, there is an $L$-coloring $\psi$ of $V(\textnormal{Ext}(D))$. Let $\mathcal{T}':=(\Sigma, \textnormal{Int}(D), \{D\}, L^D_{\psi}, D)$. Every facial subgraph of $\textnormal{Int}(D)$, other than $D$, is a triangle, so $\mathcal{T}$ is a tessellation, where $D$ is a closed ring. We claim now that $\mathcal{T}'$ is a mosaic. M0) and M1) are trivially satisfied, since $|E(D)|\leq N_{\textnormal{mo}}$ and $D$ is a closed ring. It follows from Claim \ref{ClaimDHighPred1} that $\mathcal{T}'$ satisfies M2) as well. Since $\mathcal{T}'$ has only one ring, M3)-M4) are trivially satisfied, so we just need to check M5). It follows from Fact \ref{HighEwTriangleFwF1Cycle} that $\textnormal{fw}(\textnormal{Int}(D))\geq 2.1\beta\cdot 6^{g-1}$, and $\textnormal{ew}(\textnormal{Int}(D))\geq\textnormal{ew}(G)$, so $\mathcal{T}'$ is indeed a mosaic. Since $D$ is a separating cycle of $G$, we have $|V(\textnormal{Int}(D))|<|V(G)|$. Since $\mathcal{T}$ is critical, it follows that $V(\textnormal{Int}(D))$ is $L^D_{\psi}$-colorable, so $\psi$ extends to an $L$-coloring of $G$, contradicting the fact that $\mathcal{T}$ is not colorable. \end{proof}

In the second of the two main results of Section \ref{ObstructingCycleSec}, which is stated below in Theorem \ref{MainRes2CloseCycleBounds}, we establish some bounds on the distance between separating cycles and rings in a critical mosaic.

\begin{theorem}\label{MainRes2CloseCycleBounds}
Let $\mathcal{T}=(\Sigma, G, \mathcal{C}, L, C_*)$ be a critical mosaic and $D\subseteq G$ be a cycle of length at most $N_{\textnormal{mo}}$, where $D$ is a separating cycle of $G$ and $\mathcal{C}^{\subseteq\textnormal{Int}(D)}\neq\varnothing$. Let $g:=g(\Sigma)$ and $g':=g(\Sigma^D)$. Then both of the following hold.
\begin{enumerate}[label=\arabic*)]
\itemsep-0.1em 
\item there is a $C\in\mathcal{C}^{\subseteq\textnormal{Int}(D)}$ with $\max\{d(v, \mathpzc{w}(C)): v\in V(D)\}<2.9\beta\cdot 6^{(g-g')-1}+\textnormal{Rk}(C)+(N_{\textnormal{mo}}+\frac{1}{2})\cdot |E(D)|$.
\item For each $C\in\mathcal{C}^{\subseteq\textnormal{Int}(D)}$, we have $d(D, \mathpzc{w}(C))>2.9\beta\cdot (6^{g-1}-6^{g'-1})+\textnormal{Rk}(C)-\left(N_{\textnormal{mo}}+\frac{1}{2}\right)\cdot|E(D)|$.
\end{enumerate}

 \end{theorem}

\begin{proof} We first prove 1). Let $\mathcal{F}$ be the set of separating cycles $D$ of length at most $N_{\textnormal{mo}}$ such that $\mathcal{C}^{\subseteq\textnormal{Int}(D)}\neq\varnothing$ and $C\in\mathcal{C}$ violates 1). We show that $\mathcal{F}=\varnothing$. Suppose $\mathcal{F}\neq\varnothing$ and choose a $D\in\mathcal{F}$ which minimizes the quantity $|V(\textnormal{Int}(D))|$. Let $g':=g(\Sigma^D)$. Since $D$ is a separating cycle, there is precisely one of components of $\Sigma\setminus D$ whose closure contain $C_*$, so let $U, U_*$ be the connected component of $\Sigma\setminus D$, where $C_*\subseteq \textnormal{Cl}(U_*)$.

\begin{claim}\label{DExpectClaiminU} $D$ is shortcut-free in $\textnormal{Int}(D)$. \end{claim}

\begin{claimproof} Suppose not. Thus, there exist $x,y\in E(D)$ and a proper generalized chord $P$ of $D$ with endpoints $x,y$, where $|E(P)|<d_D(x,y)$ and $P\subseteq\textnormal{Int}(D)$. Let $D_0, D_1$ be the two cycles of $D\cup P$ which are distinct from $D$. Each of $D_0, D_1$ is of length at most $|E(D)|$ and contractible. For each $i=0,1$ and any subgraph $H$ of $\textnormal{Int}(D_i)$, we have $$\max\{d(v, H): v\in V(D_i)\}\geq\max\{d(v, H): v\in V(D)\}-\frac{|E(D_{1-i})|-|E(P)|}{2}$$ 
For each $i=0,1$, we have $g-g'\geq g-g(\Sigma^{D_i}))$. Thus, for any $i\in\{0,1\}$ and $C\in\mathcal{C}^{\subseteq\textnormal{Int}(D^i)}$, we have
$$\max\{d(v, \mathpzc{w}(C)): v\in V(D_i)\}\geq 2.9\beta\cdot 6^{\left(g-g(\Sigma^{D_i})\right)-1}+\textnormal{Rk}(C)+\left(N_{\textnormal{mo}}+\frac{1}{2}\right)\cdot |E(D)|-\frac{|E(D_{1-i})|-|E(P)|}{2}$$
There is at least one $k\in\{0,1\}$ with $\mathcal{C}^{\subseteq\textnormal{Int}(D_k)}\neq\varnothing$, so the minimality of $|V(\textnormal{Int}(D))|$ implies that there exists a $k\in\{0,1\}$ such that $$\left(N_{\textnormal{mo}}+\frac{1}{2}\right)\cdot |E(D_k)|>\left(N_{\textnormal{mo}}+\frac{1}{2}\right)\cdot |E(D)|-\frac{|E(D_{1-k})|-|E(P)|}{2}$$ 
Let $|E(D)|=n$ and $|E(P)|=p$, and for each $i=0,1$, let $|E(D_i)|=n_i$. Thus, we get $\frac{n_{1-k}-p}{2}>(N_{\textnormal{mo}}+\frac{1}{2})(n-n_k)$. On the other hand, we have $n=n_k+n_{1-k}-2p$, so we get $\frac{n_{1-k}-p}{2}>(N_{\textnormal{mo}}+\frac{1}{2})(n_{1-k}-2p)$. Yet our assumption on $P$ implies that $n_{1-k}-2p\geq 1$, so we get $\frac{n_{1-k}-p}{2}>N_{\textnormal{mo}}+\frac{1}{2}$, which is false. \end{claimproof}

\begin{claim}\label{AtMostOnevMrtwn} For each $w\in U$, the graph $G[N(w)\cap V(D)]$ is a path of length at most two. Furthermore, there is at most one $w\in U$ such that $G[N(w)\cap V(D)]$ is a path of length two.\end{claim}

\begin{claimproof} Since $G$ is short-inseparable and has no induced 4-cycles, together with the fact that $D$ is shortcut-free in $\textnormal{Int}(D)$. Now suppose there is a $w\in U$ such that $G[N(w)\cap V(D)]$ is a path $p_1p_2p_3$ of length two. Let $D^*:=(D-p_2)+p_1wp_3$. Note that $|E(D^*)|=|E(D)|$. As $G$ is short-inseparable and a 2-cell embedding, we get $g(\Sigma^{D_*})= g(\Sigma^D)$, and $\textnormal{Int}(D^*)$ contains an element of $\mathcal{C}$. By the minimality of $|V(\textnormal{Int}(D))|$, there is a $C\in\mathcal{C}^{\textnormal{Int}(D_*)}$ with $\max\{d(v, \mathpzc{w}(C)): v\in V(D_*)\}=\max\{d(v, \mathpzc{w}(C)): v\in V(D)\}-1$. Thus, $p_2$ is the unique vertex of $D$ at which $\max\{d(v, \mathpzc{w}(C)): v\in V(D)\}$ is attained, so $w$ is unique. \end{claimproof}

Now, since $D\in\mathcal{F}$ and $g-g'\geq 0$ and $\beta\geq 100N_{\textnormal{mo}}^2$, it follows from Proposition \ref{CritMosaicExtSideColor} that $V(\textnormal{Ext}(D))$ admits an $L$-coloring $\phi$. Let $\Sigma_*$ be the surface obtained from $\Sigma$ by replacing $U_*$ with an open disc. Now, $\textnormal{Int}(D)$ can be regarded as an embedding on $\Sigma_*$. Note that $g(\Sigma_*)=g-g'$. We claim now that $\mathcal{T}_*:=(\Sigma_*, \textnormal{Int}(D), \mathcal{C}^{\subseteq\textnormal{Int}(D)}\cup\{D\}, L^D_{\phi}, D)$ is a mosaic, where $D$ is a closed $\mathcal{T}_*$-ring. M0) and M1) are clearly satisfied, and it follows from Claims \ref{DExpectClaiminU} and \ref{AtMostOnevMrtwn} that M2) is satisfied as well. 

\begin{claim} $\mathcal{T}_*$ satisfies M5). \end{claim}

\begin{claimproof} This is trivial if $\Sigma_*=\mathbb{S}^2$ so suppose that $\Sigma_*\neq\mathbb{S}^2$. Since $D$ is contractible, all the noncontractible closed curves of $\Sigma$ have nonempty intersection with $\textnormal{Cl}(U)$, and thus $\Sigma_*=\Sigma$ and $g'=0$. Since edge-width is monotone, we have $\textnormal{ew}(\textnormal{Int}(D))\geq 2.1\beta\cdot 6^{g}$. Because $D$ is a counterexample to 1), each element of $\mathcal{C}^{\subseteq\textnormal{Int}(D)}$ is of distance at least $2.9\beta\cdot 6^{g-1}$ from $D$. Since $G$ is a 2-cell embedding, $\textnormal{Int}(D)$ is also a 2-cell embedding. Since $|E(D)|\leq N_{\textnormal{mo}}$ it follows from Fact \ref{HighEwTriangleFwF1Cycle} that $\textnormal{fw}(\textnormal{Int}(D))\geq 2.1\beta\cdot 6^{g-1}$. \end{claimproof}

Since $D\in\mathcal{F}$ and any two vertices of $D$ are of distance at most $\frac{|E(D)|}{2}$ apart, it follows that $\mathcal{T}_*$ satisfies the distance conditions M3)-M4) as well.  Thus, $\mathcal{T}_*$ is a mosaic. Since $D$ is a separating cycle of $G$, we have $|V(\textnormal{Int}(D))|<|V(G)|$. It follows from the criticality of $\mathcal{T}$ that $\textnormal{Int}(D)$ admits an $L^D_{\phi}$-coloring, so $\phi$ extends to an $L$-coloring of $\mathcal{T}$, a contradiction. This proves 1) of Theorem \ref{MainRes2CloseCycleBounds}.

 Now we prove 2). We proceed analogously. Suppose 2) does not hold, and let $\mathcal{F}'$ be the set of separating cycles of $D$ of length at most $N_{\textnormal{mo}}$ such that $\mathcal{C}^{\subseteq\textnormal{Int}(D)}\neq\varnothing$ and there exists a $C\in\mathcal{C}^{\subseteq\textnormal{Int}(D)}$ violating 2). We show that $\mathcal{F}'=\varnothing$. Suppose $\mathcal{F}'\neq\varnothing$ and choose a $D'\in\mathcal{F}'$ so that quantity $|V(\textnormal{Ext}(D'))|$ is minimized over the elements of $\mathcal{F}'$. Let $g':=g(\Sigma^{D'})$. Since $D'\in\mathcal{F}'$, there is a $C^{\dagger}\in\mathcal{C}$ with $C^{\dagger}\subseteq\textnormal{Int}(D')$, where
\begin{equation}\label{daggerEqnForDwCDag}\tag{Ineq1} d(D', \mathpzc{w}( C^{\dagger}))\leq 2.9\beta\cdot (6^{g-1}-6^{g'-1})+\textnormal{Rk}(C^{\dagger})-\left(N_{\textnormal{mo}}+\frac{1}{2}\right)\cdot |E(D')|\end{equation}

\begin{claim}\label{DistBoundsExtFor2ofMainThm2} 
\textcolor{white}{aaaaaaaaaaaaaaaaaaaaaaaaaa}
\begin{enumerate}[label=\roman*)]
\item $d(D', \mathpzc{w}(C_*))\geq 2.9\beta\cdot 6^{g'-1}+\textnormal{Rk}(C)+N_{\textnormal{mo}}\cdot |E(D')|$; AND
\item For each $C\in\mathcal{C}^{\subseteq\textnormal{Ext}(D')}\setminus\{C_*\}$, we have $d(D', \mathpzc{w}(C))\geq 2\beta\cdot 6^{g'}+\textnormal{Rk}(C)+N_{\textnormal{mo}}\cdot |E(D')|$. 
\end{enumerate}
 \end{claim}

\begin{claimproof} Let $C\in\mathcal{C}$. Since any two vertices of $D'$ are of distance at most $\frac{|E(D')|}{2}$ apart, we have $$d(\mathpzc{w}(C^{\dagger}, \mathpzc{w}(C))\leq d(\mathpzc{w}(C^{\dagger}), D')+d(\mathpzc{w}(C), D')+\frac{|E(D')|}{2}$$
Suppose first that $C=C_*$, but i) does not hold. Thus, $d(\mathpzc{w}(C^{\dagger}, \mathpzc{w}(C))< 2.9\beta\cdot 6^{g-1}+\textnormal{Rk}(C^{\dagger})+\textnormal{Rk}(C)$. Since $C_*\neq C^{\dagger}$, we contradict our distance conditions on $\mathcal{T}$. Now suppose that $C\neq C_*$, but ii) does not hold. Thus, we get
 $$d(\mathpzc{w}(C^{\dagger}, \mathpzc{w}(C))< 2.9\beta\cdot (6^{g-1}-6^{g'-1})+\textnormal{Rk}(C^{\dagger})+\left(2\beta\cdot 6^{g'}+\textnormal{Rk}(C)\right)$$
Note that $C\neq C^{\dagger}$. If $g=g'$, then we contradict M4). If $g'<g$, then $g'\leq g-1$ and we again contradict M4). \end{claimproof}

Now, let $U'$ be the connected component of $\Sigma\setminus D'$ which is $\mathcal{T}$-internally bounded by $D'$. Analogous to Claims \ref{DExpectClaiminU}-\ref{AtMostOnevMrtwn}, we have the following: 

\begin{claim}\label{DPredictableOnOneSide2}
$D'$ is shortcut-free in $\textnormal{Ext}(D')$.\end{claim}

\begin{claimproof} Suppose not. Thus, there exist $x,y\in E(D')$ and a proper generalized chord $P$ of $D'$ with endpoints $x,y$, where $|E(P)|<d_{D'}(x,y)$ and $P\subseteq\textnormal{Ext}(D')$. Let $D_0, D_1$ be the two cycles of $D'\cup P$ which are distinct from $D'$. Each of $D_0, D_1$ is of length at most $|E(D')|$ and contractible. There is a unique $k\in\{0,1\}$ such that $\textnormal{Int}(D')\subseteq\textnormal{Int}(D_k)$. It follows from Claim \ref{DistBoundsExtFor2ofMainThm2} that $D_k$ is also a separating cycle of $G$. Furthermore, $|E(D_k)|<|E(D)|\leq N_{\textnormal{mo}}$, and, by the minimality of $|V(\textnormal{Ext}(D))|$, we have $D_k\not\in\mathcal{F}'$. Now, $g'\geq g(\Sigma^{D_k})$, so $6^{g-1}-6^{g(\Sigma^{D_k})-1}\geq 6^{g-1}-6^{g'-1}$ and thus
$$d(D_k, \mathpzc{w}(C^{\dagger}))>2.9\beta (6^{g-1}-6^{g'-1})+\textnormal{Rk}(C^{\dagger})-\left(N_{\textnormal{mo}}+\frac{1}{2}\right)\cdot |E(D_k)|$$
Note that $d(D_k, \mathpzc{w}(C^{\dagger}))\leq d(D', \mathpzc{w}(C^{\dagger})+\frac{|E(D_{1-k})|-|E(P)|}{2}$. Thus, by $(\dagger)$, we obtain
$$\left(N_{\textnormal{mo}}+\frac{1}{2}\right)\cdot |E(D_k)|>\left(N_{\textnormal{mo}}+\frac{1}{2}\right)\cdot |E(D')|-\frac{|E(D_{1-k})|-|E(P)|}{2}$$
As in Claim \ref{DExpectClaiminU}, since $|E(D_0)|+|E(D_1)|=|E(D')|+2|E(P)|$, our assumption that $|E(D_{1-k})|-2|E(P)|\geq 1$ implies that $\frac{|E(D_{1-k})|-|E(P)|}{2}>N_{\textnormal{mo}}+\frac{1}{2}$, which is false. \end{claimproof}

\begin{claim}\label{ImprovPredfromShort}  For each $w\in U'$, the graph $G[N(w)\cap V(D')]$ is a path of length at most two. Furthermore, there is at most one $w\in U$ such that $G[N(w)\cap V(D')]$ is a path of length two. \end{claim}

\begin{claimproof} Since $G$ is short-inseparable, the first part of the claim follows from our triangulation conditions, together with the fact that $D'$ is shortcut-free in $\textnormal{Ext}(D')$. Now suppose there is a $w\in U$ such that $G[N(w)\cap V(D')]$ is a path $p_1p_2p_3$ of length two. Let $D^*:=(D'-p_2)+p_1wp_3$. As $G$ is short-inseparable and a 2-cell embedding, we get $g(\Sigma^{D_*})= g'$, and $C^{\dagger}\subseteq\textnormal{Int}(D')\subsetneq\textnormal{Int}(D^*)$. By the minimality of $|V(\textnormal{Ext}(D'))|$, we have $D^*\not\in\mathcal{F}'$, so we have
$$d(D^*, \mathpzc{w}(C^{\dagger}))>2.9\beta\cdot (6^{g-1}-6^{g'-1})+\textnormal{Rk}(C^{\dagger})-\left(N_{\textnormal{mo}}+\frac{1}{2}\right)\cdot |E(D')|$$
Thus, (\ref{daggerEqnForDwCDag}) implies that $p_2$ is the unique vertex of $D'$ at which $d(D', \mathpzc{w}(C^{\dagger}))=\min\{d(v, \mathpzc{w}(C^{\dagger}))): v\in V(D')\}$ is attained, so $w$ is indeed unique.\end{claimproof}

It follows from Claim \ref{DistBoundsExtFor2ofMainThm2} that each element of $\mathcal{C}^{\subseteq\textnormal{Ext}(D')}$ is of distance at least $10N_{\textnormal{mo}}^2$ from $D'$. By Proposition \ref{CritMosaicIntSideColor}, $V(\textnormal{Int}(D'))$ admits an $L$-coloring $\psi$. We claim now that $\mathcal{T}'':=(\Sigma^{D'}, \textnormal{Ext}(D'), \mathcal{C}^{\subseteq\textnormal{Ext}(D')}\cup\{D\}, L^{D'}_{\psi}, C_*)$ is a mosaic. Note that $\mathcal{T}''$ is a tessellation and $D'$ is a closed $\mathcal{T}''$-ring. It is clear that $\mathcal{T}''$ satisfies M0) and M1), and it follows from Claims \ref{DPredictableOnOneSide2}-\ref{ImprovPredfromShort} that $\mathcal{T}''$ satisfies M2). As all the rings of $\mathcal{C}^{\subseteq\textnormal{Ext}(D')}$ have the same rank in $\mathcal{T}$ and $\mathcal{T}''$, it follows from Claim \ref{DistBoundsExtFor2ofMainThm2} that $\mathcal{T}''$ also satisfies M3)-M4). 

\begin{claim} $\mathcal{T}''$ satisfies M5). \end{claim}

\begin{claimproof} This is trivial if $g'=0$ so suppose that $g'>0$. Thus, we have $g=g'$ and $\Sigma^D=\Sigma$. Since edge-width is monotone, we have $\textnormal{ew}(\textnormal{Ext}(D'))\geq\textnormal{ew}(G)\geq 2.1\beta\cdot 6^{g}$. Since $G$ is a 2-cell embedding $\textnormal{Ext}(D')$ is also a 2-cell embedding, and since $|E(D')|\leq N_{\textnormal{mo}}$, it follows from Fact \ref{HighEwTriangleFwF1Cycle} that $\textnormal{fw}(\textnormal{Ext}(D'))\geq 2.1\beta\cdot 6^{g-1}$ \end{claimproof}

We conclude that $\mathcal{T}''$ is a mosaic. Since $D'$ is a separating cycle of $G$, we have $|V(\textnormal{Ext}(D'))|<|V(G)|$. As $\mathcal{T}$ is critical, $\mathcal{T}''$ is colorable, so $\psi$ extends to an $L$-coloring of $G$, a contradiction. This proves Theorem \ref{MainRes2CloseCycleBounds}. \end{proof}

To conclude Section \ref{ObstructingCycleSec}, we prove the following the following corollary to Theorems \ref{FirstMainNoHandleInside} and \ref{MainRes2CloseCycleBounds}. 

\begin{cor}\label{NonRingSeparatingCor}
Let $\mathcal{T}=(\Sigma, G, \mathcal{C}, L, C_*)$ be a critical mosaic and $D\subseteq G$ be a cycle of length at most $N_{\textnormal{mo}}$, where $D$ is a separating cycle of $G$. Let $g':=g(\Sigma^D)$ and suppose there is a $C\in\mathcal{C}$ such that one of the following holds.
\begin{enumerate}[label=\alph*)]
\itemsep-0.1em
\item $C\neq C_*$ and $d(D, \mathpzc{w}(C))\leq 2.9\beta\cdot (6^{g-1}-6^{g'-1})+\textnormal{Rk}(C)-(N_{\textnormal{mo}}+\frac{1}{2})\cdot |E(D)|$; OR
\item $C=C_*$ and $d(D, \mathpzc{w}(C))\leq 2.9\beta\cdot (6^{g-1}-6^{(g-g')-1})+\textnormal{Rk}(C)-(N_{\textnormal{mo}}+\frac{1}{2})\cdot |E(D)|$
\end{enumerate}
Then every element of $\mathcal{C}$ lies in $\textnormal{Ext}(D)$, and furthermore, $g=g'$, i.e $D$ is inward contractible.\end{cor}

\begin{proof} It suffices to show that every element of $\mathcal{C}$ lies in $\textnormal{Ext}(D)$. If this holds, then Theorem \ref{FirstMainNoHandleInside} implies that $D$ is also inward contractible. Suppose toward a contradiction that not every element of $\mathcal{C}$ lies in $\textnormal{Ext}(D)$. Since each element of $\mathcal{C}$ is a facial subgraph of $G$, it follows that $\mathcal{C}^{\textnormal{Int}(D)}\neq\varnothing$. Thus, by 1) of Theorem \ref{MainRes2CloseCycleBounds}, there is a $C^{\dagger}\subseteq\textnormal{Int}(D)$ with $\max\{d(\mathpzc{w}(C^{\dagger}), v): v\in V(D)\}<2.9\beta\cdot 6^{(g-g')-1}+\textnormal{Rk}(C^{\dagger})+\left(N_{\textnormal{mo}}+\frac{1}{2}\right)\cdot |E(D)|$. Now, by assumption, there is a $C\in\mathcal{C}$ satisfying either a) or b). We have $d(\mathpzc{w}(C), \mathpzc{w}(C^{\dagger}))\leq d(\mathpzc{w}(C), D)+\max\{d(v, \mathpzc{w}(C^{\dagger}): v\in V(D)\}$. We now show that $C\neq C_*$. Suppose $C=C_*$. In this case, $C$ satisfies b) and, since $D$ is a separating cycle, we have $C^{\dagger}\neq C$. By our assumption on $C^{\dagger}$, we get $d(\mathpzc{w}(C), \mathpzc{w}(C^{\dagger}))<2.9\beta\cdot 6^{g-1}+\textnormal{Rk}(C)+\textnormal{Rk}(C^{\dagger})$. This contradicts M3) applied to $\mathcal{T}$. Thus, $C\neq C_*$, so $C$ satisfies a). Now it follows from 2) of Theorem \ref{MainRes2CloseCycleBounds} that $C\subseteq\textnormal{Ext}(D)$. Since $D$ is a separating cycle, we have $C\neq C^{\dagger}$. Thus, by our assumption on $C^{\dagger}$, we get $$d(\mathpzc{w}(C), \mathpzc{w}(C^{\dagger}))<2.9\beta\cdot (6^{g-1}+6^{(g-g')-1}-6^{g'-1})+\textnormal{Rk}(C)+\textnormal{Rk}(C^{\dagger})$$ 
This contradicts M4) applied to $\mathcal{T}$.\end{proof}

\section{Generalized Chords of Closed Rings in Critical Mosaics}\label{GenChordClosedRingCritSec}

Below, we state and prove the lone theorem which makes up Section \ref{GenChordClosedRingCritSec}. Note that, in the statement below, the natural $Q$-partition of $G$ is well-defined by our face-width conditions.

\begin{theorem}\label{ShortGenChordClosedRingThm}
Let $\mathcal{T}=(\Sigma, G, \mathcal{C}, L, C_*)$ be a critical mosaic and $C\in\mathcal{C}$ be closed. Let $Q$ be a generalized chord of $C$ with $|E(Q)|\leq\frac{N_{\textnormal{mo}}}{3}$ and $G_0\cup G_1$ be the natural $Q$-partition of $G$. Then there is a $j\in\{0,1\}$ such that each element of $\mathcal{C}\setminus\{C\}$ lies in $G_j$ and the unique component of $\Sigma\setminus (C\cup Q)$ whose closure contains $G_{1-j}$ is a disc. \end{theorem}

We break the proof of Theorem \ref{ShortGenChordClosedRingThm} into two lemmas, the first of which deals with proper generalized chords.

\begin{lemma}\label{FirstCasePropChordClosedRingLem} Let $\mathcal{T}=(\Sigma, G, \mathcal{C}, L, C_*)$ be a critical mosaic and let $C\in\mathcal{C}$ be a closed ring. Let $Q$ be a proper generalized chord of $C$ with $|E(Q)|\leq\frac{2N_{\textnormal{mo}}}{5}$ and let $G=G_0\cup G_1$ be the natural $Q$-partition of $G$. Then there exists a $j\in\{0,1\}$ such that each element of $\mathcal{C}\setminus\{C\}$ lies in $G_j$ and the unique open component of $\Sigma\setminus (C\cup Q)$ whose closure contains $G_{1-j}$ is a disc. \end{lemma}

\begin{proof}  Suppose toward a contradiction that $Q$ violates Lemma \ref{FirstCasePropChordClosedRingLem}, and let $g:=g(\Sigma)$.

\begin{claim}\label{OuterFaceCaseClaimForL1} $C\neq C_*$. \end{claim}

\begin{claimproof} Suppose $C=C_*$. Let $G=G_0\cup G_1$ be the natural $Q$-partition of $G$. For each $j=0,1$, let $U_j$ be the unique component of $\Sigma\setminus (C_*\cup Q)$ such that $G\cap\textnormal{Cl}(U_j)=G_j$, and let $C_j:=G_j\cap (C_*\cup Q)$. Since $Q$ is a proper generalized chord of $C$, each of $C_0, C_1$ is a cycle. Since $C_*$ is the outer ring of $G$, we have $G_j=\textnormal{Int}_{\mathcal{T}}(C_j)$ for each $j=0,1$. 

\vspace*{-8mm}
\begin{addmargin}[2em]{0em}
\begin{subclaim}\label{LengthBoundsC0C1Sub1} Each of $C_0, C_1$ has length at most $N_{\textnormal{mo}}$ \end{subclaim}

\begin{claimproof} Suppose not, and suppose without loss of generality that $|E(C_0)|>N_{\textnormal{mo}}$. Since $|E(C_0)|+|E(C_1)|=|E(C_*)|+2|E(Q)|\leq \frac{9N_{\textnormal{mo}}}{5}$, we have $|E(C_1)|<\frac{4N_{\textnormal{mo}}}{5}$. We note now that at least one element of $\mathcal{C}$ lies in $\textnormal{Int}(C_1)$. If $C_1$ is not inward contractible, then this just follows from Theorem \ref{FirstMainNoHandleInside}. On the other hand, if $C_1$ is not inward contractible, then, by our assumption on $Q$, at least one element of $\mathcal{C}$ lies in $\textnormal{Int}(C_1)$, so we have $\mathcal{C}^{\subseteq\textnormal{Int}(C_1)}\neq\varnothing$ in any case. Let $g':=g(\Sigma^{C_1})$. It now follows from 1) of Theorem \ref{MainRes2CloseCycleBounds} that there is $C^{\dagger}\in\mathcal{C}$ with $C^{\dagger}\subseteq\textnormal{Int}(C_1)$ and $\max\{d(v, \mathpzc{w}(C^{\dagger})): v\in V(C_1)\}<2.9\beta\cdot 6^{(g-g')-1}+\textnormal{Rk}(C^{\dagger})+\left(N_{\textnormal{mo}}+\frac{1}{2}\right)\cdot |E(C_1)|$. 

By our distance conditions on $\mathcal{T}$, we have $d(C_*, \mathpzc{w}(C^{\dagger}))\geq 2.9\beta\cdot 6^{g-1}+\textnormal{Rk}(C^{\dagger})+N_{\textnormal{mo}}\cdot |E(C_*)|$, so, since $V(C_*\cap C_1)\neq\varnothing$ and $|E(C_1)|<\frac{4N_{\textnormal{mo}}}{5}$ we get $$\left(1+\frac{1}{2N_{\textnormal{mo}}}\right)\cdot |E(C_1)|>|E(C_*)|=|E(C_1)|+\left(|E(C_0)|-2|E(Q)|\right)\geq |E(C_1)|+\frac{N_{\textnormal{mo}}}{5}\geq\frac{5|E(C_1)|}{4}$$
which is false.  \end{claimproof}\end{addmargin}

\vspace*{-8mm}
\begin{addmargin}[2em]{0em}
\begin{subclaim} $g=0$. \end{subclaim}

\begin{claimproof} Suppose that $g>0$. Since each of $C_0, C_1$ is contractible, there is an $\ell\in\{0,1\}$ such that all the noncontractible closed curves of $\Sigma$ intersect with $\textnormal{Cl}(U_{\ell})$, say $\ell=0$ without loss of generality Thus, $C_1$ is inward contractible. By our assumption on $Q$, at least one element of $\mathcal{C}$ lies in $\textnormal{Int}(C_1)$. Let $g':=g(\Sigma^{C_1})$. As $C_1$ is inward contractible, we have $g'=g$. By Subclaim \ref{LengthBoundsC0C1Sub1}, $|E(C_1)|\leq N_{\textnormal{mo}}$. Now, since $G_1=\textnormal{Int}(C_1)$ and $V(C_*\cap C_1)\neq\varnothing$. Since $g>0$, we have $g>g-g'$, so we get
$$d(C_1, \mathpzc{w}(C_*))\leq 2.9\beta\cdot (6^{g-1}-6^{(g-g')-1})+\textnormal{Rk}(C_*)-\left(N_{\textnormal{mo}}+\frac{1}{2}\right)\cdot |E(C_1)|$$
It follows from b) of Corollary \ref{NonRingSeparatingCor} that every element of $\mathcal{C}$ lies in $\textnormal{Ext}(C_1)$, a contradiction. \end{claimproof}\end{addmargin}

Since $g=0$, it follows from 1) of Theorem \ref{MainRes2CloseCycleBounds}, together with our assumption on $Q$, that, for each $j=0,1$, there is a $C_j'\in\mathcal{C}$ with $C_j'\subseteq\textnormal{Int}(C_j)$ such that $d(C_j, \mathpzc{w}(C_j'))\leq\frac{2.9\beta}{6}+4N_{\textnormal{mo}}^2$. Since each of $C_0', C_1'$ is an inner ring of $\mathcal{T}$ and any two vertices of $C_0\cup C_1$ are of distance at most $N_{\textnormal{mo}}/2$ apart, we contradict M4). This proves Claim \ref{OuterFaceCaseClaimForL1}. \end{claimproof}

Since $C\neq C_*$, we have $Q\subseteq\textnormal{Ext}(C)$ and there exists precisely one $k\in\{0,1\}$ such that $C_*\subseteq G_k$.  Thus, there exist cycles $C_{\textnormal{in}}$ and $C_{\textnormal{out}}$ such that the $C_{\textnormal{in}}:=(C\cup Q)\cap G_{1-k}$ and $C_{\textnormal{out}}:=(C\cup Q)\cap G_{k}$. In particular, $G_{1-k}=\textnormal{Int}(C_{\textnormal{in}})$ and $G_{k}=\textnormal{Ext}(C_{\textnormal{out}})$, and $C_{\textnormal{out}}$ separates $C$ from $C_*$. 

\begin{claim}\label{StrenClaimOnCOut} $|E(C_{\textnormal{out}})|\geq |E(C)|$ . \end{claim}

\begin{claimproof} Suppose not. In particular, $|E(C_{\textnormal{out}})|\leq N_{\textnormal{mo}}$. We have $V(C\cap C_{\textnormal{out}})\neq\varnothing$ Since $C_{\textnormal{out}}$ separates $C$ from $C_*$ and $\mathpzc{w}(C)=V(C)$, it follows from 2) of Theorem \ref{MainRes2CloseCycleBounds} that $$0>2.9\beta\cdot (6^{g-1}-6^{g(\Sigma^{C_{\textnormal{out}}})-1})+N_{\textnormal{mo}}\cdot |E(C)|-\left(N_{\textnormal{mo}}+\frac{1}{2}\right)\cdot |E(C_{\textnormal{out}})|$$
Thus, $g(\Sigma^{C_{\textnormal{out}}})=g$, so $C_{\textnormal{out}}$ is inward contractible and, by assumption, $|E(C_{\textnormal{out}})|\leq |E(C)|-1$ and we get 
$$N_{\textnormal{mo}}\cdot |E(C)|=\textnormal{Rk}(C)<\left(N_{\textnormal{mo}}+\frac{1}{2}\right)\cdot |E(C_{\textnormal{out}})|\leq \left(N_{\textnormal{mo}}+\frac{1}{2}\right)\cdot (|E(C)|-1)$$
which is false, since $|E(C)|\leq N_{\textnormal{mo}}$.  \end{claimproof}

Since $|E(C_{\textnormal{in}})|+|E(C_{\textnormal{out}})|=|E(C)|+2|E(Q)|$, it follows from Claim \ref{StrenClaimOnCOut} that $|E(C_{\textnormal{in}})|\leq 2|E(Q)|\leq \frac{4N_{\textnormal{mo}}}{5}$. 

\begin{claim}\label{CInRingSepCycle1} At least one element of $\mathcal{C}\setminus\{C\}$ lies in $\textnormal{Int}(C_{\textnormal{in}})$. \end{claim}

\begin{claimproof}  Suppose that no elements of $\mathcal{C}\setminus\{C\}$ lie in $\textnormal{Int}(C_{\textnormal{in}})$. Since $Q$ is unacceptable, it follows that $g>0$ and all the noncontractible closed curves of $\Sigma$ intersect with the closed region externally bounded by $C_{\textnormal{in}}$. Since $|E(C_{\textnormal{in}})|\leq N_{\textnormal{mo}}$, we contradict Theorem \ref{FirstMainNoHandleInside}. \end{claimproof}

Now, 1) of Theorem \ref{MainRes2CloseCycleBounds}, together with Claim \ref{CInRingSepCycle1}, implies that there is an $X\in\mathcal{C}$ with $X\subseteq\textnormal{Int}(C_{\textnormal{in}})$, where 
$$d(\mathpzc{w}(C), \mathpzc{w}(X))\leq \max\{d(v, \mathpzc{w}(X)): v\in V(C_{\textnormal{in}})\}<2.9\beta\cdot 6^{(g-g(\Sigma^{C_{\textnormal{in}}}))-1}+\textnormal{Rk}(X)+\left(N_{\textnormal{mo}}+\frac{1}{2}\right)\cdot |E(C_{\textnormal{in}})|$$
But since each of $C, X$ is an internal ring of $\mathcal{T}$ and $C\neq X$, this contradicts our distance conditions. This completes . This completes the proof of Lemma \ref{FirstCasePropChordClosedRingLem}. \end{proof}

We now complete the proof of Theorem \ref{ShortGenChordClosedRingThm} by dealing with improper generalized chords. For technical reasons, we actually prove something slightly stronger.

\begin{lemma}\label{ImpropGenChordClosedRing2Lem} Let $\mathcal{T}=(\Sigma, G, \mathcal{C}, L, C_*)$ be a critical mosaic and $C\in\mathcal{C}$ be closed. Let $D\subseteq G$ be a cycle with $|E(D)|\leq\frac{N_{\textnormal{mo}}}{3}$ and $d(D, C)\leq 1$. Let $G_0\cup G_1$ be the natural $D$-partition of $G$. Then there is a $j\in\{0,1\}$ such that each element of $\mathcal{C}\setminus\{C\}$ lies in $G_j$ and the unique component of $\Sigma\setminus (C\cup Q)$ whose closure contains $G_{1-j}$ is a disc. \end{lemma}

\begin{proof} Suppose toward a contradiction that Lemma \ref{ImpropGenChordClosedRing2Lem} does not hold.  Given a cycle $D$, we say that $D$ is \emph{undesirable} if $d(D, C)\leq 1$ and $|E(D)|\leq\frac{N_{\textnormal{mo}}}{3}$, but $D$ violates Lemma \ref{ImpropGenChordClosedRing2Lem}. We choose an undesirable cycle cycle $F$ which minimizes the size of the (unique) side of $F$ containing $C$. More formally, we choose $F$ so that, letting $G=G_0\cup G_1$ be the natural $F$-partition of $G$, where $C\subseteq G_0$, the quantity $|V(G_0)|$ is minimized among all undesirable cycles. Note that, if $C=C_*$ then $G_0=\textnormal{Ext}_{\mathcal{T}}(F)$. Let $g:=g(\Sigma)$ and let $U_1$ be the unique component of $\Sigma\setminus F$ with $G\cap\textnormal{Cl}(U_1)=G_1$. Let $\Sigma_1$ be the surface obtained from $\Sigma$ by replacing $U_1$ with an open disc. Now, if $\textnormal{Int}(F)$ contains no elements of $\mathcal{C}$, then, since $F$ violates Lemma \ref{ImpropGenChordClosedRing2Lem}, it follows that $F$ is not $\mathcal{T}$-inward-contractible, contradicting Theorem \ref{FirstMainNoHandleInside}. Thus, $\textnormal{Int}(F)$ contains an element of $\mathcal{C}$. Furthermore, $F\cap C$ consists of at most a lone vertex of $C$, or else there is a proper generalized chord of $C$ violating Lemma \ref{FirstCasePropChordClosedRingLem}. 

\begin{claim}\label{SigmaDaggerSurfaceFactsLem2} 
\textcolor{white}{aaaaaaaaaaaaaaaaa}
\begin{enumerate}[label=\arabic*)]
\itemsep-0.1em
\item $\Sigma_1=\mathbb{S}^2$ and $F$ separates two elements of $\mathcal{C}\setminus\{C\}$; AND
\item If $C\neq C_*$, then  $C\subseteq\textnormal{Int}_{\mathcal{T}}(F)$ and $F$ is $\mathcal{T}$-inward contractible. On the other hand, if $C=C_*$, then $C\subseteq\textnormal{Ext}_{\mathcal{T}}(F)$ and $F$ is $\mathcal{T}$-outward contractible.  
\end{enumerate}\end{claim}

\begin{claimproof} We deal with the two cases.

\textbf{Case 1:} $C\neq C_*$

In this case, we first note that $C\subseteq\textnormal{Int}(F)$, or else, since $\mathcal{C}^{\textnormal{Int}(F)}\neq\varnothing$ and $C\neq C_*$, it follows from 1) of Theorem \ref{MainRes2CloseCycleBounds} that there is an $X\in\mathcal{C}\setminus\{C\}$ such that $d(\mathpzc{w}(C), \mathpzc{w}(X))$ violates M4), which is false. Thus, $C\subseteq\textnormal{Int}(F)$. In particular, it follows from 2) of Theorem \ref{MainRes2CloseCycleBounds} that $6^{g-1}-6^{g(\Sigma^F)-1}=0$, so $g=g(\Sigma^F)$ and $F$ is inward-contractible. Since $C\subseteq\textnormal{Int}(F)$, we have $G_0=\textnormal{Int}(F)$. Since $F$ is undesirable and inward-contractible, there is at least one element of $\mathcal{C}\setminus\{C\}$ in $\textnormal{Int}(F)$, so $F$ separates an element of $\mathcal{C}\setminus\{C\}$ from $C_*$. Since $G_0=\textnormal{Int}(F)$, $U$ is the component of $\Sigma\setminus F$ which is internally bounded by $F$, so  $\Sigma_1=\mathbb{S}^2$.

\textbf{Case 2:} $C=C_*$

In this case, $G_1=\textnormal{Int}(F)$ and $G_0=\textnormal{Ext}(F)$. We firsst show that $F$ is outward contractible. Suppose not. Since $F$ is contractible, it is inward contractible. Recall that $\mathcal{C}^{\subseteq\textnormal{Int}(F)}\neq\varnothing$.  Since $g>0$ and $g=g'$, we contradict  b) of Corollary \ref{NonRingSeparatingCor}. Thus, $F$ is outward contractible. Since $G_1=\textnormal{Int}(F)$, we get $\Sigma_1=\Sigma^F=\mathbb{S}^2$. To finish, we just need to show that $F$ separates two elements of $\mathcal{C}\setminus\{C\}$. Suppose not. Since $F$ is undesirable, we get $g>0$ and $\mathcal{C}\setminus\{C\}\neq\varnothing$, and furthermore, every noncontractible closed curve of $\Sigma$ which intersects an element of $\mathcal{C}\setminus\{C\}$ also intersects with $F$. Since $g>0$ and $F$ is outward contractible, it is not inward contractible, so we contradict Theorem \ref{FirstMainNoHandleInside}. \end{claimproof}

\begin{claim}\label{ChordsAnd2ChordsForFSide} $F$ has no chords in $G_0$, and furthermore, for any $v\in V(G_0)\setminus V(C\cup F)$, either $N(v)\cap V(F)|\leq 1$ or $N(v)\cap V(F)$ consists of two consecutive vertices of $F$.  \end{claim}

\begin{claimproof} Suppose not, and, for some $k\in\{1,2\}$, let $P\subseteq G_0$ be a $k$-chord of $F$ which violates Claim \ref{ChordsAnd2ChordsForFSide}. Since $P$ violates Claim \ref{ChordsAnd2ChordsForFSide}, there is a pair of cycles $K, K'\subseteq (F\cup P)\cap G_0$ such that the following hold.
\begin{enumerate}[label=\arabic*)]
\item $d(K, C)\leq 1$ and each of $K, K'$ has length at most $|E(F)|$; \emph{AND}
\item $K\cap K'=P$ and $C\subseteq\textnormal{Ext}(K')$
\end{enumerate}

By the minimality of $|V(G_0)|$, $K$ is not an undesirable cycle. By 1) of Claim \ref{SigmaDaggerSurfaceFactsLem2}, $F$ separates two elements of $\mathcal{C}\setminus\{C\}$. Siince $K$ is not undesirable, it follows that $\textnormal{Int}(K')$ contains an element of $\mathcal{C}\setminus\{C\}$. Since $|E(K')|\leq N_{\textnormal{mo}}$, it follows from 1) of Theorem \ref{MainRes2CloseCycleBounds} that there is a $C_1\in\mathcal{C}$ with $C_1\subseteq\textnormal{Int}_{\mathcal{T}}(K')$ with $d(\mathpzc{w}(C_1), K')<2.9\beta\cdot 6^{g-1}+4N_{\textnormal{mo}}^2$.  Each vertex of $K'$ has distance at most $\frac{N_{\textnormal{mo}}}{3}$ from $V(C)$, so if $C\neq C_*$, then $C, C_1$ are distinct inner rings and we contradict M4). Thus we have $C=C_*$, and $\textnormal{Int}(K')\setminus V(K')$ is disjoint to $\textnormal{Int}(F)$. Since $F$ separates two elements of $\mathcal{C}\setminus\{C\}$, it follows from 1) of Theorem \ref{MainRes2CloseCycleBounds} that there is a $C_2\in\mathcal{C}$ with $C_2\subseteq\textnormal{Int}(F)$ and $d(\mathpzc{w}(C_2), F)<2.9\beta\cdot 6^{g-1}+4N_{\textnormal{mo}}^2$. Since $C_1\neq C_2$ and each of $C_1, C_2$ is an inner ring, we again contradict M4). \end{claimproof}

Recall that $F\cap C$ consists of at most one vertex. Since $d(F, C)\leq 1$, there is a shortest $(F,C)$-path $R$, where $|E(R)|\leq 1$. 

\begin{claim}\label{Chords2ChordsFromFToC}  Any $(F,C)$-path of length at most two has at least one endpoint in $R$.
 \end{claim}

\begin{claimproof} Suppose there is such a path $P$ violating the claim. Since $F$ is undesirable, it follows that there is a proper generalized chord $Q$ of $C$ violating Lemma \ref{FirstCasePropChordClosedRingLem}, where, in particular, this proper generalized chord of $C$ is a path of length at most $|E(F)|+3\leq\frac{2N_{\textnormal{mo}}}{5}$. \end{claimproof}
 
We now have the following. 

\begin{claim}\label{OneSideColorabilityClaimForLem2} $V(G_1\cup C)$ is $L$-colorable \end{claim}

\begin{claimproof} Since $C$ is shortcut-free in $G$, there are no chords of $C$.  By Claim \ref{Chords2ChordsFromFToC}, there is a $u\in V(F\cup C)$ such that any $(F,C)$-path of length at most two has $u$ as an endpoint. Let $G_1^+=(G_1\cup C)+E(F\setminus C, C\setminus F)$. Possibly the edge-set $E(F\setminus C, C\setminus F)$ is empty and $R$ has length zero. In any case, since $F$ has no chords in $G_0$ and $C$ has no chords, we get that $G_1^+$ is an induced subgraph of $G$, so it suffices to show that $G_1^+$ is $L$-colorable. To show this, we construct a mosaic whose underlying graph has fewer vertices than $G$ and contains $G_1^+$ as a subgraph. 

Let $U^+$ be the unique component of $\Sigma\setminus G_1^+$ such that $\partial(U^+)$ contains at least one of the cycles $F, C$. Note that $u$ is a cut-vertex of $G_1^+$. In particular, $\partial(U^+)$ is a closed facial walk of $G_1^+$, but not a cycle. Let $e^1e^2\cdots e^r$ denote this walk. Note that this walk contains at most two edges of $E(F\setminus C, C\setminus F)$, and $r\leq |E(C)|+|E(F)|+2$. 

 By 1) of Claim \ref{SigmaDaggerSurfaceFactsLem2}, $U^+$ is an open disc. We now ``pad" $U^+$ with a cycle in the following way: Let $G_1'$ be a graph obtained from $G_1^+$ by adding to $U^+$ a cycle $K$ of length $r$, where, for each $x\in V(K)$, the graph $G[N(x)\cap V(C\cup F)]$ is an edge of $e_x$ of $e^1\cdots e^r$, and, for any two vertices $x,y\in V(K)$, $xy$ is an edge of $K$ if and only if $e_x$ and $e_y$ are consecutive in $e^1\cdots e^r$. Note that $K$ is contractible, as each of the cycles of $C\cup Q$ is contractible. The idea now is to construct a new, smaller mosaic whose underlying graph is an embedding on $\Sigma$. The outer ring of this new mosaic will still be $C_*$. We define an embedding $G_1^{\dagger}$ on $\Sigma$ which is obtained from $G_1'$ by adding a $K$-web in $\mathcal{P}$ in the open disc of $\Sigma$ which is $\mathcal{T}$-externally bounded by $K$. Note that it follows from Claim \ref{SigmaDaggerSurfaceFactsLem2} that the open disc of $\Sigma$ which is $\mathcal{T}$-externally bounded by $K$ is contained within $U'$, and that $C_*\subseteq G_1\cup C\subseteq G^{\dagger}_1$. 

\vspace*{-8mm}
\begin{addmargin}[2em]{0em}
\begin{subclaim} $|V(G_1^{\dagger})|<|V(G)|$. \end{subclaim}

\begin{claimproof}Since $|E(K)|\leq\frac{5N_{\textnormal{mo}}}{3}$, we just need to check that $|V(G_0)\setminus V(C\cup F)|>25N_{\textnormal{mo}}^2$, and then it follows from 3) of Proposition \ref{InAndOutWebFactsProp} that $|V(G_1^{\dagger})|<|V(G)|$. By Claim \ref{SigmaDaggerSurfaceFactsLem2}, $F$ separates two elements of $\mathcal{C}\setminus\{C\}$, so it separates an element of $\mathcal{C}\setminus\{C\}$ from $C$. Since $\beta=100N_{\textnormal{mo}}^2$, it follows from our distance conditions on $\mathcal{T}$ that $|V(G_0)\setminus V(C\cup Q)|>25N_{\textnormal{mo}}^2$, as desired. \end{claimproof}\end{addmargin}

Let $L^{\dagger}$ be a list-assignment for $V(G_1^{\dagger})$, where all the vertices of $G^1_{\dagger}\setminus G$ are given arbitrary 5-lists, and otherwise $L^{\dagger}=L$. Let $\mathcal{T}^{\dagger}:=(\Sigma, G_1^{\dagger}, \mathcal{C}^{\subseteq G_1}, L^{\dagger}, C^*)$. It follows from 1) of Proposition \ref{InAndOutWebFactsProp} that $\mathcal{T}^{\dagger}$ is a tessellation, where $C$ is a closed $\mathcal{T}^{\dagger}$-ring, and it follows from our construction of $G_1^{\dagger}$ that $\mathcal{T}^{\dagger}$ satisfies M0)-M2). It also follows from 2) of Proposition \ref{InAndOutWebFactsProp} that $\mathcal{T}^{\dagger}$ still satisfies the distance conditions M3) and M4). Now we just need to check M5). All the noncontractible cycles of $G_1^{\dagger}$ lie in $\textnormal{Cl}(U)$, so we get $\textnormal{ew}(G_1^{\dagger})\geq\textnormal{ew}(G_1)\geq\textnormal{ew}(G)$. Furthermore, since $G$ is a 2-cell embedding, it follows from our construction of $G_1^{\dagger}$ that $G_1^{\dagger}$ is also a 2-cell embedding. Since $\textnormal{ew}(G)\geq 2.1\beta\cdot 6^{g}$, it follows from Fact \ref{HighEwTriangleFwF1Cycle}, together with our distance conditions on $\mathcal{T}$, that $\textnormal{fw}(G_1^{\dagger})\geq\textnormal{fw}(G_1)\geq 2.1\beta\cdot 6^{g-1}$ We conclude that $\mathcal{T}^{\dagger}$ satisfies M5), so $\mathcal{T}^{\dagger}$ is a mosaic. Since $|V(G_1^{\dagger})|<|V(G)|$, it follows that $G_1^{\dagger}$ is $L^{\dagger}$-colorable, and thus $G_1\cup C$ is $L$-colorable, as desired. \end{claimproof}

Applying Claim \ref{OneSideColorabilityClaimForLem2}, let $\phi$ be an $L$-coloring of $V(G_1\cup C)$. We now show that $\phi$ extends to an $L$-coloring of $G$. Since $C$ satisfies M2), it follows from Claims \ref{ChordsAnd2ChordsForFSide}-\ref{Chords2ChordsFromFToC} that there exists an $z\in V(G_0)\cap D_1(C\cup F)$ such that $|L_{\phi}(z)|\geq 2$ and $|L_{\phi}(y)|\geq 3$ for each $y\in (V(G_0)\cap D_1(C\cup F))\setminus\{z\}$. Let $G':=G\setminus V(C\cup G_1)$ and let $L'$ be a list-assignment for $V(G')$ where $L'(x)$ consists of a lone color of $L_{\phi}(x)$ and otherwise $L'=L_{\phi}$. Regarding $G'$ as an embedding on $\Sigma_1$, we let $F'$ be the unique facial subgraph of $G'$ which consists of all the vertices of $V(G_0)\cap D_1(F\cup C)$. Let $\mathcal{T}':=(\Sigma_1, G', \mathcal{C}^{\subseteq G_0}\setminus\{C\}, L', F')$. Now, $\mathcal{T}'$ is a tessellation, where the outer ring is an open $\mathcal{T}'$-ring with precolored path $z$. We claim now that $\mathcal{T}'$ is a mosaic. Firstly, it follows from Claim \ref{SigmaDaggerSurfaceFactsLem2} that $C_*\in\mathcal{C}^{\subseteq G_0}$ if and only if $C=C_*$, so each closed, internal $\mathcal{T}'$-ring is highly $L'$-predictable. Thus, M2) is satisfied, as the outer ring of $\mathcal{T}'$ is open. Furthermore, since the precolored path of $F'$ is a lone vertex, M0) and M1) are also immediate. By 1) of Claim \ref{SigmaDaggerSurfaceFactsLem2}, $\Sigma_1=\mathbb{S}^2$, so M5) is trivally satisfied. Since $F'$ is the outer ring of $\mathcal{T}'$, M4) is immediate as well so we just need to check that $\mathcal{T}'$ satisfies M3). 

\begin{claim} $\mathcal{T}'$ satisfies M3). \end{claim}

\begin{claimproof} Firstly, every vertex of $F'$ has distance at most $\frac{N_{\textnormal{mo}}}{3}$ from $C$.  We have $\textnormal{Rk}(\mathcal{T}'|F')=2N_{\textnormal{mo}}^2$, i.e the rank of $F'$ is greater than that of $C$, as $F'$ is an open $\mathcal{T}'$-ring and $C$ is a closed $\mathcal{T}$-ring, but since $F'$ is the outer ring of $\mathcal{T}'$, and the genus has not increased, it is immediate that $\mathcal{T}'$ satisfies M4) if $C$ is an inner ring of $\mathcal{T}$.. Now suppose that $C=C_*$. Thus, we have $G_0=\textnormal{Ext}_{\mathcal{T}}(F)$ and, by 1) of Claim \ref{SigmaDaggerSurfaceFactsLem2}, $\mathcal{C}^{\subseteq\textnormal{Int}_{\mathcal{T}}(F)}\neq\varnothing$. 

Since $V(F\cap C_*)\neq\varnothing$, it follows from 1) of Theorem \ref{MainRes2CloseCycleBounds} that there is a $C'\in\mathcal{C}^{\subseteq\textnormal{Int}_{\mathcal{T}}(F)}$ with $d(C_*, \mathpzc{w}(C'))<2.9\beta\cdot 6^{g-1}+4N_{\textnormal{mo}}^2$. For each $C''\in\mathcal{C}^{\subseteq G_0}$, we have $d(\mathpzc{w}(C''), C')\geq 2\beta\cdot 6^g$ by our distance conditions on $\mathcal{T}$, as each of $C', C''$ is an inner ring of $\mathcal{T}$. As $C=C_*$ is a closed ring, any two vertices of $C_*$ are of distance at most $\frac{N_{\textnormal{mo}}}{2}$ apart, it follows that, for each $C''\in\mathcal{C}^{\subseteq G_0}$, we have $d(\mathpzc{w}(C''), C_*)> (2\beta\cdot 6^g-2.9\beta\cdot 6^{g-1})-4N_{\textnormal{mo}}^2$ and thus $$d(\mathpzc{w}(C''), F)\geq (12-2.9)\beta\cdot 6^{g-1}-4N_{\textnormal{mo}}^2-\left(\frac{N_{\textnormal{mo}}}{6}-1\right)>2.9\beta\cdot 6^{g-1}+4N_{\textnormal{mo}}^2$$
 so $\mathcal{T}'$ satisfies M3).  \end{claimproof}

Since $\mathcal{T}'$ satisfies all of M0)-M5), $\mathcal{T}'$ is indeed a mosaic. Since $|V(G')|<|V(G)|$ and $\mathcal{T}$ is critical, we conclude that $G'$ is $L'$-colorable, so $\phi$ extends to an $L$-coloring of $G$, a contradiction. \end{proof}

Combining Lemmas \ref{FirstCasePropChordClosedRingLem} and \ref{ImpropGenChordClosedRing2Lem}, we complete the proof of Theorem \ref{ShortGenChordClosedRingThm}.

\section{Generalized Chords of Open Rings in Critical Mosaics}\label{GenChordOpenRingCritSec2}

The purpose of Section \ref{GenChordOpenRingCritSec2} is to prove an analogue to Theorem \ref{ShortGenChordClosedRingThm} for open rings. In order to prove this result, we first prove the following lemma.

\begin{lemma}\label{OneSideColorableOpRingLem1} Let $\mathcal{T}=(\Sigma, G, \mathcal{C}, L, C_*)$ be a critical mosaic and let $C\in\mathcal{C}$ be an open ring. Let $Q$ be a proper generalized chord of $C$ with $|E(Q)|\leq\frac{2N_{\textnormal{mo}}}{3}$. Let $G=G_0\cup G_1$ be the natural $Q$-partition of $G$ and, for each $j=0,1$, let $U_j$ be the unique component of $\Sigma\setminus (C\cup Q)$ such that $G\cap\textnormal{Cl}(U_j)=G_j$. Suppose that $\textnormal{Cl}(U_1)$ contains either a noncontractible closed curve of $\Sigma$ or an element of $\mathcal{C}\setminus\{C\}$. Then $G_0$ is $L$-colorable. \end{lemma}

\begin{proof} For each $j=0,1$, let $C_j:=(C\cap G_j)+Q$. Since $Q$ is a proper generalized chord of $C$, each of $C_0, C_1$ is a cycle. We construct a mosaic whose underlying subgraph has strictly fewer vertices than $G$ and has $G_0\cup C$ as a subgraph. Since $G$ is a 2-cell embedding, it follows that $\textnormal{Cl}(U_1)$ either contains a noncontractible cycle of $G$ or an element of $\mathcal{C}\setminus\{C\}$, so it follows from either our face-width conditions or our distance conditions on $\mathcal{T}$ that $|V(G_1\setminus C_1)|\neq\varnothing$. Let $\Sigma^{\dagger}$ be a surface obtained from $\Sigma$ by replacing $U_1$ with an open disc $U^{\dagger}$. Note that, if $U_1$ is not externally bounded by $C_1$, then $C\neq C_*$. We also note that $\Sigma^{\dagger}$ is connected. The idea here is to apply a construction similar to that of Definition \ref{InAndOutWebDefn} to construct a new mosaic whose outer ring is $C$, whose underlying surface is $\Sigma^{\dagger}$ and whose underlying embedding on $\Sigma^{\dagger}$ contains $G_0\cup C$ as a subgraph and has fewer vertices than $G$. Because $C$ is an open ring, the path $C\cap G_1$ can be arbitrarily long, so we need to do this in two steps in order to bound the size of the new embedding we construct. Now, $G_0\cup C$ can be regarded as an embedding on $\Sigma^{\dagger}$, so we let $\mathcal{P}$ be the pre-chart $(\Sigma^{\dagger}, G_0\cup C, C)$. Let $k=\min\{|E(Q)|+1, |E(C\cap G_1)|\}$. Regarding $G_0\cup C$ as an embedding on $\Sigma^{\dagger}$, we define an embedding $H$ on $\Sigma^{\dagger}$ which is obtained from $G_0\cup C$ by doing the following.

\begin{enumerate}[label=\arabic*)]
\itemsep-0.1em
\item First we add $k$ new vertices $w_1, \ldots, w_{k}$ to the open disc which replaces $U_1$, and we add some edges to the closure of this open disc, so that $w_1\ldots w_{k}$ is a path and, for each $1\leq i<i'\leq k$, the neighborhoods of $w_i$ and $w_{i'}$ on $C_1$ are subpaths of $C_1$ of length at least one which intersecting precisely on a common endpoint, where the respective neighborhoods of $w_1$ and $w_{k}$ on $C_1$ are disjoint terminal subpaths of $C\cap G_1$. 
\item Now, letting $K$ be the cycle $Q+w_1\ldots w_{k}$, we add a $K$-web to the open disc of $\Sigma^{\dagger}$ which is $\mathcal{P}$-externally bounded by $K$. 
\end{enumerate}

Let $L^{\dagger}$ be a list-assignment for $H$ where each vertex of $H\setminus G_0$ is given arbitrary 5-list and otherwise $L^{\dagger}=L$. Since $|E(K)|\geq 5$, it follows from Proposition \ref{InAndOutWebFactsProp} that $H$ is short-inseparable and $\mathcal{T}^{\dagger}:=(\Sigma^{\dagger}, H, L^{\dagger}, \{C\}\cup\mathcal{C}^{\subseteq G_0}, C)$ is a tessellation, where $C$ is an open $\mathcal{T}^{\dagger}$-ring with precolored path $\mathbf{P}^{\dagger}:=\mathbf{P}_{\mathcal{T}, C}\cap G_0$. Recall that, by assumption, we have $\mathbf{P}_{\mathcal{T}, C}\cap G_0\neq\varnothing$. We claim now that $\mathcal{T}^{\dagger}$ is a mosaic. M0), M1), and M2) are immediate. Applying Proposition \ref{InAndOutWebFactsProp} to our construction, there is no $F\in\mathcal{C}^{\subseteq G_0}$ such that $d_H(\mathpzc{w}(F), C\setminus\mathbf{P}^{\dagger})<d_G(\mathpzc{w}(F), C\setminus\mathbf{P}_{\mathcal{T}, C})$, as lone vertex of $K\setminus C$ has no neighbors in $\mathring{Q}$. Thus, $\mathcal{T}^{\dagger}$ also satisfies the distance conditions M3)-M4).  Let $g:=g(\Sigma^{\dagger})$.

\begin{claim}\label{TUpdownarrowM5} $\mathcal{T}^{\dagger}$ satisfies M5).  \end{claim}

\begin{claimproof}. If $\textnormal{Cl}(U_1)$ contains a noncontractible closed curve of $\Sigma$, then $g(\Sigma^{\dagger})=0$ so M5) is trivially satisfied in that case. Now suppose that $\textnormal{Cl}(U_1)$ contains no noncontractible closed curve of $\Sigma$. Thus, we have $g=g(\Sigma)$.

\vspace*{-8mm}
\begin{addmargin}[2em]{0em}
\begin{subclaim} $\textnormal{ew}(H)\geq 2.1\beta\cdot 6^{g}$. \end{subclaim}

\begin{claimproof} Suppose not. Thus, there is a noncontractible cycle $F$ of $H$ with $|E(F)|<2.1\beta\cdot 6^{g}$. It follows from our construction of $H$ that $d_H(x,y)=d_{G_0\cup C}(x,y)$ for any $x,y\in V(Q)$. As $U^{\dagger}$ is an open disc, it follows that there is a noncontractible cycle $F'$ of $G_0\cup C$ which is obtained from $F$ by replacing (if necessary) some edges of $H\setminus (G_0\cup C)$ with a subpath of $C_1$ whose endpoints lie in $F\cap Q$, where $|E(F')|\leq |E(F)|$. Thus, $\textnormal{ew}(G_0\cup C)<2.1\beta\cdot 6^{}$. As $\textnormal{ew}(G_0\cup C)\geq\textnormal{ew}(G)$, this contradicts the fact that $\mathcal{T}$ satisfies M5). \end{claimproof}\end{addmargin}

Now we just need to check that the face-width conditions are satisfied. As $G$ is a 2-cell embedding on $\Sigma$, it  follows from our construction of $H$ that $H$ is a 2-cell embedding on $\mathcal{T}^{\dagger}$. Combining this with our distance conditions and the fact that $\textnormal{fw}(G)\geq 2.1\beta\cdot 6^{g-1}$, it follows from Fact \ref{HighEwFwF2Chord} that $\textnormal{fw}(H)\geq 2.1\beta\cdot 6^{g-1}$, so $\mathcal{T}^{\dagger}$ satisfies M5). \end{claimproof}

Since $\mathcal{T}^{\dagger}$ satisfies all of M0)-M5), we conclude that $\mathcal{T}^{\dagger}$ is indeed a mosaic. 

\begin{claim}\label{HStrictlySmallerThanGClaim2} $|V(H)|<|V(G)|$. \end{claim}

\begin{claimproof} Note that $|E(K)|\leq 2|E(Q)|+1$, so $|E(K)|\leq 2N_{\textnormal{mo}}$. It follows from 3) of Proposition \ref{InAndOutWebFactsProp} that the closed disc $\textnormal{Cl}(U^{\dagger})$ of $\Sigma^{\dagger}$ contains at most $36N_{\textnormal{mo}}^2$ vertices of $H$, and we get $|V(H)\cap\textnormal{Cl}(U^{\dagger})|\leq |V(C_1)|+36N_{\textnormal{mo}}^2$. We just need to show that $|V(G)\cap\textnormal{Cl}(U_1)|>|V(C_1)|+36N_{\textnormal{mo}}^2$. Since $|E(Q)|\leq\frac{2N_{\textnormal{mo}}}{3}$, it follows from our distance conditions that every element of $\mathcal{C}\setminus\{C\}$ has distance at least $2.9\beta\cdot 6^{g(\Sigma)-1}$ from $C_1$. If there is an element of $\mathcal{C}\setminus\{C\}$ in $\textnormal{Cl}(U_1)$, then we are done, so suppose that no such element of $\mathcal{C}\setminus\{C\}$ exists. Thus, every facial subgraph of $G\cap\textnormal{Cl}(U_1)$, except for $C_1$, is a triangle, and $G\cap\textnormal{Cl}(U_1)$ contains a noncontractible cycle. In particular, $g(\Sigma)>0$ and $\textnormal{fw}(G)\geq 2.9\beta$. As $\textnormal{fw}(G)>2$, $G\cap\textnormal{Cl}(U_1)$ contains a noncontractible cycle $K$ such that $K$ contains at most two chords of $C$, and since $\textnormal{fw}(G)\geq 2.9\beta$ it follows that $(G\cap\textnormal{Cl}(U_1))\setminus E(C)$ contains at least $2.9\beta-3$ edges which each have at most one endpoint in $V(C)$.  As $|E(C_1)\setminus E(C)|\leq\frac{2N_{\textnormal{mo}}}{3}$, we have $|V(G)\cap\textnormal{Cl}(U_1)|>|V(C_1)|+36N_{\textnormal{mo}}^2$. \end{claimproof}

Since $\mathcal{T}^{\dagger}$ is a mosaic and $|V(H)|<|V(G)|$, it follows from the criticality of $\mathcal{T}$ that $H$ is $L^{\dagger}$-colorable and so $G_0$ is $L$-colorable, as desired. \end{proof}

We now state and prove the lone main result of Section \ref{GenChordOpenRingCritSec2}. As in the statement of Theorem \ref{ShortGenChordClosedRingThm}, the natural $Q$-partition of $G$ in the statement below is well-defined by the face-width conditions on mosaics. 

\begin{theorem}\label{GeneralChordCritMosaicOpMain0} Let $\mathcal{T}=(\Sigma, G, \mathcal{C}, L, C_*)$ be a critical mosaic and let $C\in\mathcal{C}$ be an open ring. Let $Q$ be a proper generalized chord of $C$ with $|E(Q)|\leq\frac{2N_{\textnormal{mo}}}{3}$. Let $G=G_0\cup G_1$ be the natural $Q$-partition of $G$, and suppose that $\mathbf{P}_C\cap G_0$ is connected and has at least one edge. Suppose further that $|E(\mathbf{P}_C\cap G_1)|+|E(Q)|\leq \frac{2N_{\textnormal{mo}}}{3}$. Then every element of $\mathcal{C}\setminus\{C\}$ lies in $G_0$, and the component of $\Sigma\setminus (C\cup Q)$ whose closure contains $G_1$ is a disc. \end{theorem}

\begin{proof} By 4) of Proposition \ref{BasicPropertiesCricMosProp}, $\mathbf{P}_C$ is a path of length at least one, and since $C$ is an open ring, we have $V(C)\setminus V(\mathbf{P}_C)\neq\varnothing$. Given a proper generalized chord $Q$ of $C$, we say that $Q$ is \emph{$\mathbf{P}_C$-splitting} if it has one endpoint in $V(\mathring{\mathbf{P}}_C)$ and the other endpoint in $V(C)\setminus V(\mathbf{P}_C)$. In particular, we have the following simple observation. 

\begin{claim}\label{SplittingChordUniquelySpecifiedClaim} Let $Q$ be a proper generalized chord of $C$ and let $C_0, C_1$ be the two cycles of $C\cup Q$ which are distinct from $C$. Then precisely one of the following holds.
\begin{enumerate}[label=\arabic*)]
\itemsep-0.1em
\item $Q$ is $\mathbf{P}_C$-splitting; OR
\item There is precisely one $j\in\{0,1\}$ such that $\mathbf{P}_C\cap C_j$ either has two connected components or at most one vertex.
\end{enumerate}\end{claim}

Given Claim \ref{SplittingChordUniquelySpecifiedClaim}, we introduce the following notation. For any proper generalized $Q$ of $C$ of length at most $\frac{2N_{\textnormal{mo}}}{3}$ and any endpoint $p$ of $\mathbf{P}_C$, we let $G=G_0^{Q, p}\cup G_1^{Q, p}$ be the natural $Q$-partition of $G$, where $\mathbf{P}_C\cap G^{Q, p}_0$ is a path of length at least one, and either $Q$ is not $\mathbf{P}_C$-splitting or $p\in V(G^{Q, p}_1)\setminus V(G^{Q,p}_0)$. By Claim \ref{SplittingChordUniquelySpecifiedClaim}, $G^{Q, p}_0$ and $G^{Q, p}_1$ are uniquely specified by the definition above. For each endpoint $p$ of $\mathbf{P}_C$, let $\mathcal{B}_p$ be the set of proper generalized chords $Q$ of $C$ with $E(\mathbf{P}_C\cap G^{Q, p}_1)+|E(Q)|\leq\frac{2N_{\textnormal{mo}}}{3}$, where $G^{Q, p}_1$ either contains a noncontractible cycle or an element of $\mathcal{C}\setminus\{C\}$. Suppose toward a contradiction that Theorem \ref{GeneralChordCritMosaicOpMain0} does not hold. Thus, letting $p, p'$ be the endpoints of $\mathbf{P}_C$, we have $\mathcal{B}_p\cup\mathcal{B}_{p'}\neq\varnothing$. We now choose a $q\in\{p, p'\}$ and a $Q\in\mathcal{B}_q$ which minimizes $|V(G^{Q, q}_1)|$ among all the elements of $\mathcal{B}_p\cup\mathcal{B}_{p'}$. For each $j=0,1$, we let $C_j:=(C\cap G_j^{Q, q})+Q$, and, to avoid clutter, we let $G_j:=G^{Q, q}_j$. 

\begin{claim}\label{ChordOneEndinQClaim} Let $P\subseteq G_1$ be a $k$-chord of $C_1$ for some $1\leq k\leq 2$, where $P$ has one endpoint in $\mathring{Q}$ and the other endpoint in $C\cap G_1$. Then the following hold.
\begin{enumerate}[label=\arabic*)]
\itemsep-0.1em
\item $|E(P)|=2$; AND
\item If both endpoints of $P$ lie in $Q\cup (\mathbf{P}_C\cap G_1)$, then the endpoints of $P$ are consecutive in $Q$.
\end{enumerate}
\end{claim}

\begin{claimproof} Suppose toward a contradiction that, for some $1\leq k\leq 2$, $G_1$ contains such a $k$-chord $P$ of $C_1$ which violates Claim \ref{ChordOneEndinQClaim}. Let $x,y$ be the endpoints of $P$, where $x\in V(\mathring{Q})$ and $y\in V(C\cap G_1)$. Let $Q_0, Q_1$ be the two proper generalized chords of $C$ such that $Q_0\cup Q_1=Q+P$ and $Q_0\cap Q_1=P$. By our face-width conditions, each of the three cycles in $C_1+P$ is contractible, so let $G_1=H_0\cup H_1$ be the natural $(C_1, P)$-partition of $G_1$, where, for each $i=0,1$, $H_i\cap (Q+P)=Q_i$. 

\vspace*{-8mm}
\begin{addmargin}[2em]{0em}
\begin{subclaim}\label{IneqSubclDaggerReplace} For each $i=0,1$, $|E(\mathbf{P}_C\cap H_i)|+|E(Q_i)|\leq |E(\mathbf{P}_C\cap G_1)|+|E(Q)|\leq\frac{2N_{\textnormal{mo}}}{3}$. \end{subclaim}

\begin{claimproof} By our choice of $Q$, $|E(\mathbf{P}_C\cap G_1)|+|E(Q)|\leq\frac{2N_{\textnormal{mo}}}{3}$. Suppose that Subclaim \ref{IneqSubclDaggerReplace} does not hold, and suppose for the sake of definiteness that $|E(\mathbf{P}_C\cap H_0)|+|E(Q_0)|>|E(\mathbf{P}_C\cap G_1)|+|E(Q)|$. In particular, since $\mathbf{P}_C\cap H_0$ is a (possibly empty) subpath of $\mathbf{P}_C\cap G_1)$, we have $|E(Q_0)|>|E(Q)|$. Thus, $x$ is an endpoint of $\mathring{Q}$ and $|E(P)|=2$. In particular, letting $xx'$ be the unique terminal edge of of $Q$ incident to $x$, we have $Q_0=(Q-x')+P$ and $|E(Q_0)|=|E(Q)|+1$. Furthermore, $Q_1$ has endpoints $x', y$. Since $|E(P)|=2$, $P$ violates 2) of Claim \ref{ChordOneEndinQClaim}, so $y\in V(\mathbf{P}_C\cap G_1)$. Since $|E(Q_0)|=|E(Q)|+1$, and since $\mathbf{P}_C\cap G_0$ is connected and has at least one edge, we get $\mathbf{P}_C\cap H_0=\mathbf{P}_C\cap G_1$, so $y$ is an endpoint of $\mathbf{P}_C$ and $x'\not\in V(\mathbf{P}_C)$. Furthermore, letting $x''$ be the endpoint of $Q$ which is distinct from $y$, we have $x''\in V(\mathbf{P}_C)$. In particular, $x''$ is an internal vertex of $\mathbf{P}_C$, as $\mathbf{P}_C\cap G_0$ has at least one edge. Let $D:=(Q-x')+P+(C\cap H_0)$. Note that $D$ is a cyclic facial subgraph of $H_0$. Furthermore, since $x''\in V(\mathring{P}_C)$ and $y$ is an endpoint of $\mathbf{P}_C$, and since $H_1$ contains no edges of $\mathbf{P}_C$, it follows that $C\cap H_0$ is a terminal subpath of $\mathbf{P}_C$. Since $|E(\mathbf{P}_C\cap G_1)|+|E(Q)|\leq\frac{2N_{\textnormal{mo}}}{3}$, we have $|E(D)|\leq\frac{2N_{\textnormal{mo}}}{3}$ as well. Furthermore, since $x'\not\in V(\mathbf{P}_C)$, $Q$ is $\mathbf{P}_C$-splitting. Consider the following cases.

\textbf{Case 1:} Either $H_0$ contains an element of $\mathcal{C}\setminus\{C\}$ or the unique open component of $\Sigma\setminus (C_1\cup P)$ whose closure contains $H_0$ is not a disc. 

Note that, since $C$ is an open ring, $\textnormal{Rk}(C)=2N_{\textnormal{mo}}^2$. We first show that $H_0=\textnormal{Ext}(D)$. Suppose $H_0=\textnormal{Int}(D)$. Thus, either $D$ is not inward contractible or $\textnormal{Int}(D)$ contains an element of $\mathcal{C}\setminus\{C\}$. Letting $g=g(\Sigma)$ and $g'=g(\Sigma^D)$, we have $g\geq\max\{g', g-g'\}$, and, since $D\cap C\neq\varnothing$, we contradict Corollary \ref{NonRingSeparatingCor}. Thus, $H_0=\textnormal{Ext}(D)$. In particular, $C\neq C^*$ and $D$ separates $C^*$ from $C$, so $C\subseteq\textnormal{Int}(D)$. Since $C$ is an open ring and $V(C\cap D)\neq\varnothing$, this contradicts Corollary \ref{NonRingSeparatingCor}. 

\textbf{Case 2:} $H_0$ contains no element of $\mathcal{C}\setminus\{C\}$ and the unique open component of $\Sigma\setminus (C_1\cup P)$ whose closure contains $H_0$ is a disc. 

In this case, by assumption on $Q$, either $H_1$ it contains an element of $\mathcal{C}\setminus\{C\}$ or the unique open component of $\Sigma\setminus (C_1\cup P)$ whose closure contains $H_1$ is a not disc. Let $R=x'x+P$. Note that $R$ is a proper 3-chord of $C$, neither endpoint of which is an internal vertex of $\mathbf{P}_C$, so, for each endpoint $p^*$ of $\mathbf{P}_C$, we have $H_1=G_1^{R, p^*}$. As $|V(H_1)|<|V(G_1)|$, we contradict the minimality of $|V(G_1)|$.  \end{claimproof}\end{addmargin}

It follows from Subclaim \ref{IneqSubclDaggerReplace} that each $Q_0$ and $Q_1$ has length at most $\frac{2N_{\textnormal{mo}}}{3}$. Each of $Q_0, Q_1$ is a proper generalized chord of $C$, so, for each $i\in\{0,1\}$ and endpoint $p_*$ of $\mathbf{P}_C$, each of $G_0^{Q_i, p_*}$ and $G_1^{Q_i, p_*}$ is well-defined. 

\vspace*{-8mm}
\begin{addmargin}[2em]{0em}
\begin{subclaim}\label{HSideIs1SideSub} There exists an $i\in\{0,1\}$ such that, for each endpoint $p_*$ of $\mathbf{P}_C$, we have $H_i=G_0^{Q_i, p_*}$.  \end{subclaim}

\begin{claimproof} Suppose not. Thus, there exist endpoints $p_0, p_1$ of $\mathbf{P}_C$ (where possibly $p_0=p_1$) such that $H_i=G_1^{Q_i, p_i}$ for each $i=0,1$. Since $Q\in\mathcal{B}_q$ and $H_0\cup H_1=G_1$, it follows from $(\dagger)$ that there is at least one $i\in\{0,1\}$ such that $H_i$ contains either a noncontractible cycle or an element of $\mathcal{C}\setminus\{C\}$. Since $|V(H_i)|<|V(G_1)|$ for each $i=0,1$, we contradict the minimality of $|V(G_1)|$.   \end{claimproof}\end{addmargin}

\vspace*{-8mm}
\begin{addmargin}[2em]{0em}
\begin{subclaim}\label{QEndpointPrec} At least one endpoint of $Q$ lies in $\mathring{\mathbf{P}}_C$, and furthermore, $Q$ is not $\mathbf{P}_C$-splitting. \end{subclaim}

\begin{claimproof} If neither endpoint of $Q$ lies in $\mathring{\mathbf{P}}_C$, then, for each $i=0,1$, the path $\mathbf{P}_C\cap H_i$ has at most one vertex and thus $H_i=G_1^{Q_i, q}$, contracting Subclaim \ref{HSideIs1SideSub}. Thus, at least one endpoint of $Q$ lies in $\mathring{\mathbf{P}}_C$. Now suppose toward a contradiction that $Q$ is $\mathbf{P}_C$-splitting. Thus, $\mathbf{P}_C\cap G_1$ is a path of length at least one which contains $q$. If $y\in V(C\cap G_1)\setminus\mathring{\mathbf{P}}_C$, then, for each $i=0,1$, we have $H_i=G_1^{Q_i, q}$, contradicting Subclaim \ref{HSideIs1SideSub}. Thus, we have $y\in V(C\cap G_1)\cap\mathring{\mathbf{P}}_C$. It follows that precisely one of $Q_0, Q_1$ has both endpoints in $\mathring{\mathbf{P}}_C$, and the other one is $\mathbf{P}_C$-splitting, so suppose without loss of generality that $Q_0$ has both endpoints in $\mathring{\mathbf{P}}_C$ and $Q_1$ is $\mathbf{P}_C$-splitting. Thus, we have $q\in V(H_1)\setminus V(Q_1)$ and furthermore, $\mathbf{P}_C\cap H_0$ has one connected component and $\mathbf{P}_C\setminus (H_0\setminus Q_0)$ has two connected components. We conclude that $H_0=G^{Q_0, q}_0$ and $H_1=G^{Q_1, q}_1$. Since $|V(H_1)|<|V(G_1)$, it follows from the minimality of $G_1$ that $Q_1\not\in\mathcal{B}_q$, so it follows from $(\dagger)$ that $H_0$ either contains a noncontractible cycle or an element of $\mathcal{C}\setminus\{C\}$. Since $|E(\mathbf{P}_C\cap G_1)|+|E(Q)|\leq\frac{2N_{\textnormal{mo}}}{3}$, it follows that $G$ contains a cycle $F$ with $|E(F)|\leq N_{\textnormal{mo}}$, where $F$ has nonempty intersection with $\mathbf{P}_C$ and $F$ separates $C$ from either a noncontractible cycle of $G$ or an element of $\mathcal{C}\setminus\{C\}$. Since $\textnormal{Rk}(C)=2N_{\textnormal{mo}}^2$ and each vertex of $C$ has distance at most $\frac{N_{\textnormal{mo}}}{3}$ from $\mathpzc{w}(C)$, we contradict Corollary \ref{NonRingSeparatingCor}. \end{claimproof}\end{addmargin}

It follows from Subclaim \ref{QEndpointPrec} that both endpoints of $Q$ lie in $\mathbf{P}_C$ and at least one of them is an internal vertex of $\mathbf{P}_C$. Thus, $\mathbf{P}_C\cap G_1$ has two connected components. 

\vspace*{-8mm}
\begin{addmargin}[2em]{0em}
\begin{subclaim} $y\in V(G\cap C_1)\cap V(\mathbf{P}_C)$. \end{subclaim}

\begin{claimproof} Suppose not. Thus, for each $i=0,1$, $\mathbf{P}_C\cap H_i$ is nonempty and connected. Let $q'$ be the endpoint of $\mathbf{P}_C$ which is distinct from $q$. Since $y\in V(G\cap C_1)\setminus V(\mathbf{P}_C)$, there exists an $i\in\{0,1\}$ such that $H_i=G^{q, Q_i}_1$ and $H_{1-i}=G^{q', Q_{1-i}}_1$, contradicting Subclaim \ref{HSideIs1SideSub}. \end{claimproof}\end{addmargin}

Since $y\in V(G\cap C_1)\cap V(\mathbf{P}_C)$, and both endpoints of $Q$ lie in $\mathbf{P}_C$, there exists an $i\in\{0,1\}$ such that $\mathbf{P}_C\cap H_i$ has one connected component and $\mathbf{P}_C\cap H_{1-i}$ has two connected components. Thus, $H_{1-i}$ contains no noncontractible cycles or elements of $\mathcal{C}\setminus\{C\}$, or else we contradict the minimality of $|V(G_1)|$. Since $G_1=H_0\cup H_1$, it follows that $H_i$ contains either an element of $\mathcal{C}\setminus\{C\}$ or a noncontractible cycle, so $(\mathbf{P}_C\cap H_i)+Q_i$ is a cycle of length at most $N_{\textnormal{mo}}$ which separates $C$ from either a noncontractible cycle of $G$ or an element of $\mathcal{C}\setminus\{C\}$, contradicting Corollary \ref{NonRingSeparatingCor}. This completes the proof of Claim \ref{ChordOneEndinQClaim}. \end{claimproof}

\begin{claim}\label{ChordsAnd2ChordsFromQClaim} $G_1$ contains no chord of $C_1$ with both endpoints in $Q$. Furthermore, for any 2-chord $P$ of $C_1$ with $P\subseteq G_1$ and both endpoints in $Q$, the endpoints of $P$ are consecutive vertices of $Q$. \end{claim}

\begin{claimproof} Suppose toward a contradiction that there is a path $P$ of length at most two which violates Claim \ref{ChordsAnd2ChordsFromQClaim}. Since $|E(Q)|+|E(\mathbf{P}_C\cap G_1)|\leq\frac{2N_{\textnormal{mo}}}{3}$ by assumption, and since the endpoints of $P$ are not consecutive in $Q$, it immediately follows that each vertex of $C_1\cup P$ has distance at most $\frac{N_{\textnormal{mo}}}{3}$ from $\mathpzc{w}(C)$. Possibly $P=C\cap G_1$ and the endpoints of $P$ are also the endpoints of $Q$. In that case, $C_1$ is a cycle of length at most $N_{\textnormal{mo}}$, and since $|V(C)|\geq 5$ by 4) of Proposition \ref{BasicPropertiesCricMosProp}, we have $|V(C\cap G_0)|>2$ and $C_1$ is a separating cycle of $G$. On the other hand, if $P\neq C\cap G_1$, then, letting $Q'$ be the proper generalized chord of $C$ obtained from $Q$ by replacing $xQy$ with $P$, we have $|E(Q')|\leq |E(Q)|$ by our assumption on $P$, and it follows from the minimality of $|V(G_1)|$ that $C$ lies on the opposite side of the cycle $(xQy)+xy$ from either a noncontractible cycle of $G$ or an element of $\mathcal{C}\setminus\{C\}$. In any case, since $|E(Q)|\leq\frac{2N_{\textnormal{mo}}}{3}$ and $G$ intersects with $\mathbf{P}_C$ at most on the endpoints of $\mathbf{P}_C$, it follows that there is a separating cycle $D$ of $G$ with $d(D, C\setminus\mathbf{P}_C)\leq\frac{N_{\textnormal{mo}}}{3}$, where $|E(D)|\leq N_{\textnormal{mo}}$ and $D$ separates $C$ from either a noncontractible cycle of $G$ or an element of $\mathcal{C}\setminus\{C\}$. As $\textnormal{Rk}(C)=2N_{\textnormal{mo}}$, we contradict Corollary \ref{NonRingSeparatingCor}. \end{claimproof}

Let $\mathbf{P}_1:=(\mathbf{P}_C\cap G_1)+Q$. Possibly $\mathbf{P}_C\cap G_1=\varnothing$, but in any case, since $G_0$ contains at least one edge of $\mathbf{P}_C$, it follows that $\mathbf{P}_1$ is a subpath of $C_1$, and it follows from Claim \ref{ChordsAnd2ChordsFromQClaim} that any edge of $G_1$ with both endpoints in $Q$ is an edge of $Q$. It follows from Claim \ref{ChordOneEndinQClaim} that $G_1$ has no chord of $C_1$ with an endpoint in $V(\mathring{Q})$ and the other endpoint in $V(C\cap G_1)$, so any $L$-coloring of $G_0$ is also a proper $L$-coloring of the subgraph of $G$ induced by $V(G_0)$ and extends to an $L$-coloring of the subgraph of $G$ induced by $V(G_0\cup\mathbf{P}_C)$, even if $G_1$ contains edges of $\mathbf{P}_C$. Thus, by Lemma \ref{OneSideColorableOpRingLem1}, $V(G_0)$ admits an $L$-coloring $\psi$, where $V(\mathbf{P}_1)$ is $L^Q_{\psi}$-colorable. Now, let $\mathcal{C}':=\{C_1\}\cup\{F\in\mathcal{C}: F\subseteq G_1\}$ and let $F_*$ be a cycle degined as follows. We set $F_*=C_*$ if $C\neq C_*$ and otherwise $F_*=C_1$. Finally, let $\mathcal{T}':=(\Sigma, G_1, \mathcal{C}', L^Q_{\psi}, F_*)$. Now, $\mathcal{T}'$ is a tessellation in which $C_1$ is an open ring with precolored path $\mathbf{P}_1$. In particular, we have $\mathpzc{w}_{\mathcal{T}'}(C_1)=V(C_1)\setminus V(\mathbf{P}_1)$, and thus $\mathpzc{w}_{\mathcal{T}'}(C_1)=\mathpzc{w}_{\mathcal{T}}(C)\cap V(G_1)$. 

\begin{claim} $\mathcal{T}'$ is a mosaic. \end{claim}

\begin{claimproof} Since $\mathpzc{w}_{\mathcal{T}'}(C_1)=\mathpzc{w}_{\mathcal{T}}(C)\cap V(G_1)$, it is immediate that $\mathcal{T}'$ satisfies the distance conditions M3)-M4). Since $|E(\mathbf{P}_1)|\leq\frac{2N_{\textnormal{mo}}}{3}$, M0) is satisfied and M2) is immediate. It follows from Claims \ref{ChordOneEndinQClaim} and \ref{ChordsAnd2ChordsFromQClaim} that M1) is still satisfied, so now we just need to check M5). We have $\textnormal{ew}(G^{Q, p}_1)\geq\textnormal{ew}(G)\geq 2.1\beta\cdot 6^{g(\Sigma)}$ by the monotonicity of edge-with, and since $\textnormal{fw}(G)\geq 2.1\beta\cdot 6^{g(\Sigma)-1}$, it follows from our distance conditions, together with Fact \ref{HighEwFwF2Chord}, that $\textnormal{fw}(G_1)\geq 2.1\beta\cdot 6^{g(\Sigma)-1}$, so $\mathcal{T}'$ satisfies M5) as well. \end{claimproof}

We claim now that $G_1$ is $L_{\psi}^Q$-colorable. If $|V(G_1)|<|V(G)|$, then this just follows immediately from the criticality of $\mathcal{T}$ so suppose that $|V(G_1)|=|V(G)|$. Thus, $V(G_0)=V(Q)$ and $|V(Q)|\geq 3$ (indeed, $G_0$ is a triangle), so $Q$ has a vertex with an $L$-list of size five. Thus, $\sum_{v\in V(G_1)}|L^Q_{\psi}(v)|<\sum_{v\in V(G)}|L(v)|$. It again follows from the criticality of $\mathcal{T}$ that $G_1$ is $L^Q_{\psi}$-colorable. In any case, $\psi$ extends to an $L$-coloring of $G$, a contradiction.  \end{proof}

To conclude Section \ref{GenChordOpenRingCritSec2}, we prove the following consequence of Theorem \ref{GeneralChordCritMosaicOpMain0}. 

\begin{prop}\label{LowGenus+OneFaceCase}
Let $\mathcal{T}=(\Sigma, G, \mathcal{C}, L, C_*)$ be a critical mosaic. Then each ring of $\mathcal{C}$ is an induced cycle. Furthermore, $|\mathcal{C}|+g(\Sigma)>1$ and, if $C_*$ is a closed ring, then $|\mathcal{C}|>1$. 
\end{prop}

\begin{proof} By M2), each closed ring is an induced cycle, so let $C$ be an open ring and suppose toward a contradiction that there is a chord $xy$ of $C$. By M1), we have $x,y\not\in V(\mathring{\mathbf{P}}_C)$. Let $G=G_0\cup G_1$ be the natural $xy$-partition. By 4) of Proposition \ref{BasicPropertiesCricMosProp}, $|E(\mathbf{P}_C)|\geq 1$, so there is a unique $j\in\{0,1\}$ with $\mathbf{P}_C\subseteq G_j$, say $j=0$ without loss of generality. For each $j=0,1$, let $C_j:=(G_j\cap C)+xy$.  By Theorem \ref{GeneralChordCritMosaicOpMain0}, each element of $\mathcal{C}\setminus\{C\}$ lies in $G_0$ and $G_1$ can be regarded as an embedding on a disc with facial cycle $C_1$. 

\begin{claim}\label{GPxyPrecSideColorable} $G_0$ is $L$-colorable. \end{claim}

\begin{claimproof} We define a cycle $C_*'$ as follows: We set $C_*'=C_*$ if $C$ is an inner ring of $\mathcal{T}$, and otherwise $C_*'=C_0$. Let $\mathcal{T}':=(\Sigma, G_0, (\mathcal{C}\setminus\{C\})\cup\{C_0\}, L, C_*')$. We claim now that $\mathcal{T}'$ is a mosaic. Firstly, $\mathcal{T}'$ is a tessellation, and, since $xy$ has one endpoint outside of $\mathbf{P}_C$, it follows that $C^{\textnormal{p}}_{xy}$ is an open $\mathcal{T}'$-ring with precolored path $V(\mathbf{P}_C)$, so $\mathcal{T}'$ satisfies M0), and it is clear that $\mathcal{T}'$ satisfies M1) as well. Since $C_0$ is an open ring of $\mathcal{T}'$ with the same precolored path as $C$, and since $V(C_0)\subseteq V(C)$, it is immediate that M2)-M4) are also satisfied, so we just need to check M5). Since edge-width is monotone, we have $\textnormal{ew}(G_0)\geq 2.1\beta\cdot 6^{g(\Sigma)}$, so it follows from Fact \ref{HighEwFwF2Chord}, together with our distance conditions, that $\textnormal{fw}(G_0)\geq 2.1\beta\cdot 6^{g(\Sigma)-1}$. \end{claimproof}

Applying Claim \ref{GPxyPrecSideColorable}, let $\psi$ be an $L$-coloring of $G_0$. By Theorem \ref{thomassen5ChooseThm}, $G_1$ is $L^{xy}_{\psi}$-colorable, so $\psi$ extends to an $L$-coloring of $G$, contradicting the fact that $\mathcal{T}$ is critical. Thus, each element of $\mathcal{C}$ is an induced cycle. Now suppose that either $|\mathcal{C}|+g(\Sigma)\leq 1$, or, if $C_*$ is a closed ring, then $\mathcal{C}=\{C_*\}$. Let $\psi$ be the unique $L$-coloring of $V(\mathbf{P}_{C_*})$ and let $F$ be the unique facial subgraph of $G\setminus\mathbf{P}_{C_*}$ consisting of all the vertices of $V(C_*\setminus\mathbf{P}_{C_*})\cup D_1(\mathbf{P}_{C_*})$.  Since $G$ is not $L$-colorable, $G\setminus\mathbf{P}_{C_*}$ is not $L_{\psi}$-colorable. Consider the following cases.

\textbf{Case 1:} $C_*$ is an open ring

In this case, $C_*\setminus\mathbf{P}_C$ is a nonempty path, and, by assumption, $\mathcal{C}=\{C_*\}$ and $\Sigma=\mathbb{S}^2$. By 1), $C_*$ is an induced cycle of $G$. Let $p, p'$ be the endpoints of the path $C_*\setminus\mathbf{P}_{C_*}$. Possibly $p=p'$. Since $C_*$ has no chords, it follows from M1) that every vertex of $V(F)\setminus\{p, p'\}$ has an $L_{\psi}$-list of size three, and furthermore, $\{p, p'\}$ either consists of one vertex with an $L_{\psi}$-list of size at least one, or two vertices with $L_{\psi}$-lists of size at most two. In the first case, it follows from Theorem \ref{thomassen5ChooseThm} that $G\setminus\mathbf{P}_{C_*}$ is $L_{\psi}$-colorable, and in the second case, it follows from Theorem \ref{Two2ListTheorem} that $G\setminus\mathbf{P}_{C_*}$ is $L_{\psi}$-colorable. In either case, we have a contradiction

\textbf{Case 2:} $C_*$ is a closed ring

In this case, $V(C_*)=V(\mathbf{P}_{C_*})$ and $V(F)=D_1(C_*)$ and, by assumption, $\mathcal{C}=\{C_*\}$, but possibly $g(\Sigma)>0$. Consider the tuple $\mathcal{T}'=(\Sigma, G\setminus C, \{F\}, L_{\psi}, F)$. By M2), there is a $z\in V(F)$ such that $|L_{\psi}(z)|\geq 2$, where each vertex of $F-z$ has an $L_{\psi}$-list of size at least three. We claim that $\mathcal{T}'$ is a mosaic, where the lone ring $F$ is an open ring with precolored path $z$. Since $F$ is the only ring, M3)-M4) are immediate. Since the precolored path of $F$ has length zero, all of M0)-M2) are immediate as well. We have $\textnormal{ew}(G\setminus C)\geq\textnormal{ew}(G)$. Since $|E(C)|\leq N_{\textnormal{mo}}$ and $\textnormal{ew}(G\setminus C)\geq\textnormal{ew}(G)$, it follows from Fact \ref{HighEwTriangleFwF1Cycle} that both bounds in M5) are satisfied as well and $\mathcal{T}'$ is a mosaic. By the criticality of $\mathcal{T}$, we get that $G\setminus C$ is $L_{\psi}$-colorable, which is false. \end{proof}

\section{Bands of Open Rings in Critical Mosaics}\label{BandOpenRingCritMosaicProperSec}

To state the main result for this section, we first introduce the following definitions.

\begin{defn}\label{QBandOpRingDefn}

\emph{Let $\mathcal{T}=(\Sigma, G, \mathcal{C}, L, C_*)$ be a critical mosaic and let $C\in\mathcal{C}$ be an open ring. Given a proper generalized chord $Q$ of $\mathbf{P}_C$ with $|E(Q)|\leq N_{\textnormal{mo}}$, we have the following.}
\begin{enumerate} [label=\emph{\arabic*)}]
\itemsep-0.1em
\item\emph{We let $Q_{\textnormal{aug}}$ be the unique cycle of $Q\cup\mathbf{P}_C$ and we let $U_Q$ be the unique connected component of $\Sigma\setminus (C\cup Q)$ such that $\partial(U)=Q_{\textnormal{aug}}$ and $\textnormal{Cl}(U_Q)\cap V(G)\not\subseteq V(C)$}
\item\emph{We say that $Q$ is a \emph{$C$-band} if there is either an element of $\mathcal{C}\setminus\{C\}$ or a noncontractible closed curve of $\Sigma$ in $U_Q$, and we say that $Q$ is a \emph{short} $C$-band if $Q$ is a $C$-band and $|E(Q_{\textnormal{aug}})|\leq N_{\textnormal{mo}}$}
\end{enumerate}

\end{defn}

In the setting above, note that $Q_{\textnormal{aug}}$ is not necessarily a generalized $C$-chord, as it possibly intersects with $C$ on many vertices. Indeed, if the path $C\setminus\mathbf{P}_C$ is sufficiently short, then there is a $C$-band $Q$ with $\mathring{Q}=C\setminus\mathbf{P}_C$. In any case, the open set $U_Q$ from 1) of Definition \ref{QBandOpRingDefn} is uniquely specified. Our main result for Section \ref{BandOpenRingCritMosaicProperSec} is Theorem \ref{QBandMainResCondensed}. To prove Theorem \ref{QBandMainResCondensed}, we first prove the following intermediate result. 

\begin{lemma}\label{QBandMainLemmaForOpRing} Let $\mathcal{T}=(\Sigma, G, \mathcal{C}, L, C_*)$ be a critical mosaic and let $C\in\mathcal{C}$ be an open $\mathcal{T}$-ring. Then, 
\begin{enumerate}[label=\arabic*)]
\itemsep-0.1em
\item  $G$ does not contain a short $C$-band; AND
\item For any proper generalized chord $Q$ of $C$ with $|E(Q)|\leq \frac{N_{\textnormal{mo}}}{3}$, where the endpoints of $Q$ lie in $V(\mathbf{P}_C)$, we have $|E(\mathbf{P}_C)|+|E(Q)|>|E(Q_{\textnormal{aug}}\cap\mathbf{P}_C)|+\frac{2N_{\textnormal{mo}}}{3}$.
\end{enumerate}
 \end{lemma}

\begin{proof} We first prove 1). We begin with the following intermediate result.

\begin{claim}\label{ShortCBandIsNonSepCL} For any short $C$-band $Q$, every vertex of $C$ lies in $\partial(U_Q)$. \end{claim}

\begin{claimproof} Suppose toward a contradiction that there is a $C$-band $Q$ with $V(C)\not\subseteq\partial(U_Q)$. Thus, $Q_{\textnormal{aug}}$ separates either an element of $\mathcal{C}\setminus\{C\}$ or a noncontractible cycle of $G$ from at least one vertex of $C$. We have $\textnormal{Rk}(C)=2N_{\textnormal{mo}}^2$, and since each vertex of $\mathbf{P}_C$ has distance at most $\frac{N_{\textnormal{mo}}}{3}$ from $\mathpzc{w}(C)$, we contradict Corollary \ref{NonRingSeparatingCor}. \end{claimproof}

\begin{claim}\label{ShortCBand2ChordCL} For any short $C$-band $Q$ and any proper generalized chord $P$ of $Q_{\textnormal{aug}}$ with $P\subseteq\textnormal{Cl}(U_Q)$ and $|E(P)|\leq 2$, the endpoints of $P$ are consecutive in $Q$. In particular, $|E(P)|=2$. \end{claim}

\begin{claimproof} Suppose there is a path $P\subseteq\textnormal{Cl}(U_Q)$ which violates Claim \ref{ShortCBand2ChordCL}. By M1), $P$ has at most one endpoint in $\mathbf{P}_C$. If $P$ has one endpoint in $\mathring{\mathbf{P}}_C$ and one endpoint in $Q_{\textnormal{aug}}\setminus\mathbf{P}_C$, then, since $\textnormal{Cl}(U_Q)$ contains either a noncontractible cycle or an element of $\mathcal{C}\setminus\{C\}$, it follows that there exists a short $C$-band $Q'$ such that $V(C)\not\subseteq\partial(U_{Q'})$, contradicting Claim \ref{ShortCBandIsNonSepCL}. Thus, both endpoints of $P$ lie in $Q$. By our face-width conditions, each cycle of $Q_{\textnormal{aug}}\cup P$ is contractible, so let $H_0\cup H_1$ be the natural $(Q_{\textnormal{aug}}, P)$-partition of $G\cap\textnormal{Cl}(U_Q)$, where $\mathbf{P}_C\cap Q_{\textnormal{aug}}\subseteq H_0$. Since $\mathbf{P}_C\cap Q_{\textnormal{aug}}$ has at least one edge, $H_0$ and $H_1$ are uniquely specified. Furthermore, since the endpoints of $P$ are not consecutive in $Q_{\textnormal{aug}}$, we get that $H_0$ contains no noncontractible cycle of $G$ or element of $\mathcal{C}\setminus\{C\}$, or else there is a short $C$-band $Q'$ with $P\subseteq Q'$, where $Q'$ violates Claim \ref{ShortCBandIsNonSepCL}. For each $i=0,1$, let $F_i$ be the cycle $(H_i\cap Q_{\textnormal{aug}})+P$. Since the endpoints of $P$ are not consecutive in $Q_{\textnormal{aug}}$ we have $|E(F_i)|\leq |E(Q)|\leq N_{\textnormal{mo}}$ for each $i=0,1$, and $F_1$ separates at least one vertex of $C$ from either a noncontractible cycle of $G$ or an element of $\mathcal{C}\setminus\{C\}$. It follows from Corollary \ref{NonRingSeparatingCor} that $d(F_1, C\setminus\mathbf{P}_C)>2N_{\textnormal{mo}}-\frac{3|E(F_1)|}{2}$. Now, at least one endpoint of $P$ has distance at most $\frac{(|E(F_0)|-|E(P)|)+|E(\mathbf{P}_C)|}{2}-|E(F_0\cap\mathbf{P}_C)|$ from $V(C\setminus\mathbf{P}_C)$, so we get 
$$\frac{(|E(F_0)|-|E(P)|)+|E(\mathbf{P}_C)|}{2}-|E(F_0\cap\mathbf{P}_C)|+\frac{3|E(F_1)|}{2}>2N_{\textnormal{mo}}$$
Since $|E(F_0)|+|E(F_1)|=|E(F_0\cap\mathbf{P}_C)|+|E(Q)|+2|E(P)|$, we have the following inequality
$$|E(F_1)|+|E(P)|+\frac{|E(Q)|+(|E(\mathbf{P}_C)|-|E(F_0\cap\mathbf{P}_C)|)}{2}>2N_{\textnormal{mo}}$$
Since $|E(\mathbf{P}_C)|\leq\frac{2N_{\textnormal{mo}}}{3}$ and $|E(Q)|\leq |E(Q_{\textnormal{aug}})|\leq N_{\textnormal{mo}}$, the fractional term on the left-hand side of the inequality above is bounded from above by $\frac{5N_{\textnormal{mo}}}{6}$. Since $|E(F_1)|\leq N_{\textnormal{mo}}$, we have $\frac{11N_{\textnormal{mo}}}{6}+2>2N_{\textnormal{mo}}$, which is false. \end{claimproof}

Now we have enough to finish the proof of 1) of Lemma \ref{QBandMainLemmaForOpRing}. Suppose toward a contradiction that $G$ contains a short $C$-band $Q$. By Proposition \ref{LowGenus+OneFaceCase}, $C$ is an induced cycle, so $V(C)$ admits an $L$-coloring $\psi$. Let $\mathcal{C}':=(\mathcal{C}\setminus\{C\})\cup\{Q_{\textnormal{aug}}\}$ and let $G':=G\cap\textnormal{Cl}(U_Q)$. We define a cycle $F'$ of $G'$ as follows. We set $F'=C_*$ if $C_*\neq C$, and otherwise $F'=Q_{\textnormal{aug}}$. Since $C$ is a facial subgraph of $G$, it follows from Claim \ref{ShortCBandIsNonSepCL} that $V(G)\subseteq\textnormal{Cl}(U_Q)$ and $\mathcal{T}'=(\Sigma, G', \mathcal{C}', L^{Q_{\textnormal{aug}}}_{\psi}, F')$ is a tessellation, where $Q_{\textnormal{aug}}$ is a closed $\mathcal{T}'$-ring. We claim now that $\mathcal{T}'$ is a mosaic. Since $|E(Q_{\textnormal{aug}})|\leq N_{\textnormal{mo}}$, M0) is satisfied, and M1) is immediate, since $Q_{\textnormal{aug}}$ is a closed $\mathcal{T}'$-ring. 

It follows from Claim \ref{ShortCBand2ChordCL} that $Q_{\textnormal{aug}}$ is a highly $L^{Q_{\textnormal{aug}}}_{\psi}$-predictable facial subgraph of $G'$, so M2) is satisfied as well. Since $V(Q_{\textnormal{aug}})=V(C)$, we get that each vertex of $Q_{\textnormal{aug}}$ has distance at most $\frac{N_{\textnormal{mo}}}{3}$ from $\mathpzc{w}_{\mathcal{T}}(C)$. Since $\textnormal{Rk}(\mathcal{T}'|Q_{\textnormal{aug}})\leq\textnormal{Rk}(\mathcal{T}|C)-N_{\textnormal{mo}}$, the distance conditions M3)-M4) are still satisfied. Now we just need to check M5). Since edge-width is monotone, we have $\textnormal{ew}(G')\geq\textnormal{ew}(G)\geq 2.1\beta\cdot 6^{g(\Sigma)}$, and since $|E(Q_{\textnormal{aug}})|\leq N_{\textnormal{mo}}$ and $\textnormal{fw}(G)\geq 2.1\beta\cdot 6^{g(\Sigma)-1}$, it follows from our distance conditions, together with Fact \ref{HighEwTriangleFwF1Cycle}, that $\textnormal{fw}(G')\geq 2.1\beta\cdot 6^{g(\Sigma)-1}$. Thus, M5) is satisfied as well, and $\mathcal{T}'$ is indeed a mosaic. Since $V(G)=V(G')$ and $C$ contains at least one vertex with an $L$-list of size at least three, we have $\sum_{v\in V(G')}|L^{Q_{\textnormal{aug}}}_{\psi}(v)|<\sum_{v\in V(G)}|L(v)|$, so it follows from the criticality of $\mathcal{T}$ that $G'$ is $L'$-colorable, and thus $G$ is $L$-colorable, a contradiction. This proves 1).

Now we prove 2) of Lemma \ref{QBandMainLemmaForOpRing}. Suppose toward a contradiction that $|E(\mathbf{P}_C)|+|E(Q)|\leq |E(Q_{\textnormal{aug}}\cap\mathbf{P}_C)|+\frac{2N_{\textnormal{mo}}}{3}$. Let $U$ be the unique connected component of $\Sigma\setminus (C\cup Q)$ with $U\neq U_Q$ and $Q\subseteq\partial(U)$. Since $C$ is an open ring, $\mathbf{P}_C\cap\textnormal{Cl}(U)$ has two connected components. By Theorem \ref{GeneralChordCritMosaicOpMain0}, $\textnormal{Cl}(U)$ contains no noncontractible cycles of $G$ or elements of $\mathcal{C}\setminus\{C\}$. By Proposition \ref{LowGenus+OneFaceCase}, $g(\Sigma)+|\mathcal{C}\setminus\{C\}|>1$, and since $G$ is a 2-cell embedding, it follows that that $\textnormal{Cl}(U_Q)$ contains either an element of $\mathcal{C}\setminus\{C\}$ or a noncontractible cycle, contradicting 1). \end{proof}

With Lemma \ref{QBandMainLemmaForOpRing} in hand, we prove Theorem \ref{QBandMainResCondensed}.

\begin{theorem}\label{QBandMainResCondensed} Let $\mathcal{T}=(\Sigma, G, \mathcal{C}, L, C_*)$ be a critical mosaic and $C\in\mathcal{C}$ be an open ring. Then $|E(\mathbf{P}_C)|=\frac{2N_{\textnormal{mo}}}{3}$ and, for any $x,y\in V(\mathbf{P}_C)$, there is no $(x,y)$-path in $G$ of length at most $\min\left\{|E(x\mathbf{P}_Cy)|, \frac{N_{\textnormal{mo}}}{3}\right\}$, except possibly $x\mathbf{P}_Cy$. \end{theorem}

\begin{proof} Suppose that $|E(\mathbf{P}_C)|\neq\frac{2N_{\textnormal{mo}}}{3}$. Since $|E(\mathbf{P}_C)|\leq\frac{2N_{\textnormal{mo}}}{3}$ and $N_{\textnormal{mo}}$ is a multiple of three, we have $|E(\mathbf{P}_C)|<\frac{2N_{\textnormal{mo}}}{3}$. By Proposition \ref{LowGenus+OneFaceCase}, $C$ is induced, so, by our triangulation conditions, there is a 2-chord of $C$ whose endpoints are consecutive vertices of $\mathbf{P}_C$. Applying 2) of Lemma \ref{QBandMainLemmaForOpRing} to this 2-chord of $C$, we have $2+(|E(\mathbf{P}_C)|-1)>\frac{2N_{\textnormal{mo}}}{3}$, contradicting our assumption that $|E(\mathbf{P}_C)|<\frac{2N_{\textnormal{mo}}}{3}$. Thus, $|E(\mathbf{P}_C)|=\frac{2N_{\textnormal{mo}}}{3}$. Now let $x,y\in V(\mathbf{P}_C)$, and suppose there is a path $P$, distinct from $x\mathbf{P}_Cy$, such that $|E(P)|\leq\min\left\{|E(x\mathbf{P}_Cy)|, \frac{N_{\textnormal{mo}}}{3}\right\}$. Since $|E(\mathbf{P}_C)|\leq\frac{2N_{\textnormal{mo}}}{3}$, it follows that there is either a short $C$-band, contradicting 1) of Lemma \ref{QBandMainLemmaForOpRing}, or there is a subpath of $P$ which is a proper generalized chord of $C$ violating Theorem \ref{GeneralChordCritMosaicOpMain0}. In any case, we have a contradiction. \end{proof}

Bringing together the result above with the results of Section \ref{GenChordClosedRingCritSec} and \ref{GenChordOpenRingCritSec2}, we obtain the following.

\begin{theorem}\label{NOver4GoodSideCorCritMos}
Let $\mathcal{T}=(\Sigma, G, \mathcal{C}, L, C_*)$, let $C\in\mathcal{C}$, and let $Q$ be a generalized chord of $Q$ with $|E(Q)|\leq \frac{N_{\textnormal{mo}}}{3}$. Let $G=G_0\cup G_1$ be the natural $Q$-partition of $G$. Then there exists a unique $j\in\{0,1\}$ such that
\begin{enumerate}[label=\Alph*)]
\itemsep-0.1em
\item Each $C'\in\mathcal{C}\setminus\{C\}$ lies in $G_j$ and the unique open component of $\Sigma\setminus (C\cup Q)$ whose closure contains $G_{1-j}$ is a disc; AND
\item If $C$ is an open ring, then the following hold.
\begin{enumerate}[label=\roman*)]
\itemsep-0.1em
\item $|E(Q)|+|E(\mathbf{P}_C\cap G_j)|>\frac{2N_{\textnormal{mo}}}{3}$. In particular, if $Q$ is disjoint to $\mathring{\mathbf{P}}_C$, then $\mathbf{P}_C\subseteq G_j$; AND
\item If $Q$ is a proper generalized chord of $C$ and both endpoints of $Q$ lie in $\mathbf{P}_C$, then $\mathbf{P}_C\cap G_{1-j}$ has one connected component and $\mathbf{P}_C\cap G_j$ has two connected components.
\end{enumerate} 
\end{enumerate}\end{theorem}

\begin{proof} Given a $j\in\{0,1\}$, we call $j$ \emph{ideal} if it satisfies both of A) and B). Let $U$ be the unique connected component of $\Sigma\setminus C$ such that $Q\not\subseteq\textnormal{Cl}(U)$, i.e the unique component of $\Sigma\setminus C$ containing no vertices of $G$. Now, $U$ is an open disc, and, by Proposition \ref{LowGenus+OneFaceCase}, $g(\Sigma)+|\mathcal{C}\setminus\{C\}|>0$. It follows that at most one $j\in\{0,1\}$ can satisfy A), so we just need to show that there exists an ideal $j\in\{0,1\}$. If $C$ is a closed ring, it it follows from Theorem \ref{ShortGenChordClosedRingThm} that there is an ideal $j\in\{0,1\}$. Now suppose that $C$ is an open ring. Consider the following cases. 

\textbf{Case 1:} $Q$ is an improper generalized chord of $C$ 

In this case, we let $j\in\{0,1\}$ be the unique index such that $C\subseteq G_j$. Thus, $\mathbf{P}_C\subseteq G_j$ and $V(\mathbf{P}_C\cap G_j)$ consists of at most one vertex. By Theorem \ref{QBandMainResCondensed}, $|E(\mathbf{P}_C)|=\frac{2N_{\textnormal{mo}}}{3}$, so B i) is satisfied. We just need to check A). If $C\neq C_*$, then it follows from Corollary ref{NonRingSeparatingCor} that $j$ is ideal, so suppose that $C=C_*$. In that case, we have $G_j=\textnormal{Ext}(Q)$, and since $\textnormal{Rk}(C)=2N_{\textnormal{mo}}^2$, it follows from a) of Corollary \ref{NonRingSeparatingCor} that $j$ is ideal. 

\textbf{Case 2:} $Q$ is a proper generalized chord of $C$

In this case, if neither endpoint of $Q$ lies in $\mathring{\mathbf{P}}_C$, then, by Theorem \ref{GeneralChordCritMosaicOpMain0}, there is an ideal $j\in\{0,1\}$. Now suppose precisely one endpoint of $Q$ lies in $\mathbf{P}_C$. As $|E(Q)|\leq\frac{N_{\textnormal{mo}}}{3}$ and $|E(\mathbf{P}_C)|\leq\frac{2N_{\textnormal{mo}}}{3}$, there is a $k\in\{0,1\}$ such that $|E(Q)|+|E(\mathbf{P}_C\cap G_k)|\leq\frac{2N_{\textnormal{mo}}}{3}$, so it follows from Theorem \ref{GeneralChordCritMosaicOpMain0} that every element of $\mathcal{C}\setminus\{C\}$ lies in $G_{1-k}$ and the unique open component of $\Sigma\setminus (C\cup Q)$ whose closure contains $G_k$ is a disc. As $g(\Sigma)+|\mathcal{C}\setminus\{C\}|>0$, it also follows from Theorem \ref{GeneralChordCritMosaicOpMain0} that $|E(Q)|+|E(\mathbf{P}_C\cap G_{1-k})|>\frac{2N_{\textnormal{mo}}}{3}$, so $1-k$ is ideal. Now suppose both endpoints of $Q$ lie in $\mathring{\mathbf{P}}_C$. Let $j\in\{0,1\}$ be the unique index such that $G_{1-j}\cap\mathbf{P}_C=Q_{\textnormal{aug}}\cap\mathbf{P}_C$. By 1) of Lemma \ref{QBandMainLemmaForOpRing}, $j$ is ideal. \end{proof}

\section{Rainbows and Consistent Paths}\label{RainbowsandConsistPathSec}

In the language of Section \ref{BlackBoxPIISec}, Theorem \ref{NOver4GoodSideCorCritMos} tells us that, given a critical mosaic $\mathcal{T}=(\Sigma, G, \mathcal{C}, L, C_*)$, each ring $C\in\mathcal{C}$ is uniquely $N_{\textnormal{mo}}/3$-determined in $G$ with respect to the list-assignment $L$. To complete the proof of Theorem \ref{AllMosaicsColIntermRes1-4}, we want to apply Theorem \ref{FaceConnectionMainResult} to produce a smaller mosaic from $\mathcal{T}$, but Theorem \ref{FaceConnectionMainResult} applies to facial cycles which all have lists of size at least three, and the elements of $\mathcal{C}$ have some precolored vertices. In Sections \ref{RedForOpRings}-\ref{RedNearRingsSec}, we show how we can perform a ``reduction" operation near each of the rings of $\mathcal{T}$, to give us a new tessellation $\mathcal{T}'$, whose underlying graph is a subgraph of $G$, where Theorem \ref{FaceConnectionMainResult} can be applied to $\mathcal{T}'$. We make this precise in Definition \ref{ReductionOpDefn}. The two results we need from this section to perform the reduction operations in Sections \ref{RedForOpRings}-\ref{RedNearRingsSec} are Theorems \ref{MainLinkingResultAlongFacialCycle} and \ref{LinkPlusOneMoreVertx}. To prove these, we first need some machinery and facts about planar graphs. 

\begin{defn}\label{GUniversalDefinition}
\emph{Let $H$ be a graph and $q\in V(H)$. We call $H$ a \emph{broken wheel with principal vertex $q$} if $H-q$ is a path of length at least two whose vertices are all adjacent to $q$. Given a length-two path $P:= pqp'\subseteq H$, a list-assignment $L$ for $V(H)$, and an $a\in L(p)$, we say that $a$ is \emph{$(P, H)$-universal} (with respect to $L$) if, for each $b\in L(q)\setminus\{a\}$ and each $c\in L(p)\setminus\{b\}$, there is an $L$-coloring of $\{p,q, p'\}$ which extends to $L$-color $H$ and uses $a,b,c$ on $p, q, p'$ respectively.}
 \end{defn}

\begin{defn}\label{RainbowDefnX} \emph{A \emph{rainbow} is a tuple $(G, C, P, L)$, where $G$ is a planar graph with outer cycle $C$, $P$ is a path on $C$ with $|E(P)|\geq 1$, and $L$ is a list-assignment for $V(G)$ such that each endpoint of $P$ has a nonempty list and furthermore, $|L(v)|\geq 3$ for each $v\in V(C\setminus P)$ and $|L(v)|\geq 5$ for each $v\in V(G\setminus C)$. }
 \end{defn}

\begin{defn}\label{GeneralAugCrownNotForLink}  \emph{Let $G$ be a graph with list-assignment $L$, and $P\subseteq G$ be a path with endpoints $p, p'$. We let $\textnormal{End}_L(P, G)$ denote the set of $L$-colorings $\phi$ of $\{p, p'\}$ such that any extension of $\phi$ to an $L$-coloring of $V(P)$ extends to $L$-color $G$.}
 \end{defn}

We usually drop the subscript $L$ or the coordinate $G$ from the notation above (or both) if these are clear from the context. The following is proven in Section 3 of \cite{JNevinHolepunchI}. 

\begin{theorem}\label{SumTo4For2PathColorEnds} Let $(G, C, P, L)$ be a rainbow, where $P:=p_0qp_1$ and $L(p_0)+L(p_1)\geq 4$. Then  $\textnormal{End}(P,G)\neq\varnothing$. Furthermore, either $|\textnormal{End}(P, G)|\geq 2$ or there is an even-length path $Q$ with $V(Q)\subseteq V(C)$, where $Q$ has endpoints $p_0, p_1$ and each vertex of $Q$ is adjacent to $q$.   \end{theorem}

Using Theorem \ref{SumTo4For2PathColorEnds}, the following is straightforward to check:

\begin{prop}\label{ForSplToEachWalk} Let $(G, C, P, L)$ be a rainbow, where $P=p_0qp_1$ and $|L(p_i)|\geq 2$ for each $i=0,1$. Then either 
\begin{enumerate}[label=\arabic*)]
\itemsep-0.1em
\item There is an $i\in\{0,1\}$ such that there is a $(P, G)$-universal color of $L(p_i)$; OR
\item There is a path $Q$ with $V(Q)\subseteq V(C-q)$, where $Q$ has endpoints $p_0, p_1$ and every vertex of $Q$ is adjacent to $q$, and furthemore, either $Q$ is an edge with $L(p_0)=L(p_1)$, or every internal vertex of $Q$ has a list of size three which contain $L(p_0)\cup L(p_1)$.
\end{enumerate}
\end{prop}

The result below about 3-paths is obtained from T4) of Theorem 4.5 of \cite{JNevinHolepunchI}. The statement below is weaker than the corresponding statement in \cite{JNevinHolepunchI}, but it is sufficient for our purposes.

\begin{theorem}\label{ObtFromT4Split}  Let $(G, C, P, L)$ be a rainbow, where $P:=p_0q_0q_1p_1$ and at least endpoint of $P$ has an $L$-list of size at least three. Then there is an $L$-coloring $\psi$ of $\{p_0, p_1\}$ such that
\begin{enumerate}[label=\arabic*)]
\itemsep-0.1em
\item All but at most two extensions of $\psi$ to $L$-colorings of $V(P)$ extend to $L$-color $G$; AND
\item Either $q_0, q_1$ have a common neighbor in $C\setminus P$, or all but at most one extension of $\psi$ to $L$-colorings of $V(P)$ extends to $L$-color $G$;
\end{enumerate}
  \end{theorem}

To state the main results we prove in Section \ref{RainbowsandConsistPathSec}, we first define the following. 

\begin{defn}\label{uniquekDetLocalPVers} \emph{Let $G$ be an embedding on a surface $\Sigma$ and $L$ be a list-assignment for $V(G)$. Let $k\geq 2$ be an integer and $C$ be a uniquely $k$-determined contractible facial cycle of $G$. For any subpath $P$ of $C$,}
\begin{enumerate}[label=\emph{\arabic*)}]
\itemsep-0.1em
\item\emph{we associate to $P$ a vertex set $\textnormal{Sh}^k(P)$, where $v\in\textnormal{Sh}^k(P)$ if there is a proper generalized chord $Q$ of $C$ of length at most $k$, such that both endpoints of $Q$ lie in $P$, and $v\in V(G^{\textnormal{small}}_Q\setminus Q)$, and furthermore, $G^{\textnormal{small}}_Q\cap P$ has one connected component and $G^{\textnormal{\large}}_Q$ has two connected components.}
\item \emph{We say $P$ is \emph{$k$-consistent} if both of the following hold.}
\begin{enumerate}[label=\emph{\alph*)}]
\itemsep-0.1em
\item\emph{There is no chord of $C$ with one endpoint in $P$ and the other endpoint in $C\setminus P$}; AND
\item \emph{For any proper generalized chord $Q$ of $C$ of length at most $k$, if $Q$ has at both endpoints in $P$ and at least one endpoint in $\mathring{P}$, then $G^{\textnormal{small}}_Q\cap P$ has one connected component and $G^{\textnormal{large}}_Q\cap P$ has two connected components.}
\end{enumerate}
\item\emph{We let $\textnormal{Link}_L(P)$ denote the set of $L$-colorings $\phi$ of $V(P)\setminus\textnormal{Sh}^2(P)$ such that $\textnormal{Sh}^2(P)$ is $(L, \phi)$-inert in $G$. We drop the subscript if $L$ is clear from the context. Let $p,p'$ denote the endpoints of $P$. Given a $c\in L(p)$ and an $A\subseteq L(p')$, we say $c$ is \emph{$(P, p, A)$-linking} if, for any $d\in A$, there is an element of $\textnormal{Link}(P)$ using $c,d$ on $p,p'$ respectively.} 
\end{enumerate}
\end{defn}

We now note the following, which is straightforward to check. In particular, in the statement below, L1) is obtained by repeated applications of Theorem \ref{SumTo4For2PathColorEnds} and L2) is straightforward to check by induction.
 
\begin{theorem}\label{MainLinkingResultAlongFacialCycle}  Let $\Sigma, G, C, L, k$ be as in Definition \ref{uniquekDetLocalPVers}, where $k\geq 2$ and $C$ is induced. Let $P$ be a 2-consistent subpath of $C$ with endpoints $p_0, p_1$, where $V(C\setminus P)\neq\varnothing$ and each internal vertex of $P$ has an $L$-list of size at least three. For each $i=0,1$, let $A^i\subseteq L(p_i)$, where $A^0, A^1$ are nonempty. Then the following hold.
\begin{enumerate}[label=\arabic*)]
\itemsep-0.1em
\item [\mylabel{}{\textnormal{L1)}}] If $|A^0|+|A^1|\geq 4$, there is a $\phi\in\textnormal{Link}(P)$ with $\phi(p)\in A^0$ and $\phi(p)\in A^1$; AND
\item [\mylabel{}{\textnormal{L2)}}] If $|A^0|\geq 2$ and $|A^1|\geq 2$ and, for each $x\in D_1(C)$, the graph $G[N(x)\cap V(C)]$ is a path of length at most two, then either
\begin{enumerate}[label=\alph*)]
\itemsep-0.1em
\item   there is an $i\in\{0,1\}$ and a $c\in L(p_i)$ which is $(P, p_i, A^{1-i})$-linking; OR
\item Each vertex of $\mathring{P}$ has a list of size precisely three and there is a color of $A^0\cap A^1$ common to the lists of all the vertices of $P$. Furthermore, either $|A^0|=|A^1|=2$, or there are at least two colors of $A^0\cap A^1$ common to the lists of all the vertices of $P$. 
\end{enumerate}
\end{enumerate}
  \end{theorem}

In the setting above, an arbitrary precoloring of the endpoints of $P$ does not necessarily extend to an element of $\textnormal{Link}(P)$, but, under some conditions, if we strengthen the $2$-consistency condition to $3$-consistency, we can guarantee that something analogous holds if we are allowed to color one more vertex. We actually need something slightly stronger, which is the following somewhat technical result in Theorem \ref{LinkPlusOneMoreVertx}. 

\begin{defn} \emph{Let $\Sigma, G, C, L, k$ be as in Definition \ref{uniquekDetLocalPVers}, where $k\geq 2$. A \emph{$P$-peak} is a vertex $v\in D_1(C)$ such that $|N(v)\cap V(P)|\geq 2$, where $v\not\in\textnormal{Sh}^2(P)$. We say $v$ is an \emph{internal} $P$-peak if it is not adjacent to any endpoint of $P$.} \end{defn}

Informally, the $P$-peaks are the midpoints of the ``maximal" 2-chords of $C$ with both endpoints in $P$. The result below is proven in Section 5 of \cite{JNevinThesisManyFacePartIPaperRevised}. 

\begin{theorem}\label{LinkPlusOneMoreVertx} Let $\Sigma, G, C, L, k$ be as in Definition \ref{uniquekDetLocalPVers}, where $k\geq 3$, and $G$ is short-inseparable. Suppose further that, for each $v\in B_2(C)$, every facial subgraph of $G$ containing $v$, except possibly $C$, is a triangle, and either $v\in V(C)$ or $|L(v)|\geq 5$. Let $P$ be a $3$-consistent subpath of $C$ with endpoints $p_0, p_1$, where each internal vertex of $P$ has a list of size at least three. Let $\phi$ be an $L$-coloring of $\{p_0, p_1\}$. Then, for all but at most three internal $P$-peaks $w$, there is a vertex-set $T$ with $V(P)\subseteq T\subseteq V(P+w)\cup\textnormal{Sh}^3(P)$, where $\phi$ extends to a partial $L$-coloring $\tau$ of $T$ such that $T$ is $(L, \tau)$-inert in $G$ and every vertex of $D_1(\textnormal{dom}(\tau))\setminus T$ has an $L_{\tau}$-list of size at least three. 
\end{theorem}

\section{Reductions Near the Rings: Open Rings}\label{RedForOpRings}

\begin{defn}\label{ReductionOpDefn} \emph{Let $\mathcal{T}=(\Sigma, G, \mathcal{C}, L, C_*)$ be a critical mosaic, let $C\in\mathcal{C}$ and let $\pi$ be the unique $L$-coloring of $V(\mathbf{P}_C)$. A \emph{$C$-reduction} is a triple $(S, \psi, C^r)$, where $S$ is a vertex set with $V(C)\subseteq S$, $\psi$ is an extension of $\pi$ to a partial $L$-coloring of of $S$, and $C^r$ is a cycle such that}
\begin{enumerate}[label=\emph{\arabic*)}]
\itemsep-0.1em
\item\emph{$S\subseteq B_2(C)\cup\textnormal{Sh}^4(C)$ and, if $C$ is closed, then $S$ satisfies the stronger condition $S\subseteq B_1(C)$}; AND 
\item\emph{$S$ is $(L, \psi)$-inert in $G$ and $G[V(C)\cup S]$ is connected}; AND
\item\emph{Every vertex of $D_1(S)$ has an $L_{\phi}$-list of size at least three. Furthermore, $G\setminus S$ is 2-connected and $C^r$ is the unique facial cycle of $G\setminus S$ whose vertex set is precisely $D_1(S)$.}
\end{enumerate}
\emph{We call $C^r$ a \emph{$C$-reduction cycle.} and $\psi$ a \emph{$C$-reduction coloring}}
\end{defn}

Over Sections \ref{RedForOpRings}-\ref{RedNearRingsSec}, we prove that, for each ring $C$ of a critical mosaic, there is a $C$-reduction. We first deal with open rings in Theorem \ref{MainCollarOpRingRedClF} below. This is somewhat technical due to the possible presence of 2-chords of $C$ with one endpoint in $C\setminus\mathbf{P}$ and the other endpoint in $\mathring{\mathbf{P}}$, where $\mathbf{P}$ is the precolored path of $C$. If there were no such 2-chords of $C$, then we would obtain Theorem \ref{MainCollarOpRingRedClF} immediately by applying Theorem \ref{LinkPlusOneMoreVertx} to the path $C\setminus\mathring{\mathbf{P}}$

\begin{theorem}\label{MainCollarOpRingRedClF} Let $\mathcal{T}=(\Sigma, G, \mathcal{C}, L, C_*)$ be a critical mosaic and $C\in\mathcal{C}$ be an open ring. Then there is a $C$-reduction. \end{theorem}

\begin{proof} We let $C=p_1\cdots p_mu_n\cdots u_1$, where $\mathbf{P}=p_1\cdots p_m$. By Theorem \ref{QBandMainResCondensed}, $n\geq  N_{\textnormal{mo}}/3-1$ and $m=\frac{2N_{\textnormal{mo}}}{3}-1$. Let $\pi$ be the unique $L$-coloring of $V(\mathbf{P})$. To construct a $C$-reduction, we need to consider 3-chords of $C$ which have one endpoint in $\mathbf{P}$ and the other endpoint in $C\setminus\mathbf{P}$. Note that, by Theorem \ref{NOver4GoodSideCorCritMos}, we immediately have the following.

\begin{claim} Let $R$ be a 3-chord of $C$ with one endpoint in $\mathbf{P}$ and the other endpoint in $C\setminus\mathbf{P}$. Then either $G^{\textnormal{small}}_R\cap\mathbf{P}$ is a subpath of $p_1p_2p_3$ with $p_1$ as an endpoint, or it is a subpath of $p_{m-2}p_{m-1}p_m$ with $p_m$ as an endpoint. \end{claim}

Now we let $\mathcal{R}^-$ be the set of 3-chords $R=xyy'x'$ of $C$ such that $p_1$ is an endpoint of $G^{\textnormal{small}}_R\cap\mathbf{P}$, where $x\in V(\mathbf{P})$ and $x'\in V(C\setminus\mathbf{P})$, and $y'$ has no neighbors in $\mathbf{P}$. We define the set $\mathcal{R}^+$ analogously, where $p_1$ is replaced by $p_m$.  Let $\pi$ be the unique $L$-coloring of $V(\mathbf{P})$.. Given an $R\in\mathcal{R}^-\cup\mathcal{R}^+$, where $R=xyy'x'$ for $x\in V(\mathbf{P})$ and $x'\in V(C\setminus\mathbf{P})$, we define an $R$-\emph{incision} to be an extension of $\pi$ to a partial $L$-coloring $\phi$ of $V(\mathbf{P})\cup V(G-y')$ such that $V(R-y)\subseteq\textnormal{dom}(\phi)$ and $|L_{\phi}(y')|\geq 3$, where $V(G^{\textnormal{small}}_R-y')$ is $(L, \phi)$-inert in $G$. 

Recall that, by Theorem \ref{QBandMainResCondensed}, for each $y\in D_1(C)$, if $y$ has a neighbor in $\mathbf{P}$, then $G[N(y)\cap\mathbf{P}]$ is a path of length at most one.  By Proposition \ref{LowGenus+OneFaceCase}, $C$ is induced, so it follows from our triangulation conditions that, for each $e\in E(\mathbf{P})$, there is a unique common neighbor in $D_1(C)$ to the endpoints of $e$ in $D_1(C)$.  

\begin{claim}\label{RinMinusPlusWalk} Let $R\in\mathcal{R}^-\cup\mathcal{R}^+$. Then there is an $R$-incision.\end{claim}

\begin{claimproof} Say for the sake of definiteness that $R\in\mathcal{R}^-$.  Suppose the claim does not hold. Let $v$ be the unique neighbor of $p_1, p_2$ in $D_1(C)$ and $v^*$ be the unique neighbor of $p_2, p_3$ in $C\setminus\mathbf{P}$. Note that $v\neq v^*$, and, by Theorem \ref{NOver4GoodSideCorCritMos}, since $p_3\in N(v^*)$, we have $N(v^*)\cap V(C)=\{p_2, p_3\}$. 

\vspace*{-8mm}
\begin{addmargin}[2em]{0em}
\begin{subclaim}\label{yNotAdjP1SubH} $y\not\in N(p_1)$ and $y'\not\in N(v)$. \end{subclaim}

\begin{claimproof} Suppose either $y\in N(p_1)$ or $y'\in N(v)$. Thus, $G$ contains the 3-chord $p_1zy'x'$ of $C$, where $z\in\{y, v\}$, and, letting $H=G^{\textnormal{small}}_{p_1yzy'x'}$, we get either $H=G^{\textnormal{small}}_R$ or $H=G^{\textnormal{small}}_R-p_2$. As $|L_{\pi}(y)|\geq 3$ and $y'$ has no neighbors in $\mathbf{P}$, Theorem \ref{ObtFromT4Split} applied to $H$ implies that $\pi$ extends to an $R$-incision, contradicting our assumption. \end{claimproof}\end{addmargin}

 By Subclaim \ref{yNotAdjP1SubH}, $v\in V(G^{\textnormal{small}}_{R}\setminus R)$.

\vspace*{-8mm}
\begin{addmargin}[2em]{0em}
\begin{subclaim}\label{yNotAdjU1FinCase} $v$ is not adjacent $u_1$. \end{subclaim}

\begin{claimproof} Suppose that $vu_1\in E(G)$. Consider the following cases.

\textbf{Case 1:} $v\in N(y)$

In this case, $\{p_2\}\subseteq N(y)\cap\mathbf{P}\subseteq\{p_2, p_3\}$, and $G$ contains the 4-chord $u_1vyy'x'$ of $C$. We have $u_1\neq x'$, or else our triangulation conditions imply that $v\in N(y')$, contradicting Subclaim \ref{yNotAdjP1SubH}. Thus, $u_1vyy'x'$ is a proper 4-chord of $C$, and, letting $H:=G^{\textnormal{small}}_{u_1vvyy'x'}$, we have $H=G^{\textnormal{small}}_R\setminus\{p_1, p_2, p_3\}$. Now, there exists a set $S\subseteq L_{\pi}(u_1)$ with $|S|=2$. Since $|L_{\pi}(v)|\geq 3$, there is an $r\in L_{\pi}(v)\setminus S$. Let $L'$ be a list-assignment for $V(H)$, where $L'(v)=\{r\}$ and $L'(u_1)=S\cup\{r\}$, and otherwise $L'=L_{\pi}$. We regard $H$ as a planar embedding with outer cycle $F:=u_1vyy'x'+u_1(C\setminus\mathbf{P})x'$, where this cycle contains the 3-path $P:=vyy'x'$. Since $v^*$ has no neighbors in $F\setminus\{v, y'\}$ and $|L'(v')|\geq 3$, it follows from Theorem \ref{ObtFromT4Split} that there is an $L'$-coloring $\psi$ of $\{v, y, x'\}$ such that any extension of $\psi$ to an $L$-coloring of $V(P)$ extends to $L'$-color all of $H$. Our choice of $L'$ implies that $\pi\cup\psi$ is a proper $L$-coloring of its domain, and $V(G^{\textnormal{small}}_R-y')$ is $(L, \pi\cup\psi)$-inert in $G$. Furthermore, since $v\not\in N(y')$, we have $|L_{\pi\cup\psi}(y')|\geq 3$, so $\pi\cup\psi$ is an $R$-incision, contradicting our assumption. 

\textbf{Case 2:} $v\not\in N(y)$ and $vv^*\not\in E(G)$

In this case, we construct a smaller mosaic with underlying graph $G^{\dagger}:=G-p_1$. Let $C^{\dagger}:=(C-p_1)+u_1vp_2$. Now, there is a set $S\subseteq L_{\pi}(u_1)$ with $|S|=2$, and, since $|L_{\pi}(v)|\geq 3$, there is a $c\in L_{\pi}(v)\setminus S$. Let $L^{\dagger}$ be a list-assignment for $V(G^{\dagger})$ where $L^{\dagger}(v)=\{c\}$ and $L^{\dagger}(u_1)=S\cup\{c\}$, and otherwise $L^{\dagger}=L$. Let $C^{\dagger}_*$ be an element of $(\mathcal{C}\setminus\{C\})\cup\{C^{\dagger}\}$, where $C^{\dagger}_*:=C^{\dagger}$ if $C=C_*$, and otherwise $C^{\dagger}_*:=C_*$. We claim now that $\mathcal{T}^{\dagger}=(\Sigma, G^{\dagger}, (\mathcal{C}\setminus\{C\})\cup\{C^{\dagger}\}, L^{\dagger}, C^{\dagger}_*)$ is a mosaic, where $C^{\dagger}$ is an open ring with precolored path $vp_2p_3\cdots p_m$. M0) and M2) are immediate, and, since $v\in V(G^{\textnormal{small}}_R\setminus R)$, we get that $\mathcal{T}^{\dagger}$ still satisfies all the distance conditions, edge-width conditions, and face-width conditions of M3)-M5). We just need to check M1). Let $\pi'$ be the unique $L^{\dagger}$-coloring of $vp_2\cdots p_m$. If M1) is volated, then there is a vertex of $D_1(C)$ adjacent to all three of $v, p_2, p_3$, so $vv^*\in E(G)$, contradicting the assumption of Case 2. Thus, $\mathcal{T}^{\dagger}$ is a mosaic. SInce $\mathcal{T}$ is critical, $G-p_1$ is $L^{\dagger}$-colorable, so $G$ is $L$-colorable, which is false. 

\textbf{Case 3:} $v\not\in N(y)$ and $vv^*\in E(G)$.

In this case, we have $v^*\in V(G^{\textnormal{small}}\setminus R)$ and $N(y)\cap V(\mathbf{P})=\{p_3\}$. Thus, by Theorem \ref{NOver4GoodSideCorCritMos}, $N(y)\cap V(C)=\{p_3\}$. Now we follow an identical argument to that of Case 2, except that we precolor $vv^*$, rather than just $v$, to construct a smaller mosaic on underlying graph $G\setminus\{p_1, p_2\}$, where the ring of this new mosaic obtained from $C$ has precolored path $vv^*p_3\cdots p_m$, and we again produce a contradiction. \end{claimproof}\end{addmargin}

We now have the following. 

\vspace*{-8mm}
\begin{addmargin}[2em]{0em}

\begin{subclaim}\label{y'NotNeighU1} $y'\not\in N(u_1)$. \end{subclaim}

\begin{claimproof} Suppose that $y'\in N(u_1)$. Let $F$ be a cycle, where $F:=p_1p_2yy'u_1$ if $y\in N(p_2)$, and otherwise $F:=p_1p_2p_3yy'u_1$. Note that $F$ is an induced cycle of length either five or six. Let $G=G'\cup G''$ be the natural $F$-partition of $G$, where $v\in V(G')$. Since $u_1, y', p_3\not\in N(v)$, our triangulation conditions imply that $N(v)\cap V(F)=\{p_1, p_2\}$, so it follows from Theorem \ref{BohmePaper5CycleCorList} that any $L$-coloring of $V(F)$ extends to $L$-color $G'$. This proves Subclaim \ref{y'NotNeighU1}. We define a graph $H$ as follows. Let $H:=y'u_1$ if $x'=u_1$, and otherwise $H:=G^{\textnormal{small}}_{u_1y'x'}$. Given an $L_{\pi}$-coloring $\psi$ of $\{u_1, x'\}$, we call $\psi$ \emph{sufficient} if any extension of $\psi$ to an $L$-coloring of $\{u_1, y', x'\}$ extends to $L_{\pi}$-color $H$. Note that $u_1, x'\not\in N(y)$. Since $F$ is induced and any $L$-coloring of $V(F)$ extends to $L$-color $G'$, our assumption that there is no $R$-incision implies the following
\begin{equation}\label{TagForEqSuff}\tag{Eq1}\textnormal{For any sufficient $\psi$ and any $c\in L_{\pi}(y)$, we have $L(y')\setminus\{c, \psi(u_1), \psi(x')\}|=2$} \end{equation}
Thus, $H$ is not an edge, i.e $u_1\neq x'$, and furthermore, for each $u\in\{u_1, x'\}$ and each $d\in L_{\pi}(u)$, there is at most one element of $\textnormal{End}(u_1y'x', G^{\textnormal{small}}_{u_1y'x'})$ using $d$ on $u$. Since $|L_{\pi}(u_1)|\geq 2$ and $|L_{\pi}(x')|\geq 3$,  Theorem \ref{SumTo4For2PathColorEnds} implies that $H$ is a broken wheel with principal vertex $y'$, where $|V(H)|$ is even, and it follows from Proposition \ref{ForSplToEachWalk} that there is an element of $\textnormal{End}(u_1y'x', G^{\textnormal{small}}_{u_1y'x'})$ using the same color on $u_1, x'$, contradicting (\ref{TagForEqSuff}).  \end{claimproof}\end{addmargin}

Let $w$ be the unique neighbor of $p_1, u_1$ in $D_1(C)$. By Subclaim \ref{yNotAdjU1FinCase}, $w\neq v$. Recall that, by Theorem \ref{NOver4GoodSideCorCritMos}, $N(w)\cap V(\mathbf{P})\subseteq\{p_1, p_2\}$. Since $G$ is short-inseparable and $w\neq v$, we have $N(w)\cap V(\mathbf{P})=\{p_1\}$. Now we construct a smaller counterexample using a similar trick to that of \cite{AllPlanar5ThomPap}. Our triangulation conditions imply that there is a unique path $P$ with endpoints $w, u_2$, where $V(\mathring{P})\subseteq D_1(C)$ and $P$ consists of all the vertices of $N(u_1)\setminus\{p_1\}$. The fact that $G$ is short-inseparable, and $y'\not\in N(u_1)$, immediately implies that no vertex of $P-u_2$ lies in $V(R)\cup\{v\}$. We now construct a smaller mosaic with underlying graph $K:=G-u_1$. Let $C^K:=(C-u_1)+p_1wPu_2$. Then $C_K$ is a facial cycle of $K$. There is a set $S\subseteq L_{\pi}(u_1)$ with $|S|=2$. Let $L^K$ be a list-assignment for $V(K)$, where $L^K(z)=L(z)\setminus S$ for each $z\in V(P-u_2)$, and otherwise $L^K=L$. Define an element $C^K_*$ of $(\mathcal{C}\setminus\{C\})\cup\{C_K\}$, where $C^K_*=C^K$, and otherwise $C_K^*:=C^K$.. Now, we claim that $\mathcal{T}^K=(\Sigma, K, (\mathcal{C}\setminus\{C\})\cup\{C^K\}, L^K, C^K_*)$ is a mosaic, where $C^K$ is an open ring with precolored path $\mathbf{P}$. It is immediate that M0) and M2) are still satisfied, and, since $V(P-u_2)\subseteq V(G^{\textnormal{small}}_R\setminus R)$, we get that $\mathcal{T}^K$ still satisfies all the distance conditions, face-width conditions, and edge-width conditions, of M3)-M5). We just need to check M1). There is no chord of $C^K$ with an endpoint in $\mathring{\mathbf{P}}$, and the rest of M1) is inherited from $\mathcal{T}$, so $\mathcal{T}^K$ is a mosaic. Since $|V(K)|=|V(G)|-1$, we get thet $G-u_1$ is $L^K$-colorable, so $G$ is $L$-colorable, which is false. This proves Claim \ref{RinMinusPlusWalk}. \end{claimproof}

Our triangulation conditions imply that each of $\mathcal{R}^-$ and $\mathcal{R}^+$ is nonempty. For any $R, R'\in\mathcal{R}^-$, we have either $G^{\textnormal{small}}_R\subseteq G^{\textnormal{small}}_{R'}$, so we let $Q^-$ be the unique element of $\mathcal{R}^-$ such that $G^{\textnormal{small}}_{Q^-}$ contains all the elements of $\mathcal{R}^-$. Now, there is a unique path $\mathbf{P}^1$ consisting of all the vertices of $G\setminus C$ with a neighbor in $\mathbf{P}$, and Theorem \ref{QBandMainResCondensed} implies that $\mathbf{P}^1$ is an induced path. Let $Q^-\setminus\mathbf{P}=w^-y^-u^-$ and $Q^+\setminus\mathbf{P}=w^+y^+u^+$, where $u^-, u^+\in V(C\setminus\mathbf{P})$. The following is immediate from the definition of $Q^-, Q^+$, together with our triangulation conditions. 

\begin{claim}\label{NoPathLenQQ'} $G$ contains no path of length most two with one endpoint in $w^-\mathbf{P}^1w^+$ and the other endpoint in $u^-(C\setminus\mathbf{P})u^+$, except for the paths $Q^-\setminus\mathbf{P}$ and $Q^+\setminus\mathbf{P}$. \end{claim}

By Claim \ref{RinMinusPlusWalk}, there is a $Q^-$-incision $\psi^-$ and a $Q^+$-incision $\psi^+$. Note that Theorem \ref{QBandMainResCondensed} implies that $d(Q^-, Q^+)\geq\frac{N_{\textnormal{mo}}}{3}-10$, so $\sigma=\psi^-\cup\pi\cup\psi^+$ is a proper $L$-coloring of its domain. Let $D$ be the cycle $w^-\mathbf{P}^1w^+y^+u^+(C\setminus\mathbf{P})u^-y^-$ and let $G'\cup G''$ be the unique $D$-partition of $G$, where $C\subseteq G'$. Note that $V(G'\setminus D)$ is $(L, \tau)$-inert in $G$, and $D$ is a chordless cycle. Furthermore, each of $y^-, y^+$ has an $L_{\sigma}$-list of size at least three. Now, let $P_0:=w^-\mathbf{P}^1w^+$ and $P_1:=u^-(C\setminus\mathbf{P})u^+$. Since $G$ is short-inseparable, it follows from Theorem \ref{NOver4GoodSideCorCritMos} that $D$ is uniquely 4-determined in $G''$ with respect to the list-assignment $L_{\pi}$, and furthermore, each of $P_0, P_1$ is a 3-consistent subpath of $D$. Since $d(Q^-, Q^+)\geq\frac{N_{\textnormal{mo}}}{3}-10$, it follows that, for each $k=0,1$, there are at least $\frac{1}{2}\left(\frac{N_{\textnormal{mo}}}{3}-10\right)$ $P_k$-peaks. Note that $\frac{1}{2}\left(\frac{N_{\textnormal{mo}}}{3}-10\right)>28$. Given a $P_k$-peak $v\in V(G''\setminus D)$, we say that $v$ is \emph{walled} if $v$ is an internal $P_k$-peak such that both $d(v, Q^-\cup Q^+)>2$ and $d(v, P_{1-k})>2$. Theorem \ref{NOver4GoodSideCorCritMos}, together with our bound on $d(Q^-, Q^+)$, implies that, for each $k=0,1$, there are at least four walled $P_k$-peaks, so, by applying Theorem \ref{LinkPlusOneMoreVertx} to each of $P_0, P_1$, regarded as subpaths of $D$, where each of $u^-, u^+, w^-, w^+$ is precolored by $\sigma$, we extend $\sigma$ to a $C$-reduction coloring. \end{proof}

\section{Boundary Analysis for Closed Rings}\label{BoundCloseRiSec}

We prove the closed-ring analogue of Theorem \ref{MainCollarOpRingRedClF} in Section \ref{RedNearRingsSec}. To prove that it holds, we first need to describe the structure of a critical mosaic near each closed ring. Section \ref{BoundCloseRiSec} consists of the proof of the following result.

\begin{lemma}\label{StructureofC1inClosedRingLemma} Let $\mathcal{T}=(\Sigma, G, \mathcal{C}, L, C_*)$ be a critical mosaic, $C\in\mathcal{C}$ be a closed ring, and $\pi$ be the unique $L$-coloring of $V(C)$. Suppose there is a $y\in D_1(C)$ with $|L_{\pi}(y)|=2$. Let $U:=\{x\in D_1(C): |N(x)\cap V(C)|=1\}$. Then there is unique facial cycle $C^1$ of $G\setminus C$ with $V(C^1)=D_1(C)$ satisfying the following:
\begin{enumerate}[label=\arabic*)]
\itemsep-0.1em
\item $U\neq\varnothing$ and $C^1$ is an induced cycle which is uniquely 4-determined in $G\setminus C$ with respect to $L_{\pi}$; AND
\item Let $Q$ be a 4-chord of $C$, where $V(Q)\cap D_2(C)$ consists of a lone vertex $z$. Letting $P:=G^{\textnormal{small}}_Q\cap C^1$, we have
\begin{enumerate}[label=\alph*)]
\itemsep-0.1em
\item $V(\mathring{P})\cap U=\varnothing$. Furthermore, either $|E(P)|\leq 2$ or $|U\cap V(C^1\setminus\mathring{P})|>1$; AND
\item If $y\in V(\mathring{P})$, then $y\in N(z)$ and $|E(P)|\leq 4$. 
\end{enumerate}
\end{enumerate}
 \end{lemma}

\begin{proof} As $C$ is semi-shortcut-free, our triangulation conditions imply that there is a unique facial cycle $C^1$ of $G\setminus C$ with $V(C^1)=D_1(C)$, and for each $x\in V(C^1)$, the graph $G[N(x)\cap V(C)]$ is a path of length at most two. Since $G$ is short-inseparable and $C$ is uniquely $N_{\textnormal{mo}}/3$-determined in $G$ with respect to $L$, it follows that $C^1$ is uniquely 4-determined in $G\setminus C$ with respect to $L_{\pi}$. Now suppose there is a $y\in V(C^1)$ with $|L_{\pi}(y)|=2$. By M2), every vertex of $C^1-y$ has an $L_{\pi}$-list of size at least three.  

\begin{claim}\label{AvShortcutWrap12}  $U\neq\varnothing$ and, for any generalized chord $Q$ of $C$ with $|E(Q)|\leq 5$, we have $|E(Q)|\geq |E(G^{\textnormal{small}}_Q\cap C)|$. Furthermore, if there is a 4-chord $Q'$ of $C$, where $G^{\textnormal{small}}_{Q'}\cap C^1$ is a path $P'$ with $|E(P')|>2$, then $|U\cap V(C^1\setminus\mathring{P}')|\geq 2$. \end{claim}

\begin{claimproof} Since $G[N(y)\cap V(C)]$ is a path of length two, we have $|E(C^1)|\leq (|E(C)|-1)+|U|$. Now suppose any one of the three statements of Claim \ref{AvShortcutWrap12} does not hold. Thus, there is a separating cycle $D$ in $G$, where $V(D)\subseteq B_2(C)$, such that, letting $G=G_0\cup G_1$ be the natural $D$-partition of $G$, there is a $k\in\{0,1\}$ satisfying the following. 

\begin{itemize}
\itemsep-0.1em
\item $C\subseteq G_k$ and, in particular, if $C=C_*$ then $G_k=\textnormal{Ext}(D)$, and otherwise $G_k=\textnormal{Int}(D)$; AND
\item The unique component of $\Sigma\setminus D$ containing $G_k$ is an open disc. Furthermore, $|V(D)|<|V(C)|$ and every element of $\mathcal{C}\setminus\{C\}$ is contained in $G_{1-k}$. 
\end{itemize}

In particular, if $U=\varnothing$, we take $D=C^1$. By  Proposition \ref{LowGenus+OneFaceCase}, $\mathcal{C}\neq\{C_*\}$, so $|\mathcal{C}|>1$. Thus, $D$ separates $C$ from an element of $\mathcal{C}\setminus\{C\}$. By Corollary \ref{NonRingSeparatingCor}, we get $\textnormal{Rk}(C)-\left(N_{\textnormal{mo}}+\frac{1}{2}\right)\cdot |E(D)|<2$. Since $|E(D)|\leq |E(C)|-1$, we obtain $N_{\textnormal{mo}}\cdot |E(C)|<2+\left(N_{\textnormal{mo}}+\frac{1}{2}\right)\cdot\left(|E(C)|-1\right)$, which is false, since $|E(C)|\leq N_{\textnormal{mo}}$ and $N_{\textnormal{mo}}\geq 200$. \end{claimproof}

We now show over Claims \ref{Fork23CalQ}-\ref{NoChordC1} that $C^1$ is an induced cycle. Note that, given  a 3-chord $Q$ of $C$, where $Q\setminus C$ is a chord of $C^1$, we have $|E(G^{\textnormal{small}}_Q\cap C)|>1$, as $G$ is short-inseparable. Thus, by Claim \ref{AvShortcutWrap12}, $2\leq |E(G^{\textnormal{small}}_Q\cap C)|\leq 3$. For each $k\in\{2,3\}$, let $\mathcal{Q}_{3, k}$ be the set of $3$-chords $Q$ of $C$ such that $Q\setminus C$ is a chord of $C^1$ and $|E(G^{\textnormal{small}}_Q\cap C)|=k$.

\begin{claim}\label{Fork23CalQ} For each $k=2,3$, if $\mathcal{Q}_{3, k}\neq\varnothing$, then there is a $Q\in\mathcal{Q}_{3, k}$ with the property that, for any 2-path $R\subseteq G^{\textnormal{large}}_Q$ with both endpoints in $Q$, and otherwise disjoint to $Q$, then the endpoints of $R$ either both lie in $V(Q\cap C)$ or are consecutive in $Q$. Furthermore, $Q$ is a proper 3-chord of $C$ and an induced path in $G^{\textnormal{large}}_Q$. \end{claim}

\begin{claimproof} We choose a $Q\in\mathcal{Q}_{3,k}$ maximizing $|V(G^{\textnormal{small}}_Q)|$ over the elements of $\mathcal{Q}_{3,k}$. Let $Q=x_0y_0y_1x_1$. Since $y_0y_1$ is a chord of $C^1$ and $G$ is short-inseparable, it is immediate that $x_0\neq x_1$ and $Q$ is an induced path in $G$. If there is a 2-path $R\subseteq G^{\textnormal{small}}_Q$ violating the claim, then, since $G$ is short-inseparable, there is a $Q'\in\mathcal{Q}_{3,k}$ with the same endpoints as $Q$, where $G^{\textnormal{small}}_Q\subsetneq G^{\textnormal{small}}_{Q'}$, contradicting our choice of $Q$. \end{claimproof}

\begin{claim}\label{Q3ChordMidEdge} Let $Q$ be a 3-chord of $C$ whose middle edge is a chord of $C^1$. Then $|E(G^{\textnormal{small}}_Q\cap C)|=3$. \end{claim}

\begin{claimproof} Suppose not. Thus. $\mathcal{Q}_{3, 2}\neq\varnothing$. Let $Q\in\mathcal{Q}_{3, 2}$ be as in Claim \ref{Fork23CalQ}. Let $Q=x_0y_0y_1x_1$.
 
\vspace*{-8mm}
\begin{addmargin}[2em]{0em}
\begin{subclaim}\label{attachLoneVertObtH} Let $H$ be an embedding on $\Sigma$ obtained from $G^{\textnormal{large}}_Q$ by adding a new vertex $u$ adjacent to all four vertices of $Q$. Let $L'$ be a list-assignment for $V(H)$ where $L'(u)=\{c\}$ for some color $c\not\in\{\pi(x_0), \pi(x_1))\}$, and otherwise $L'=L$, and furthermore, either $c\not\in L(y_0)\cup L(y_1)$ or $y\in V(G^{\textnormal{small}}_Q\setminus Q)$. Then $H$ is $L'$-colorable. \end{subclaim}

\begin{claimproof} We obtain $H$ by adding $u$ to the unique component of $\Sigma\setminus C$ not containing any vertices of $G$. Let $D:=(G^{\textnormal{large}}_Q\cap C)+x_0ux_1$. Let $D_*$ be an element of $\mathcal{C}\setminus\{C\})\cup\{D\}$, where $D_*=D$ if $C=C_*$, and otherwise $D_*=C_*$. Let $\mathcal{T}':=(\Sigma, H, (\mathcal{C}\setminus\{C\})\cup\{D\}, L', D_*)$. Since $y_0y_1$ is a chord of $C^1$, we have $|V(H)|<|V(G)|$, so it suffices to show that $\mathcal{T}'$ is a mosaic (where $D$ is a closed ring) and then the criticality of $\mathcal{T}$ implies that $H$ is $L'$-colorable. As $|E(G^{\textnormal{small}}_Q\cap C)|=2$, it follows from our choice of $Q$ that $\mathcal{T}'$ is still short-inseparable, and it is also chord-triangulated, so $\mathcal{T}'$ is indeed a tessellation. It is immediate that $\mathcal{T}'$ satisfies M0)-M1). In particular, $|E(D)|=|E(C)|$, which also implies that $\mathcal{T}'$ still satisfies M3)-M4) and $\textnormal{ew}(H)\geq\textnormal{ew}(G)$. Fact \ref{HighEwTriangleFwF1Cycle} applied to the cycle $(G^{\textnormal{large}}_Q\cap C)+Q$ implies that $\textnormal{fw}(H)$ also satisfies the bound in M5). There is at least one $i\in\{0,1\}$ such that $|N(y_i)\cap V(C)|=1$, or else we contradict Claim \ref{AvShortcutWrap12}, so our choice of color for $u$ implies that $D$ is $(L', V(D))$-predictable. It is also immediate that $D$ is semi-shortcut-free, so $\mathcal{T}'$ is a mosaic and we are done. \end{claimproof}\end{addmargin}

Let $F$ be the 5-cycle $Q+(G^{\textnormal{small}}_Q\cap C)$. Since $y_0y_1$ is a chord of $C^1$, we have $V(G^{\textnormal{small}}_Q)\neq V(F)$. Note that $F$ is induced in $G^{\textnormal{small}}_Q$ and $G^{\textnormal{large}}_Q\cup C$ is induced in $G$. Subclaim \ref{attachLoneVertObtH} implies that $\pi$ extends to an $L$-coloring $\phi$ of $G^{\textnormal{large}}\cup C$. Since $G$ is not $L$-colorable, it follows from Theorem \ref{BohmePaper5CycleCorList} applied to $G^{\textnormal{small}}_Q$ that $G\setminus F$ consists of a lone vertex $z$ adjacent to all five vertices of $F$, where $L_{\pi}(z)\subseteq\{\phi(y_0), \phi(y_1)\}$. Thus, $z=y$. But then, choosing any $c\in L_{\pi}(y)$, Subclaim \ref{attachLoneVertObtH} implies that there is an $L$-coloring of $G^{\textnormal{small}}_Q+y$ using $c$ on $y$, so $G$ is $L$-colorable, which is false. \end{claimproof}

\begin{claim}\label{NoChordC1} There are no chords of $C^1$. \end{claim}

\begin{claimproof} Suppose toward a contradiction that there is a chord of $C^1$. By Claim \ref{Q3ChordMidEdge}, $\mathcal{Q}_{3, 3}\neq\varnothing$, so let $Q\in\mathcal{Q}_{3,3}$ be as in Claim \ref{Fork23CalQ}. Let $Q:=x_0y_0y_1x_1$. Since $\mathcal{Q}_{2,3}=\varnothing$, Claim \ref{AvShortcutWrap12} implies that $N(y_i)\cap V(C)=\{x_i\}$ for each $i=0,1$. 

\vspace*{-8mm}
\begin{addmargin}[2em]{0em}
\begin{subclaim}\label{SecoAttAuxGraph} Let $i\in\{0,1\}$ and let $H$ be an embedding on $\Sigma$ obtained from $G^{\textnormal{large}}_Q$ by adding two new vertices $u, u^*$, where $uu^*\in E(H)$, and furthermore, $u$ is adjacent to $x_i, y_i, y_{1-i}$ and $u^*$ is adjacent to $y_{1-i}, x_{1-i}$. Let $L'$ be a list-assignment for $V(H)$ where $L'(u)=\{c\}$ and $L'(u^*)=\{d\}$ for some colors $c\neq d$ with $c\neq\pi(x_i)$ and $d\neq\pi(x_{1-i})$, and otherwise $L'=L$. Suppose further that either $|L(y_{1-i})\setminus\{c, d, \pi(x_{1-i})\}|\geq 3$ or $y\in V(G^{\textnormal{small}}_Q\setminus Q)$. Then $H$ is $L'$-colorable. \end{subclaim}

\begin{claimproof} Let $D:=(G^{\textnormal{large}}_Q\cap C)+x^iuu^*y^i$. Let $D_*$ be an element of $(\mathcal{C}\setminus\{C\})\cup\{D\}$, where $D_*=D$ if $C=C_*$, and otherwise $D_*=C_*$. Let $\mathcal{T}':=(\Sigma, H, (\mathcal{C}\setminus\{C\})\cup\{D\}, L', D_*)$. Since $y_0y_1$ is a chord of $C^1$ and $|E(G^{\textnormal{small}}_Q\cap C)|=3$, we have $|V(H)|<|V(G)|$, so, as in Subclaim \ref{attachLoneVertObtH}, it suffices to show that $\mathcal{T}'$ is a mosaic (where $D$ is a closed ring).  Since $|E(G^{\textnormal{small}}_Q\cap C)|=3$, it follows from our choice of $Q$ that $\mathcal{T}'$ is still short-inseparable, and it is also chord-triangulated, so $\mathcal{T}'$ is indeed a tessellation. It is immediate that $\mathcal{T}'$ satisfies M0)-M1). In particular, $|E(D)|=|E(C)|$, which also implies that $\mathcal{T}'$ still satisfies M3)-M4) and $\textnormal{ew}(H)\geq\textnormal{ew}(G)$. Fact \ref{HighEwTriangleFwF1Cycle} applied to the cycle $(G^{\textnormal{large}}_Q\cap C)+Q$ implies that $\textnormal{fw}(H)$ also satisfies the bound in M5). Since $|N(y_0)\cap V(C)|=|N(y_1)\cap V(C)|=1$, our choice of color for $u$ implies that $D$ is $(L', V(D))$-predictable. It is also immediate that $D$ is semi-shortcut-free, so $\mathcal{T}'$ is a mosaic and we are done. \end{claimproof}\end{addmargin}

Let $F$ be the 6-cycle $Q+(G^{\textnormal{small}}_Q\cap C)$. Note that $F$ is an induced cycle of $G^{\textnormal{small}}_Q$ and $G^{\textnormal{large}}_Q\cup C$ is induced in $G$, and Subclaim \ref{SecoAttAuxGraph} implies that $\pi$ extends to an $L$-coloring of $V(G^{\textnormal{small}}_Q\cup C)$. Since $G$ is not $L$-colorable, it follows from Theorem \ref{BohmePaper5CycleCorList} applied to $G^{\textnormal{small}}_Q$ that $1\leq |V(G^{\textnormal{small}}_Q)\setminus F|\leq 3$ and $V(G^{\textnormal{small}}_Q\setminus F)\subseteq V(C^1)$. If $|V(G^{\textnormal{small}}_Q\setminus F)|=1$, then our triangulation conditions imply that $G^{\textnormal{small}}_Q$ is a wheel, which is false, as $C$ is semi-shortcut-free in $G$. Thus, $2\leq |V(G^{\textnormal{small}}_Q\setminus F)|\leq 3$. Furthermore, there is a $y'\in V(G^{\textnormal{small}}_Q\setminus F)$ such that $G[N(y')\cap V(C\cap F)]$ is a path of length at least two, so $G[N(y')\cap V(C\cap F)]$ is path of length at precisely two, and Theorem \ref{BohmePaper5CycleCorList} gives us $|L_{\pi}(y')|=2$. Thus, $y'=y$. Now, Theorem \ref{BohmePaper5CycleCorList} also implies also implies that there exists an edge $uu^*\in E(G^{\textnormal{small}}_Q\setminus F)$ and an $i\in\{0,1\}$ such that $N(u)\cap V(Q)=\{x_i, y_i, y_{1-i}\}$ and $N(u^*)\cap V(Q)=\{y_{1-i}, x_{1-i}\}$. Note that $G^{\textnormal{small}}_Q\setminus F$ is $L_{\pi}$-colorable, so Subclaim \ref{SecoAttAuxGraph} implies that $\pi$ extends to an $L$-coloring of $G$, which is false.  \end{claimproof}

From Claims \ref{AvShortcutWrap12} and \ref{NoChordC1}, we get that $U=\varnothing$ and $C^1$ is induced, so 1) of Lemma \ref{StructureofC1inClosedRingLemma} holds. Now we show 2). 

\begin{claim}\label{QContYNbrZ} Let $Q$ be a 4-chord of $C$, where $V(Q)\cap D_2(C)$ consists of a lone vertex $z$. Let $P=G^{\textnormal{small}}_Q\cap C^1$, and suppose that $y\in V(\mathring{P})$. Then $yz\in E(G)$. \end{claim}

\begin{claimproof} Suppose $yz\not\in E(G)$. let $u, u^*$ be the neighbors of $y$ on $C^1$. If each of $u, u^*$ is adjacent to $z$, then, since $G$ has no induced 4-cycles and there are no chords of $C^1$, we get $y\in N(z)$, contradicting our assumption. Thus, suppose without loss of generality that $u\not\in N(z)$. We now define a path $T\subseteq P$ as follows: If $u\in U$, then $T=y$. Otherwise, $T=uy$. If $u\not\in U$, then $G[N(u)\cap V(C)]$ has length precisely one, or else we contradict Claim \ref{AvShortcutWrap12}. Let $P_T=\bigcup_{x\in V(T)}G[N(x)\cap V(C)]$. Note that $P_T$ is a path of length $|E(T)|+2$. Now we produce a contradiction by constructing a mosaic with underlying graph $G':=G\setminus\mathring{P}_T$. Let $C'$ be the unique facial cycle of $G'$ with $V(C')=V(T)\cup V(C\setminus\mathring{P}_T)$, i.e $C'$ is obtained from $C$ by replacing $\mathring{P}_T$ with $T$. Note that $|E(C')|=|E(C)|$ and $|V(G')|<|V(G)|$, as $P_T$ has length at least two. Let $\phi$ be an $L_{\pi}$-coloring of $T$ and let $L'$ be a list-assignment for $V(G')$, where $L'(v)=\{\phi(v)\}$ for each $v\in V(T)$, and otherwise $L'=L$. Let $C_*'$ be an element of $(\mathcal{C}\setminus\{C\})\cup\{C'\}$, where $C_*'=C'$ if $C=C_*$ and otherwise $C_*'=C_*$. We claim now that $\mathcal{T}':=(\Sigma, G', (\mathcal{C}\setminus\{C\})\cup\{C'\}, L', C_*')$ is a mosaic, where $C'$ is a closed ring. Since $|E(C')|=|E(C)|$ and no vertex of $T$ is adjacent to $z$, it follows that $\mathcal{T}'$ satisfies M3)-M4), and M0)-M1) are immediate as well. Furthermore, $\textnormal{ew}(G')\geq\textnormal{ew}(G)$, so it follows from Fact \ref{HighEwTriangleFwF1Cycle} that $\textnormal{fw}(G')$ also satisfies the bound in M5). We just need to check M2). 

Note that, for each endpoint $v$ of $C'\setminus T$, $N(v)$ intersects $V(P_T)$ on precisely one endpoint of $P_T$, as $C^1$ is induced. Furthermore, $G[N(v)\cap V(C)]$ is a path of length at most one, or else we contradict Claim \ref{AvShortcutWrap12}. Thus, $C'$ is semi-shortcut-free in $G'$.  Let $\pi'$ be the unique $L'$-coloring of $V(C')$. Now, at least one endpoint of $C'\setminus T$ lies in $U$. This is immediate from the definition of $T$ if $T=y$, and it follows from Claim \ref{AvShortcutWrap12} if $T=uy$.  Since $C$ is $(L, V(C))$-predictable, it follows that all but at most one vertex of $V(C')$ has an $L'_{\pi'}$-list of size at least three, so $\mathcal{T}'$ satisfies M2). Thus, $\mathcal{T}'$ is a mosaic. Since $\mathcal{T}$ is critical, $G'$ is $L'$-colorable, so $G$ is $L$-colorable, a contradiction.  \end{claimproof}

\begin{claim}\label{LastClaimIn12Seq} Let $Q$ be a 4-chord of $C$, where $V(Q)\cap D_2(C)$ consists of a lone vertex $z$. Let $P=G^{\textnormal{small}}_Q\cap C^1$. Then
\begin{enumerate}[label=\roman*)]
\itemsep-0.1em
\item $V(\mathring{P}\cap U)=\varnothing$. 
\item Either $y\not\in V(\mathring{P})$ or $|E(P)|\leq 4$.
\end{enumerate} \end{claim}

\begin{claimproof} It suffices to show that i) holds. If i) holds, then Claim \ref{AvShortcutWrap12} implies that ii) holds as well. Let $\mathcal{Q}$ be the set of 4-chords $Q$ of $C^1$ such that $|V(Q)\cap D_2(C)|=1$, where at least one internal vertex of $G^{\textnormal{small}}_Q\cap C^1$ lies in $U$, and $y$ is not an internal vertex of $G^{\textnormal{small}}_Q\cap C^1$. Note that, if there is a 4-chord of $C$ violating i), then Claim \ref{QContYNbrZ} implies that $\mathcal{Q}\neq\varnothing$, so it suffices to show that $\mathcal{Q}=\varnothing$. Suppose toward a contradiction that $\mathcal{Q}\neq\varnothing$. Note that, since $G$ is short-inseparable, each $Q\in\mathcal{Q}$ is a proper 4-chord of $C$. Choose a $Q\in\mathcal{Q}$ maximizing $|V(G^{\textnormal{small}}_Q)|$ over the elements of $\mathcal{Q}$. Let $P$ be the path $G^{\textnormal{small}}_Q\cap C^1$ and let $Q:=x^0y^0zy^1x^0$. Since $Q\in\mathcal{Q}$ and $G$ has no induced 4-cycles, we have $|E(P)|\geq 3$. Note that, since $Q$ is a proper 4-chord of $C$, our triangulation conditions imply that we have $V(\mathring{P})\not\subseteq U$. Now, let $H$ be an embedding on $\Sigma$ obtained from $G$ by doing the following:
\begin{itemize}
\itemsep-0.1em
\item First, we delete all the vertices of $G^{\textnormal{small}}_Q\setminus V(C\cup R)$
\item Next, we delete all the vertices of $U\cap V(\mathring{P})$ and replace them with edges, so that we obtain an induced path $P_H$ from $P$, where each internal vertex of $P_H$ is adjacent to a subpath of $C$ of length at least one.
\item Finally, we add an edge from $z$ to each vertex of $P_H$ that is not already adjacent to $z$. 
\end{itemize}

Note that $|V(H)|<|V(G)|$. 

\vspace*{-8mm}
\begin{addmargin}[2em]{0em}
\begin{subclaim}\label{HAuxGraphLCol} $H$ is $L$-colorable. \end{subclaim}

\begin{claimproof} We first show that $H$ is short-inseparable. Note that any chord of $Q$, if it exists, is an edge of $C$. Suppose $H$ is not short-inseparable. For each $i=0,1$, let $u^i$ be the unique vertex of $V(\mathring{P})\setminus U$ which is adjacent to $y^i$ on $P_H$. Since $H$ is not short-inseparable, there is an $i\in\{0,1\}$ such that $u^i\in N(x^i)$ and $z, x^i$ have a common neighbor $u^i_*\in V(G^{\textnormal{large}}_Q)\setminus V(Q)$. In particular, $|N(y^i)\cap V(C)|=1$, so $y^i\neq y$, and, letting $Q'=x^iu^i_*zy^{1-i}x^{1-i}$, our triangulation conditions imply that $V(G^{\textnormal{small}}_{Q'})=V(G^{\textnormal{small}}_Q)\cup\{u^i_*\}$. Since $y^i\neq y$, we violate the maximality of $|V(G^{\textnormal{small}}_Q)|$. Thus, $H$ is short-inseparable, and $(\Sigma, H, \mathcal{C}, L, C_*)$ is a tessellation. It is immediate that $H$ satisfies all of M0)-M5), so $(\Sigma, H, \mathcal{C}, L, C_*)$ is a mosaic. Since $\mathcal{T}$ is critical, $H$ is $L$-colorable. \end{claimproof}\end{addmargin}

Now, it follows from Subclaim \ref{HAuxGraphLCol} that $V(G^{\textnormal{small}}_Q\cup C)$ admits an $L$-coloring $\phi$. Let $K:=G^{\textnormal{small}}_Q\setminus C$. We regard a$K$ as a planar embedding with outer cycle $y^0zy^1+P$. Each chord of the outer cycle of $K$ is incident to $z$, and the $L_{\pi}$-coloring $(\phi(y^0), \phi(z), \phi(y^1))$ of $y^0zy^1$ does not extend to $L_{\pi}$-color $K$, so every vertex of $\mathring{P}$ has an $L_{\pi}$-list of size precisely three, contradicting our assumption that $U\cap V(\mathring{P})\neq\varnothing$.   \end{claimproof}

Combining Claims \ref{QContYNbrZ} and\ref{LastClaimIn12Seq}, we complete the proof of Lemma \ref{StructureofC1inClosedRingLemma}. \end{proof}

\section{Reductions Near the Rings: Closed Rings}\label{RedNearRingsSec}

Section \ref{RedNearRingsSec} consists of the proof of the following result.

\begin{theorem}\label{ClosedRingReductTrip} Let $\mathcal{T}=(\Sigma, G, \mathcal{C}, L, C_*)$ be a critical mosaic. For each closed $C\in\mathcal{C}$, there is a $C$-reduction. \end{theorem}

\begin{proof} Let $C\in\mathcal{C}$ be a closed ring and let $\pi$ be the unique $L$-coloring of $V(C)$. If every vertex of $B_1(C)$ has an $L_{\pi}$-list of size at least three, we are done, so, by M2), we suppose that there exists a $y\in V(C^1)$ such that $|L_{\pi}(y)|=2$, where each vertex of $C^1-y$ has an $L_{\pi}$-list of size at least three. Let $C^1, U$ be as in Lemma \ref{StructureofC1inClosedRingLemma}, so $C^1$ is induced. In the remainder of the proof of Theorem \ref{ClosedRingReductTrip}, given a subpath $P$ of $C^1$, whenever we write $\textnormal{Link}(P)$, we are suppressing the subscript $L_{\pi}$. This is well-defined as $C^1$ is uniquely 2-determined with respect to $L_{\pi}$. Given a vertex $v\in V(C^1)$, we say that $v$ is a \emph{hinge vertex} if there is no 4-chord $Q$ of $C$ with $D_2(C)\cap V(Q)|=1$, where $v$ is an internal vertex of the path $G^{\textnormal{small}}_Q\cap C^1$. By Lemma \ref{StructureofC1inClosedRingLemma}, each element of $U$ is a hinge vertex. 

\begin{claim}\label{ConsSubRMapTwoto} Let $R$ be a 2-consistent subpath of $C^1$, where $y\not\in V(R-p)$, and let $p, p'$ be the endpoints of $R$. Then,
\begin{enumerate}[label=\roman*)]
\itemsep-0.1em
\item for any hinge vertex $u\in V(R-p)$ and any $\phi\in\textnormal{Link}(pRu)$ and $\phi'\in\textnormal{Link}(uRp')$ with $\phi(u)=\phi'(u)$, the union $\phi\cup\phi'$ lies in $\textnormal{Link}(R)$; AND 
\item Let $c\in L_{\pi}(p)$. If $V(R-p)\cap U=\varnothing$, then there exist two elements of $\textnormal{Link}(R)$ using $c$ on $p$ and different colors on $p'$. If $|V(R-p)\cap U|>1$, then there exist three elements of $\textnormal{Link}(R)$ using $c$ on $p$ and different colors on $p'$.
\end{enumerate}  \end{claim}

\begin{claimproof} i) is immediate from the fac that $u$ is a hinge vertex. Since each vertex of $U$ has an $L_{\pi}$-list of size at least four, we get each of ii) from successive applications of L1) of Theorem \ref{MainLinkingResultAlongFacialCycle}, taking appropriate unions and using i). \end{claimproof}

\begin{claim}\label{EachXSplitWalkLen3} For each $x\in D_2(C)$, the graph $G[N(x)\cap V(C^1)]$ is a path of length at most two. \end{claim}

\begin{claimproof} Note that if there is a 2-path $v_1v_2v_3\subseteq C$, where $x$ is adjacent to both of $v_1, v_3$, then $x$ is also adjacent to $v_2$, as $C^1$ is induced and $G$ has no induced 4-cycles. Now suppose the claim does not hold. Thus, there is a 4-chord $Q$ of $C$, where $V(Q)\cap D_2(C)$ consists of a lone vertex $z$, such that the path $P:=G^{\textnormal{small}}_Q\cap C^1$ has length at least three, and both endpoints of $P$ are hinge vertices. By Lemma \ref{StructureofC1inClosedRingLemma}, $V(\mathring{P})\cap U=\varnothing$, and $|U\cap V(C^1\setminus\mathring{P})|\geq 2$. Let $x^0, x^1$ denote the endpoints of $P$ and let $H=G^{\textnormal{small}}_Q\setminus C$. We regard $H$ as a planar embedding with outer cycle $v^0zx^1+(G^{\textnormal{small}}_Q\cap C^1)$. In the proof of Claim \ref{EachXSplitWalkLen3}, whenever we write $\textnormal{End}(x^0zx^1, H)$, we are supressing the subscript $L_{\pi}$.

\vspace*{-8mm}
\begin{addmargin}[2em]{0em}
\begin{subclaim}\label{GluLinkToEnd} Given $\phi\in\textnormal{End}(x^0zx^1, H)$ and $\phi'\in\textnormal{Link}(C\setminus\mathring{P})$, either $\phi(x^0)\neq\phi'(x^0)$ or $\phi(x^1)\neq\phi'(x^1)$.  \end{subclaim}

\begin{claimproof} If $\phi(x^0)=\phi'(x^0)$ and $\phi(x^1)=\phi'(x^1)$, then there is a $C$-reduction, where, in Definition \ref{ReductionOpDefn}, we check $S=V(C\cup C^1)$ and $\psi=\pi\cup\phi\cup\phi'$, contradicting our assumption that Theorem \ref{ClosedRingReductTrip} does not hold. \end{claimproof}\end{addmargin}

\vspace*{-8mm}
\begin{addmargin}[2em]{0em}
\begin{subclaim}\label{yMidC1} $y$ is not a hinge vertex. \end{subclaim}

\begin{claimproof} Suppose $y$ is a hinge vertex. In particular, $y\not\in V(\mathring{P})$. Fix a $c\in L_{\pi}(y)$. We first show that $y$ is not an endpoint of $P$. Suppose $y$ is an endpoint of $P$, say $y=x^0$ for the sake of definiteness. Since $|U\cap V(C\setminus\mathring{P})|>1$, it follows from ii) of Claim \ref{ConsSubRMapTwoto} that there are three elements of $\textnormal{Link}(C\setminus\mathring{P})$ using $c$ on $x^0$ and different colors on $x^1$. By Theorem \ref{SumTo4For2PathColorEnds}, there is a $\phi\in\textnormal{Link}(C\setminus\mathring{P})$ and a $\phi'\in\textnormal{End}(x^0zx^1, H)$ with $\phi(x^0)=\phi'(x^0)=c$ and $\phi(x^1)=\phi'(x^1)$, contradicting Subclaim \ref{GluLinkToEnd}. Thus, $y$ is not an endpoint of $P$. Let $Q^0, Q^1$ be the subpaths of $C\setminus\mathring{P}$ such that $Q^0\cup Q^1=C\setminus\mathring{P}$ and $Q^0\cap Q^1=y$. We have $|U\cap V(Q^0-y)|+|U\cap V(Q^1-y)|>1$. It follows from Theorem \ref{SumTo4For2PathColorEnds}, together with ii) of Claim \ref{ConsSubRMapTwoto}, that there exist a $\sigma^0\in\textnormal{Link}(Q^0)$ and a $\sigma^1\in\textnormal{Link}(Q^1)$ such that $\sigma^0(y)=\sigma^1(y)=c$ and the restriction of $\sigma^0\cup\sigma^1$ to $\{x^0, x^1\}$ lies in $\textnormal{End}(x^0zx^1, H)$. Since $y\not\in\textnormal{Mid}(C^1)$, we have $\sigma^0\cup\sigma^1\in\textnormal{Link}(C\setminus\mathring{P})$, so we contradict Subclaim \ref{GluLinkToEnd}. \end{claimproof}\end{addmargin}

By Subclaim \ref{yMidC1}, there is a 4-chord $Q'$ of $C$, where $V(Q')\cap D_2(C)$ consists of a lone vertex $z'$, such that $G^{\textnormal{large}}_Q\cap C^1$ is a path $P'$ with $y\in V(\mathring{P}')$. By Lemma \ref{StructureofC1inClosedRingLemma}, $yz'\in E(G)$. We may suppose that either $|E(P')|=2$ or $P'=P$, since, if $|E(P')|>3$, then we can replace $P$ with $P'$ if necessary.

\vspace*{-8mm}
\begin{addmargin}[2em]{0em}
\begin{subclaim}\label{ArgShowsPP'} $P=P'$. In particular, $yz\in E(G)$. \end{subclaim}

\begin{claimproof} Suppose not. Thus, $P'$ has length two and $y$ is its midpoint. Let $Q^0, Q^1$ denote the components of $(C\setminus\mathring{P})-y$, where, for each $i=0,1$, $x^i$ is an endpoint of $Q^i$, and $u^i$ denotes the other endpoint of $Q^i$. Furthermore, let $A^i:=L_{\pi}(u^i)\setminus L_{\pi}(y)$ and let let $B^i$ be the set of colors $c\in L_{\pi}(x^i)$ such that there is a $\sigma\in\textnormal{Link}(Q^i)$ with $\sigma(u^i)\in A^i$ and $\sigma(x^i)=c$. Note that, if $u^i\in U$, then $|A^i|\geq 2$. In any case, it follows from L1) of Theorem \ref{MainLinkingResultAlongFacialCycle} and Claim \ref{ConsSubRMapTwoto} that, for each $i=0,1$, we have $|B^i|\geq 1$. Furthermore, if $|U\cap V(Q^i)|\geq 1$, then $|B^i|\geq 2$. Likewise, if $U\cap V(Q^i)|\geq 2$, then $|B^i|\geq 3$. Combining this with Theorem \ref{SumTo4For2PathColorEnds}, we get that there exist a $\sigma^0\in\textnormal{Link}(Q^0)$ and a $\sigma^1\in\textnormal{Link}(Q^1)$ such that $\sigma^i(u^i)\in A^i$ for each $i=0,1$, where the restriction of $\sigma^0\cup\sigma^1$ to $\{x^0, x^1\}$ lies in $\textnormal{End}(x^0zx^1, H)$. Our choice of colors for $u^0, u^1$ implies that $\sigma^0\cup\sigma^1\in\textnormal{Link}(C\setminus\mathring{P})$, so we contradict Subclaim \ref{GluLinkToEnd}. \end{claimproof}\end{addmargin}

Since $y\in V(\mathring{P})$, we have $3\leq |E(P)|\leq 4$ by Lemma \ref{StructureofC1inClosedRingLemma}. 

\vspace*{-8mm}
\begin{addmargin}[2em]{0em}
\begin{subclaim}\label{PPrec4} $|E(P)|=4$ and $y$ is the midpoint of $P$. \end{subclaim}

\begin{claimproof} Suppose not. Thus, for the sake of definiteness, we suppose $yx^0\in E(P)$. Let $A:=L_{\pi}(x^0)\setminus A$. For each $c\in A$, there is an element of $\textnormal{End}_{L_{\pi}}(x^0zx^1, H)$ using $c$ on $x^0$. This follows from Theorem \ref{SumTo4For2PathColorEnds}, by considering a list-assignment $L'$ for $H$ in which $L'(y)=L_{\pi}(y)\cup\{c\}$ and otherwise $L'=L_{\pi}$. Since $|U\cap V(C\setminus\mathring{P})|>1$ and $|A|\geq 1$, a similar argument to that of Subclaim \ref{ArgShowsPP'} shows that there is a $\phi\in\textnormal{Link}(C\setminus\mathring{P})$ with $\phi(x^0)\in A$, where the restriction of $\phi$ to $\{x^0, x^1\}$ lies in $\textnormal{End}_{L_{\pi}}(x^0zx^1, H)$, which contradicts Subclaim \ref{GluLinkToEnd}. \end{claimproof}\end{addmargin}

Applying Subclaim \ref{PPrec4}, we let $P=x^0y^0yy^1x^1$ for some vertices $y^0, y^1$. Since $zy\in E(G)$ and $C^1$ is induced, our triangulation conditions imply that  $H$ is a broken wheel with principal vertex $z$. 

\vspace*{-8mm}
\begin{addmargin}[2em]{0em}
\begin{subclaim}\label{SameListPrecX0X1} $x^0, y^0$ have the same $L_{\pi}$-list of size precisely three. Likewise for $y^1, x^1$. \end{subclaim}

\begin{claimproof} Suppose not. Thus, without loss of generality, we suppose there is a $c\in L_{\pi}(x^0)$ with $c\not\in L_{\pi}(y^0)$. By L1) of Theorem \ref{MainLinkingResultAlongFacialCycle}, there is a $\phi\in\textnormal{Link}(C\setminus\mathring{P})$ with $\phi(x^0)=c$. But the restriction of $\phi$ to $\{x^0, x^1\}$ lies in $\textnormal{End}_{L_{\pi}}(x^0zx^1, H)$, which contradicts Subclaim \ref{GluLinkToEnd}.  \end{claimproof}\end{addmargin}

Using Subclaim \ref{SameListPrecX0X1}, it is straightforward to check that $\textnormal{End}_{L_{\pi}}(x^0zx^1, H)\neq\varnothing$, so we fix a $\phi\in\textnormal{End}_{L_{\pi}}(x^0zx^1, H)$. Subclaim \ref{SameListPrecX0X1} also implies that $x^0, x^1\not\in U$, so $|U\cap (C\setminus P)|>2$. Thus, it follows from i) of Claim \ref{ConsSubRMapTwoto}, together with successive applications of L1) of Theorem \ref{MainLinkingResultAlongFacialCycle}, that any $L_{\pi}$-coloring of $\{x^0, x^1\}$ extends to an element of $\textnormal{Link}(C\setminus\mathring{P})$, contradicting Subclaim \ref{GluLinkToEnd}. This proves Claim \ref{EachXSplitWalkLen3}. \end{claimproof}

Let $\textnormal{Mid}(C^1)$ be the set of $v\in V(C^1)$ such that there is a (necessarily unique) $x\in D_2(C)$, where $G[N(x)\cap V(C^1)]$ is a path of length two with midpoint $v$. Note that $U\cap\textnormal{Mid}(C^1)=\varnothing$, and, by Claim \ref{EachXSplitWalkLen3}, $V(C^1)\setminus\textnormal{Mid}(C^1)$ is precisely the set of hinge vertices. By Lemma \ref{StructureofC1inClosedRingLemma}, $U\neq\varnothing$. We now fix a $u\in U$. We define a \emph{perfect coloring} to be an $L_{\pi}$-coloring $\phi$ of $\{y\}\cup (V(C^1)\setminus\textnormal{Mid}(C^1))$ such that $\textnormal{Mid}(C^1)\setminus\{y\}$ is $(L_{\pi}, \phi)$-inert in $G\setminus C$ and each vertex of $B_2(C)$ has an $L_{\pi\cup\phi}$-list of size at least three. To prove Theorem \ref{ClosedRingReductTrip}, it suffices to prove that there is a perfect coloring. Suppose there is no perfect coloring. We now let $C^1=v_1\cdots v_nv_1$. 

\begin{claim}\label{C1WrappedBy2Paths} Every other vertex of $C^1$ lies in $\textnormal{Mid}(P)$. In particular, $C^1$ has even length. \end{claim}

\begin{claimproof} Suppose not. Say for the sake of definiteness that $v_1, v_n\not\in\textnormal{Mid}(P)$. Consider the following cases.

\textbf{Case 1:} $y\in\textnormal{Mid}(P)$

In this case, $y=v_k$ for some $1<k<n$. Let $Q=v_1\cdots v_{k-1}$ and $Q'=v_{k+1}\cdots v_n$. Possibly one of $Q, Q'$ is a lone vertex. In any case, since $u$ lies in one of $Q, Q'$, it follows from L1) of Theorem \ref{MainLinkingResultAlongFacialCycle}, together with Claim \ref{ConsSubRMapTwoto}, that there exist $\phi\in\textnormal{Link}(Q)$ and $\phi'\in\textnormal{Link}(Q')$ with $\phi(v_1)\neq\phi'(v_n)$ and $\phi(v_{k-1}), \phi'(v_{k+1})\not\in A$, so $\phi\cup\phi'$ is a perfect coloring, contradicting our assumption.

\textbf{Case 2:} $y\not\in\textnormal{Mid}(P)$

In this case, since $|V(C^1)|>3$, there is at least one $1<k<n$ with $v_k\not\in\textnormal{Mid}(P)$,where $y\in\{v_1, v_k, n\}$.Let $R=v_1\cdots v_k$ and $R'=v_k\cdots v_n$. Possibly $u=v_k$, but in any case, since $u$ lies in at least one of $R, R'$, it follows from Claim \ref{ConsSubRMapTwoto}, together with L1) of Theorem \ref{MainLinkingResultAlongFacialCycle}, that there exist $\phi\in\textnormal{Link}(R)$ and $\phi'\in\textnormal{Link}(R')$ with $\phi(v_k)=\phi'(v_k)$ and $\phi(v_1)\neq\phi'(v_n)$, so $\phi\cup\phi'$ is a perfect coloring, contradicting our assumption. \end{claimproof}

Applying Claim \ref{C1WrappedBy2Paths}, we say for the sake of definiteness that $\textnormal{Mid}(P)=\{v_2, v_4, \cdots, v_n\}$. 

\begin{claim} There is no color common to the $L_{\pi}$-lists of all the vertices of $C^1\setminus\textnormal{Mid}(P)$. Furthermore, $y\in\textnormal{Mid}(P)$.\end{claim}

\begin{claimproof} Suppose there is a color $a$ common to the $L_{\pi}$-lists of all the vertices of $C^1\setminus\textnormal{Mid}(P)$. Since $C^1$ is induced, there is an $L$-coloring $\phi$ of $\{v_1, v_3, \cdots, v_{n-1}\}$ obtained by coloring all of these vertices with $a$. Possibly $y\in\textnormal{Mid}(P)$, but, in that case, there is a color left over for $y$. In any case, $\phi$ extends to an $L$-coloring $\phi'$ of $\textnormal{dom}(\phi)\cup\{y\}$, and $\phi'$ is a perfect coloring, contradicting our assumption, so there is no such $a$. Now suppose that $y\not\in\textnormal{Mid}(P)$. Let $Q_0, Q_1$ be the two edge-disjoint subpaths of $C^1$ with endpoints $u, y$, where $Q_0\cup Q_1=C^1$. Since $u\not\in\textnormal{Mid}(P)$, each of $Q_0, Q_1$ has length at least two. Let $A:=L_{\pi}(y)$ and $B:=L_{\pi}(u)$.

\vspace*{-8mm}
\begin{addmargin}[2em]{0em}
\begin{subclaim}\label{NoColLink1}
Both of the following hold.
\begin{enumerate}[label=\roman*)]
\itemsep-0.1em
\item For any $\phi_0\in\textnormal{Link}(Q_0)$ and $\phi_1\in\textnormal{Link}(Q_1)$ either $\phi_0(y)\neq\phi_1(y)$ or $\phi_0(u)\neq\phi_1(u)$; AND
\item Let $i\in\{0,1\}$. Then, at most one color of $B$ is $(Q_i, u, A)$-linking. Furthermore,  for any $S\subseteq B$ with $|S|=3$, no color of $A$ is $(Q_i, y, S)$-linking.
\end{enumerate} \end{subclaim}

\begin{claimproof} If i) does not hold, then $\phi_0\cup\phi_1$ is a perfect coloring, contradicting our assumption. Now suppose there is a $T\subseteq L_{\pi}(u)$ with $|T|\geq 2$, where each color of $T$ is $(Q_i, u, A)$-linking. By L1) of Theorem \ref{MainLinkingResultAlongFacialCycle}, there is a $\phi\in\textnormal{Link}(Q_{1-i})$ with $\phi(u)\in T$, contradicting i). Now suppose there is an $a\in A$ which is $(Q_i, y, S)$-linking. By L1) of Theorem \ref{MainLinkingResultAlongFacialCycle}, there is a $\phi\in\textnormal{Link}(Q_{1-i})$ with $\phi(y)=a$ and $\phi(u)\in S$, contradicting i).  \end{claimproof}\end{addmargin}

Now, since $|L(u)|\geq 4$, it follows from ii) of Subclaim \ref{NoColLink1} that for each $i=0,1$, there is a subseteq $B_i$ of $L(u)$ with $|B_i|=3$, where no color of $B_i$ is $(Q_i, u, A)$-linking and no color of $A$ is $(Q_i, u, B_i)$-linking. Thus, by L2) of Theorem \ref{MainLinkingResultAlongFacialCycle}, it follows that, for each $i=0,1$, $L_{\pi}(y)$ is contained in the $L_{\pi}$-lists of all the vertices of $Q_i$. But then, there is a color common to the $L_{\pi}$-lists of all the vertices of $C^1$, which is false, as shown above. \end{claimproof}

Since $y\in\textnormal{Mid}(P)$, say $y=v_2$ for the sake of definiteness. Let $y^+$ be the unique vertex of $B_2(C)$ such that $N(y^+)\cap V(C^1)=\{v_1, v_2, v_3\}$.  Let $Q=C^1-y$. 

\begin{claim}\label{AvoidEachAWrap} $A\cap L_{\pi}(v_1)$ and $A\cap L_{\pi}(v_3)$ are disjoint singletons. \end{claim}

\begin{claimproof} We fix $c\in L_{\pi}(v_1)\setminus A$ and $c'\in L_{\pi}(v_3)\setminus A$. We first show that, for each $v\in\{v_1, v_3\}$, we have $A\cap L_{\pi}(v)\neq\varnothing$. Suppose not. Say $A\cap L_{\pi}(v_1)=\varnothing$ for the sake of definiteness. By L1) of Theorem \ref{MainLinkingResultAlongFacialCycle}, there is a $\phi\in\textnormal{Link}(Q)$ with $\phi(v_3)=c'$. Since $\phi(v_1)\not\in A$, $\phi$ is a perfect coloring, contradicting our assumption. To finish, it suffices to show that $A\cap L_{\pi}(v_1)\cap L_{\pi}(v_3)=\varnothing$. Suppose there is an $a\in A\cap L_{\pi}(v_1)\cap L_{\pi}(v_3)$.  Let $X=\{c, a\}$ and $X'=\{c, a\}$. By L2) of Theorem \ref{MainLinkingResultAlongFacialCycle}, there is either a color of $X$ which is $(Q, v_1, X')$-linking or a color of $X'$ which is $(Q, v_3, X)$-linking. Say for the sake of definteness that there is a $d\in X$ which is $(Q, v_1, X')$-linking. If $d=c$, then there is a $\phi\in\textnormal{Link}(Q)$ with $\phi(v_1)=c$ and $\phi(v_3)=c'$, and $\phi$ is a perfect coloring, which is false, so $d=a$. But then, there is a $\phi\in\textnormal{Link}(Q)$ with $\phi(v_1)=\phi(v_3)=a$. There is a color left over for $y$, and since $|L_{\pi\cup\phi}(y^+)|\geq 4$, the lone extension of $\phi$ to an $L_{\pi}$-coloring of $\textnormal{dom}(\phi)\cup\{y\}$ is a perfect coloring, contradicting our assumption. \end{claimproof}

By Claim \ref{AvoidEachAWrap}, $|L_{\pi}(v_1)\setminus A|\geq 2$ and $L_{\pi}(v_3)\setminus A|\geq 2$. By L1) of Theorem \ref{MainLinkingResultAlongFacialCycle}, there is a $\phi\in\textnormal{Link}(Q)$ with $\phi(v_1)\not\in A$ and $\phi(v_3)\not\in A$, so $\phi$ is a perfect coloring, contradicting our assumption. This proves Theorem \ref{ClosedRingReductTrip}.  \end{proof}

With Theorems \ref{MainCollarOpRingRedClF} and \ref{ClosedRingReductTrip} in hand, we define the following:

\begin{defn}\label{AssociateAuxTessToCalF} \emph{Let $\mathcal{T}=(\Sigma, G, \mathcal{C}, L, C_*)$ be a critical mosaic and $\mathcal{F}\subseteq\mathcal{C}$, where $\mathcal{F}\neq\varnothing$. For each $F\in\mathcal{F}$, let $(S_F, \psi_F, F^r)$ be an $F$-reduction. Let $\mathcal{F}^r:=\{F^r: F\in\mathcal{F}\}$. We define a chart $\mathcal{T}_{\mathcal{F}}^r:=(\Sigma, G', (\mathcal{C}\setminus\mathcal{F})\cup\mathcal{F}^r, L', C_*')$ as follows, where, in particular, $\mathcal{T}^r_{\mathcal{F}}$ has the same underlying surface as $\mathcal{T}$, and $G'$ is a subgraph of $G$.}
\begin{enumerate}[label=\emph{\arabic*)}]
\itemsep-0.1em
\item \emph{$C_*'=C_*$ if $C_*\not\in\mathcal{F}$, and otherwise $C_*'=C_*^r$. }
\item \emph{$G'=G\setminus\left(\bigcup_{F\in\mathcal{F}}S_F\right)$ and $L'=L_{\tau}$, where $\tau=\bigcup_{F\in\mathcal{F}}\psi_F$.}
\end{enumerate}
\end{defn}

Theorems \ref{MainCollarOpRingRedClF} and \ref{ClosedRingReductTrip}  imply that, in the setting above, for any nonempty $\mathcal{F}\subseteq\mathcal{C}$, the chart $\mathcal{T}_{\mathcal{F}}^r$ exists. If $\mathcal{F}$ is a singleton, i.e $\mathcal{F}=\{F\}$ for some $F\in\mathcal{C}$, then we denote the above chart by $\mathcal{T}^r_F$.  

\begin{prop}\label{NewReducedTessPropertiesMain} Let $\mathcal{T}, \mathcal{F}$, $\mathcal{F}^r$, and $\mathcal{T}^r_{\mathcal{F}}$ be as in Definition \ref{AssociateAuxTessToCalF}. Then
\begin{enumerate}[label=\arabic*)]
\itemsep-0.1em
\item $G'$ is not $L'$-colorable and $\mathcal{T}$ is a tessellation. In particular, each facial subgraph of $G'$ not lying in $(\mathcal{C}\setminus\mathcal{F})\cup\mathcal{F}^r$ is a triangle, and $G'$ is a closed 2-cell embedding; AND
\item For each $F^r\in\mathcal{F}^r$, each vertex of $V(F^r)$ has an $L'$-list of size at least three and $F^r$ is uniquely 4-determined in $G'$ (with respect to $L'$).
\end{enumerate}  \end{prop}

\begin{proof} We just need to check that each $F^r\in\mathcal{F}^r$ is uniquely 4-determined in $G'$, the rest is immediate. Let $F\in\mathcal{F}$ and $F^r$ be the corresponding $F$-reduction. Suppose $F^r$ is not uniquely 4-determined in $G'$. As $F$ is uniquely $N_{\textnormal{mo}}/3$-determined in $G$ (with respect to $L$) it follows that $G$ contains a separating cycle $D$ of length at most six, where $D\cap F^r$ consists of precisely one vertex and, letting $G_0\cup G_1$ be the natural $F$-partition of $G$, each $k\in\{0,1\}$ satisfies the following: Either $G_k$ contains an element of $\mathcal{C}\setminus\{F\}$ or there is a noncontractible closed curve in the open component of $\Sigma\setminus F$ whose closure is $G_k$. If $F$ is a closed ring, then, by Definition \ref{ReductionOpDefn}, we have $S_F\subseteq B_1(F)$, and we contradict Lemma \ref{ImpropGenChordClosedRing2Lem}. Thus, $F$ is an open ring and $d(D, F)\leq 2$. Since $\textnormal{Rk}(F)=2N_{\textnormal{mo}}^2$, we contradict Corollary \ref{NonRingSeparatingCor}.  \end{proof}

\section{The Face-Width of Critical Mosaics}\label{CritMosFaceWidthSec}

This section consists of the proof of the following result. 

\begin{lemma}\label{MainThmResForHighRepNearRings} Let $\mathcal{T}=(\Sigma, G, \mathcal{C}, L, C_*)$ be a critical mosaic. Then $\textnormal{fw}(G)> 4.21\beta\cdot 6^{g(\Sigma)-1}$. \end{lemma}

\begin{proof} Let $g=g(\Sigma)$ and suppose toward a contradiction that $\textnormal{fw}(G)\leq 4.21\beta\cdot 6^{g-1}$. 

\begin{claim}\label{CTildeAndknotGenChord} There exists a $C\in\mathcal{C}$ and a generalized chord $P$ of $C$ such that $|E(P)|\leq 4.21\beta\cdot 6^{g-1}$, where each cycle in $C\cup P$, except for $C$, is noncontractible.\end{claim}

\begin{claimproof} We first prove the following intermediate result.

\vspace*{-8mm}
\begin{addmargin}[2em]{0em}
\begin{subclaim}\label{IfNonConThenTildeCDCycl} If there is a noncontractible cycle $D$ of $G$ contained in at most $4.21\beta\cdot 6^{g-1}$ facial subgraphs of $G$, then there is a $C\in\mathcal{C}$ such that $|E(D\cap C)|\geq 2$ and $D$ contains no vertices of any element of $\mathcal{C}\setminus\{C\}$. \end{subclaim}

\begin{claimproof} If $V(D)$ has no intersection with $\bigcup_{C\in\mathcal{C}}V(C)$, then it follows from 1) of Proposition \ref{HighEw+Triangles=HighFw} that $\textnormal{ew}(G)\leq 4.21\beta\cdot 6^{g-1}+2$, contradicting our edge-width conditions on $\mathcal{T}$.  Thus, there is a $C\in\mathcal{C}$, where $V(C\cap D)\neq\varnothing$. Suppose there is a $C'\in\mathcal{C}\setminus\{C\}$ with $V(C'\cap D)\neq\varnothing$. Possibly $C_*\in\{C, C'\}$, but, in any case, by our distance conditions on $\mathcal{T}$, we have $d(C, C')\geq 2.9\beta\cdot 6^{g-1}$, and thus, by 2) of Proposition \ref{HighEw+Triangles=HighFw}, we get $4.21\beta\cdot 6^{g-1}\geq 2(2.9\beta\cdot 6^{g-1}-1)$, which is false. Thus, $V(D)$ intersects with no elements of $\mathcal{C}\setminus\{C\}$. If $|E(D\cap C)|\leq 1$, then we contradict our edge-width conditions on $G$, so $|E(D\cap C)|\geq 2$. \end{claimproof}\end{addmargin}

Since $\textnormal{fw}(G)\leq 4.21\beta\cdot 6^{g-1}$ and $G$ is a 2-cell embedding, there is a noncontractible cycle $D\subseteq G$ contained in at most $4.21\beta\cdot 6^{g-1}$ facial walks of $G$. Let $C$ be as in Subclaim \ref{IfNonConThenTildeCDCycl}. Now, $|E(D\cap C)|\geq 2$ and every connected component of $D\setminus E(C)$ which is not an isolated vertex is a proper generalized chord of $C$. By 1) of Proposition \ref{HighEw+Triangles=HighFw}, each such generalized chord has length at most $4.21\beta\cdot 6^{g-1}$. As $C$ is a contractible cycle, there is a component $P$ of $D\setminus E(C)$ which is a generalized chord of $C$ satisfying Claim \ref{CTildeAndknotGenChord}. \end{claimproof}

Let $C\in\mathcal{C}$ be as in Claim \ref{CTildeAndknotGenChord}. Now we consider the tessellation $\mathcal{T}^r_{C}$ and let $C^r$ be a $C$-reduction cycle, where $\mathcal{T}^r_C=(\Sigma, G', \mathcal{C}', L', C_*')$ and $\mathcal{C}'=(\mathcal{C}\setminus\{C\})\cup\{C^r\}$. Note that $\textnormal{fw}(G')\geq\textnormal{fw}(G)-6\geq 2.1\beta\cdot 6^{g-1}$. 

\begin{defn}
\emph{A \emph{knot} is a generalized chord $P^r$ of $C^r$ such that each of the cycles of $C^r\cup P^r$ which is distinct from $C^r$ is noncontractible. We say that $P^r$ is a \emph{short knot} if it is of minimal length among all knots.}
\end{defn}

It is immediate from Claim \ref{CTildeAndknotGenChord} that there exists a knot, so we now fix a short knot $P^r$. 

\begin{claim}\label{MinLengKnottedShortPathCL} $P^r$ is a proper generalized chord of $C^r$ and a shortest path in $G'$ between its endpoints, where $2.1\beta\cdot 6^{g-1}-6\leq E(P^r)|\leq 4.21\beta\cdot 6^{g-1}$.
\end{claim}

\begin{claimproof} Since $V(C^r)\subseteq B_3(C)$, we have $2.1\beta\cdot 6^{g-1}-6\leq E(P^r)|$ by our face-width conditions on $G$, and it follows from our choice of $C$ that $E(P^r)|\leq 4.21\beta\cdot 6^{g-1}$. Now, if $P^r$ is not a shortest path between its endpoints, then the minimality of $P^r$ implies that $G'$ contains a noncontractible cycle of length at most $2|E(P^r)|\leq 2(4.21\beta\cdot 6^{g-1})$, which is false, since $\textnormal{ew}(G')\geq\textnormal{ew}(G)\geq 2.1\beta\cdot 6^g$.  \end{claimproof}

We use Theorem \ref{SingleFaceConnRes} to create a smaller mosaic from $\mathcal{T}^r_C$ whose underlying surface has lower genus than $\Sigma$. Given a noncontractible closed curve $N\subseteq\Sigma$, we associate to $N$ a (not necessarily connected) surface $\Sigma_N$,  where $\Sigma_N$ is obtained from $\Sigma\setminus N$ in the following way. For each connected component $\Sigma'$ of $\Sigma\setminus N$, we glue open discs to $\Sigma'$ along $\partial(\Sigma\setminus N)$ to obtain a surface with empty boundary. The following is immediate from the definition of a knot. 

\begin{claim}\label{RepUseAContainNonContract} Let $A$ be vertex-subset of $G'$ with $V(P^r)\setminus B_2(C^r)\subseteq A\subseteq B_2(C^r\cup P^r)\cup\textnormal{Sh}^4(C^r)$, where $C^r+A$ is connected. Then there is a noncontractible closed curve $N$ of $\Sigma$, where $G'\cap N\subseteq G'[A]$.  \end{claim}

Note that $\textnormal{Sh}^4(C^r)$ is indeed well-defined, since, by 2) of Proposition \ref{NewReducedTessPropertiesMain}, $C^r$ is uniquely 4-determined in $G'$ with respect to the list-assignment $L'$.

\begin{claim}\label{MosaicTopDistanceConMainCL} Let $A$ be a subset of $V(G')$ with $V(P^r)\setminus B_2(C^r)\subseteq A\subseteq B_2(C^r\cup P^r)\cup\textnormal{Sh}^4(C^r)$, where $C^r+A$ is connected. Let $N\subseteq\Sigma$ be a noncontractible closed curve with $G'\cap N\subseteq G'[A]$ and $G'_A$ be the embedding of $G'\setminus A$ on $\Sigma_N$ in the natural way. Then there is a pair $F_0, F_1$ of facial subgraphs of $G'_A$ such that all of 1)-5) below hold, where all distances below are in $G'\setminus A$. 
\begin{enumerate}[label=\arabic*)]
\itemsep-0.1em
\item $V(F_0\cup F_1)=V(C^r\setminus A)\cup D_1(A)$; AND
\item For each $i\in\{0,1\}$ and $D\in\mathcal{C}'\setminus\{C^r, C_*'\}$, we have $d(F_i, D)\geq 2\beta\cdot 6^{g-1}+4N_{\textnormal{mo}}^2$; AND
\item Either $C=C_*$ or, for each $i\in\{0,1\}$, we have $d(F_i, C_*')\geq 2.9\beta\cdot 6^{g-2}+4N_{\textnormal{mo}}^2$; AND
\item $d(F_0, F_1)\geq\beta\cdot 6^{g-1}+4N_{\textnormal{mo}}^2$; AND
\item $\textnormal{fw}(G'_A)\geq 2.1\beta\cdot 6^{g-2}$ and $\textnormal{ew}(G'_A)\geq 2.1\beta\cdot 6^{g-1}$. 
\end{enumerate}
 \end{claim}

\begin{claimproof} We first note the following.

\vspace*{-8mm}
\begin{addmargin}[2em]{0em}
\begin{subclaim}\label{DistFromCdagger+SToRest}
\textcolor{white}{aaaaaaaaaaaaaaaaaaaaaa}
\begin{enumerate}[label=\roman*)]
\itemsep-0.1em
\item For each $D\in\mathcal{C}'\setminus\{C^r, C_*\}$, we have $d(D, C^r+A)>2\beta\cdot 6^{g-1}+4N_{\textnormal{mo}}^2$; AND
\item Either $C=C_*$ or $d(C_*', C^r+A)>2.9\beta\cdot 6^{g-2}+4N_{\textnormal{mo}}^2$.
\end{enumerate}
 \end{subclaim}

\begin{claimproof} Let $D\in\mathcal{C}'\setminus\{C^r, C_*'\}$. Since $D\in\mathcal{C}$ and $V(C^r)\subseteq B_3(C)$, we have we have $d(D, C^r)\geq 2\beta\cdot 6^g$. By assumption, $|E(P^r)|\leq 4.21\cdot 6^{g-1}$, so each vertex of $B_2(P^r)$ is of distance at most $2.11\beta\cdot 6^{g-1}$ from $D$. Thus,
 $d(D, C^r+A)\geq d(D, C^r)-2.11\beta\cdot 6^{g-1}$, so we get $d(D, C^r+A)\geq (18-2.11)\beta\cdot 6^{g-1}>2\beta\cdot 6^{g-1}+4N_{\textnormal{mo}}^2$. This proves i). Now we prove ii). Suppose that $C\neq C_*$. Thus, $C_*'=C_*$. As above, we have $d(C_*', C^r+A)\geq d(C_*, C^r)-2.11\beta\cdot 6^{g-1}$, and, by our distance conditions on $\mathcal{T}$, we have $d(C_*, C^r)\geq 2.9\beta\cdot 6^{g-1}$, so we get $d(C_*, C^r+A)\geq 0.88\beta\cdot 6^{g-1}$. Since $6\cdot (0.88)>3$, we have $d(C_*, C^r+A)\geq 2.9\beta\cdot 6^{g-2}+4N_{\textnormal{mo}}^2$. \end{claimproof}\end{addmargin}

It follows from Subclaim \ref{DistFromCdagger+SToRest}, together with our triangulation conditions from Proposition \ref{NewReducedTessPropertiesMain}, that every facial subgraph of $G'$ containing a vertex in $B_1(A)$ is either $C^r$ or a triangle. Let $U$ be the unique open component of $\Sigma\setminus C^r$ with $\partial(U)=C^r$. Since $C^r+A$ is connected, there is a unique open component $U'$ of $\Sigma$ with $U'\subseteq U$ and a unique facial subgraph $F$ of $G'\setminus A$, such that $\partial(U')=F$, where $G'\setminus A$ is regarded as an embedding on $\Sigma$ and $V(F)=V(C^r\setminus A)\cup D_1(A)$. Since $V(A)\subseteq B_2(C^r\cup P)\cup\textnormal{Sh}^4(C^r)$ and $G\cap N\subseteq A$, there is a pair $F_0, F_1$ of distinct facial subgraphs of $G'_A$ in $\Sigma_N$, where $F_0\cup F_1=F$. In particular, $F_0, F_1$ satisfy 1). Subclaim \ref{DistFromCdagger+SToRest} also shows that $F_0, F_1$ satisfy 2)-3). We now check 4) and 5). Suppose toward a contradiction that $d(F_0, F_1)<\beta\cdot 6^{g-2}+4N_{\textnormal{mo}}^2$ in $G'\setminus A$. Let $Q$ be a shortest path in $G'\setminus A$ from $F_0$ to $F_1$.  Since $G'$ is a 2-cell embedding and every facial subgraph of $G'$ not lying in $\mathcal{C}'$ is a triangle, we have the following by our edge-width and face-width conditions on $G'$:

\vspace*{-8mm}
\begin{addmargin}[2em]{0em}
\begin{subclaim}\label{PrecEndQLieEndOutQ} At most one endpoint of $Q$ lies outside of $B_3(\mathring{P}^r)$. Furthermore, letting let $x,x'$ be the endpoints of $Q$, where $x\in B_3(\mathring{P}^r)$, we have $d(x', \mathring{P}^r)>3$.  \end{subclaim}\end{addmargin}

Since $x'\not\in B_3(\mathring{P}^r)$, $x'$ either lies $V(C^r\setminus A)$ or has a neighbor in $B_2(C^r)\cup\textnormal{Sh}^4(C^r)$. Since $d(x, \mathring{P})\leq 3$, it follows that there is a $(P^r, C^r)$-path $Q^*$ which is disjoint to $P^r$ except for its lone $P^r$-endpoint, where $|E(Q^*\setminus Q)|\leq 6$. Let $y$ be the unique $P^r$-endpoint of $Q^*$ and let $R, R'$ be the two subpaths of $P^r$ such that $R\cup R'=P^r$, where $R\cap R'=y$. Each of $R+Q^*$ and $R'+Q^*$ is a knot, and, since $P^r$ is a short knot, we have $|E(P^r)|+2|E(Q^*)|\geq 2|E(P^r)|$. Thus, we get $|E(Q^*)|\geq |E(P^r)|/2$. On the other hand, we have $|E(Q^*)|\leq |E(Q)|+6<\beta\cdot 6^{g-2}+4N_{\textnormal{mo}}^2+6$. By Claim \ref{MinLengKnottedShortPathCL}, we have $|E(P^r)|/2\geq \beta\cdot 6^{g-1}$. Putting these together, we have 
$\beta\cdot 6^{g-1}<\beta\cdot 6^{g-2}+4N_{\textnormal{mo}}^2+6$, which is false, so our assumption that $d(F, F')<\beta\cdot 6^{g-2}+4N_{\textnormal{mo}}^2$ is false. This proves 4). 

Now we prove 5). It is immediate from the definition of $\Sigma_N$ that $\textnormal{ew}(G'_A)\geq\textnormal{ew}(G')\geq\textnormal{ew}(G)\geq 2.1\beta\cdot 6^{g-1}$. Suppose toward a contradiction that $\textnormal{fw}(G')<2.1\beta\cdot 6^{g-2}$. As $G'$ is a 2-cell embedding, it follows from our construction of $\Sigma_N$ that there is a noncontractible cycle $D\subseteq G'_A$ which is contained in fewer than $2.1\beta\cdot 6^{g-2}$ facial subgraphs of $G'_A$. If $V(D)$ has no intersection with $V(F_0\cup F_1)$, then $D$ is also contained in at most $2.1\beta\cdot 6^{g-2}$ facial walks of $G'$, contradicting the fact that $\textnormal{fw}(G')\geq\textnormal{fw}(G)-6$, so there is an $i\in\{0,1\}$ with $V(D\cap F_i)\neq\varnothing$, say $i=0$ for the sake of definiteness. Each facial subgraph of $G_A'$, except for those of $\{F_0, F_1\}\cup (\mathcal{C}\setminus\{C\})$, is a triangle, so it follows from 2) of Proposition \ref{HighEw+Triangles=HighFw} applied to a connected component of $G_A'$  that $V(D)$ has no intersection with any element of $\mathcal{C}\setminus\{C\}$. If $|E(D\cap F_0)|\leq 1$, then there is a family of fewer tha $2.1\beta\cdot 6^{g-2}$ facial triangles of $G_A'$ whose union contains $D$, and then it follows from 1) of Proposition \ref{HighEw+Triangles=HighFw} that $\textnormal{ew}(G'_A)<2.1\beta\cdot 6^{g-2}+2$, which is false. Thus, $|E(D\cap F_0)|>1$. By our construction of $\Sigma_N$, we get that $F_i$ is a contractible cycle of $G_A'$, so $D\neq F_0$, and furthermore, each component of $D\setminus E(F_0)$ is either an isolated vertex or a proper generalized chord of $F_0$. Furthermore, there is at least one such chord $\tilde{Q}$ of $F_0$ such that each cycle of $F_0\cup\tilde{Q}$, except for $F_0$, is noncontractible in $G'_A$. 

By 1) of Proposition \ref{HighEw+Triangles=HighFw}, $|E(\tilde{Q})|\leq 2.1\beta\cdot 6^{g-2}$. Each endpoint of $\tilde{Q}$ lies in $V(C^r\setminus A)\cup D_1(A)$. At least one endpoint of $\tilde{Q}$ lies in $B_3(\mathring{P}^r)$, or else, there is a knot $Q$ such that $|E(Q\setminus\tilde{Q})|\leq 6$, which contradicts the minimality of $P^r$, as $|E(\tilde{Q})|\leq 2.1\beta\cdot 6^{g-2}$. Thus, we let $x,x'$ be the endpoints of $\tilde{Q}$, where $x\in B_3(\mathring{P}^r)$. Analogous to Subclaim \ref{PrecEndQLieEndOutQ}, it is straightfoward to check that $d(x', P)>3$. As $x'\not\in B_3(\mathring{P}^r)$ and $x'\in V(F_0)$, either $x'\in V(C^r\setminus A)$ or $x'$ has a neighbor in $\textnormal{Sh}^4(C^r)\cup B_2(C^r)$. In any case, $G'$ contains a $(P^r, C^r)$-path $Q^*$, where $|E(Q^*\setminus\tilde{Q})|\leq 6$, and $G'$ contains a knot of length at most $\frac{|E(P)|}{2}+|E(Q^*)|$. By the minimality of $P^r$, we have $|E(Q^*)|\geq\frac{|E(P^r)|}{2}$, so, by Claim \ref{MinLengKnottedShortPathCL}, $|E(\tilde{Q})|+6\geq \beta\cdot 6^{g-1}$, which is false. This proves Claim \ref{MosaicTopDistanceConMainCL}. \end{claimproof}

Letting $P^r=w_0\cdots w_{\ell}$, it also follows from the minimality of $P^r$, together with our face-width conditions on $G'$, that $P^r\cap D_k(C^r)=\{w_k, w_{\ell-k}\}$ for each $k=0,1,2,3$. By Theorem \ref{SingleFaceConnRes} applied to $\mathcal{T}^r_C$, there exists an $A\subseteq V(G')$, and a partial $L'$-coloring $\phi$ of $A$ such that 
\begin{enumerate}[label=\arabic*)]
\itemsep-0.1em
\item $A$ is $(L', \phi)$-inert in $G'$, and each vertex of $D_1(A)\setminus\{x,y\}$ has an $L'_{\phi}$-list of size at least three; AND
\item $G'[A]$ is connected and $V(C^r)\subseteq A$ and $V(P^r)\setminus B_2(C^r)\subseteq A\subseteq B_2(C^r\cup P^r)\cup\textnormal{Sh}^4(C^r)$.
\end{enumerate}

It follows that $C^r+A$ is connected. By Claim \ref{RepUseAContainNonContract}, there is a noncontractible closed curve $N\subseteq\Sigma$ such that $G'\cap N\subseteq G'[A]$. Let $G_A'$ be the embedding of $G'\setminus A$ on $\Sigma_N$ in the natural way and $F_0, F_1$ be as in Claim \ref{MosaicTopDistanceConMainCL}.  Let $\mathcal{T}_A:=(\Sigma_N, G'_A, (\mathcal{C}\setminus\{C\})\cup\{F_0, F_1\}, L'_{\phi}, F_0)$. Note that $G'_A$ is still short-inseparable, and every facial subgraph of $G'_A$ not lying in $\mathcal{C}\setminus\{C\})\cup\{F_0, F_1\}$ is a triangle, so $\mathcal{T}_A$ is a tessellation. By Proposition \ref{NewReducedTessPropertiesMain}, $G'$ is not $L'$-colorable, so it follows from the inertness condition on $A$ that $G'\setminus A$ is not $L'_{\phi}$-colorable. By minimality, $\mathcal{T}_A$ is not a mosaic. Since $g(\Sigma_N)<g$, it follows from Claim \ref{MosaicTopDistanceConMainCL} that $\mathcal{T}_A$ satisfies all of M3)-M5), and each of $F_0, F_1$ is an open ring of $\mathcal{T}_A$ with a precolored path of length at most one (actually an empty precolored path), so $\mathcal{T}_A$ satisfies M0)-M2) as well. Thus, $\mathcal{T}_A$ is indeed a mosaic, a contradiction. This proves Lemma \ref{MainThmResForHighRepNearRings}. \end{proof}

\section{Completing the Proof of Theorem \ref{AllMosaicsColIntermRes1-4}}\label{CompleteProofMosaic}

We now bring together the work of the previous sections to prove the main result.

\begin{thmn}[\ref{AllMosaicsColIntermRes1-4}] All mosaics are colorable. \end{thmn}

\begin{proof} Suppose not. Thus, there is a critical mosaic $\mathcal{T}=(\Sigma, G, \mathcal{C}, L, C_*)$. Let $m:=\min\{d(C, C_*): C\in\mathcal{C}\setminus\{C_*\}\}$. We first check that $m$ is well-defined, i.e $\mathcal{C}\neq\{C_*\}$.

\begin{claim}\label{AcCSplitAcc1} $|\mathcal{C}|>1$. Furthermore, $m\leq 2.9\beta\cdot 6^{g-1}+6N_{\textnormal{mo}}^2$. \end{claim}

\begin{claimproof} Let $\mathcal{T}^r_{C_*}=(\Sigma', G', \mathcal{C}', L', C_*')$. By Proposition \ref{NewReducedTessPropertiesMain}, $\mathcal{T}^r_{C_*}$ is a tessellation and $G'$ is not $L'$-colorable.  Applying Lemma \ref{MainThmResForHighRepNearRings}, we get $\textnormal{fw}(G')\geq\textnormal{fw}(G)-6>2.1\beta\cdot 6^{g-1}$. As edge-width is monotone with respect to subgraphs, $\mathcal{T}^r_{C_*}$ satisfies M5). It is also immediate that $\mathcal{T}^r_{C_*}$ satisfies M0)-M2) and M4). In particular, note that $C_*'$ has an empty precolored path. By minimality, since $G'$ is not $L'$-colorable and $|V(G')|<V(G)|$, $\mathcal{T}^r_{C_*}$ is not a mosaic, so it violates M3). In particular, $|\mathcal{C}|>1$, and we indeed have the desired bound on $m$. \end{claimproof}

The natural idea now would be to consider a single ring $C$ of $\mathcal{C}\setminus\{C_*\}$ of distance $m$ from $C_*$ and then apply Theorem \ref{FaceConnectionMainResult} after applying the reduction operations to $C$ and $C_*$. However, this will not work because the resulting tessellation might not satisfy M3), since, if $C_*$ is open, then there are possibly many elements of $\mathcal{C}\setminus\{C_*\}$ which are close to $C_*$ but far from each other. We can get around this because Theorem \ref{FaceConnectionMainResult} allows us to connect many faces at once. Let $\mathcal{F}:=\{C\in\mathcal{C}: d(C, C_*)\leq 5m/4\}$ and $\mathcal{T}^r_{\mathcal{F}}=(\Sigma, G', \mathcal{C}', L', C_*')$, where $\mathcal{C}'=\mathcal{F}^r\cup (\mathcal{C}\setminus\mathcal{F})$ as a disjoint union. Note that $C_*'\in\mathcal{F}$. The pairwise-distance between any two elements of $\mathcal{F}^r$ is at most six less than the pairwise distance between the corresponding elements of $\mathcal{F}$, so we immediately have the following:

\begin{claim}\label{ElFRedPairDist} Any two elements of $\mathcal{F}^r\setminus\{C_*'\}$ are of distance at least $2\beta\cdot 6^{g}$ apart, and each element of $\mathcal{F}^r\setminus\{C_*\}$ is of distance at least $2.9\beta\cdot 6^{g-1}$ from $C_*'$. \end{claim}

Now let $\mathcal{P}=\{P_D: D\in\mathcal{F}^r\setminus\{C_*\}\}$ be a family of paths, where for each $D\in\mathcal{F}^r\setminus\{C_*'\}$, $P_D$ is a shortest $(C_*', D)$-path in $G'$. Note that each element of $\mathcal{P}$ has length at most $\frac{5m}{4}$. 

\begin{claim}\label{ElementMathPDist} For any distinct $D_0, D_1\in\mathcal{F}^r\setminus\{C_*'\}$, the graphs $D_0\cup P_{D_0}$ and $D_1\cup P_{D_1}$ are pairwise of distance at least $2.2\beta\cdot 6^{g-1}$ apart. \end{claim}

\begin{claimproof} Suppose there are $D_0, D_1$ violating Claim \ref{ElementMathPDist}. Since $P_{D_i}$ has an endpoint in $D_i$ for each $i=0,1$, we have
\vspace*{-2mm}
$$d(D_0, D_1)\leq |E(P_{D_0})|+|E(P_{D_1})|+d(D_0\cup P_{D_0}, D_1\cup P_{D_1})$$
\vspace*{-2mm}
By Claim \ref{AcCSplitAcc1}, $d(D_0, D_1)<\frac{5m}{2}+2.2\beta\cdot 6^{g-1}\leq 7.25\cdot 6^{g-1}+2.2\beta\cdot 6^{g-1}+15N_{\textnormal{mo}}^2$, contradicting Claim \ref{ElFRedPairDist}. \end{claimproof}

By Proposition \ref{NewReducedTessPropertiesMain}, for each $D\in\mathcal{F}^r$, every vertex of $F$ has an $L'$-list of size at least three and $D$ is uniquely $4$-determined in $G'$. By Theorem \ref{FaceConnectionMainResult}, there is an $A\subseteq V(G')$, and a partial $L'$-coloring $\phi$ of $A$ such that:
\begin{enumerate}[label=\arabic*)]
\item $A$ is $(L', \phi)$-inert in $G'$, and $G'[A]$ is connected and has nonempty intersection with each element of $\mathcal{F}^r$; AND
\item Each vertex of $D_1(A)$ has an $L'_{\phi}$-list of size at least three; AND
\item $A\subseteq\left(B_2(C_*')\cup\textnormal{Sh}^4(C_*')\right)\cup\bigcup_{D\in\mathcal{F}^r\setminus\{C_*'\}}\left(\textnormal{Sh}^4(D)\cup B_2(D\cup P_D)\right)$
\end{enumerate}

\begin{claim}\label{DistAEachCXMosF} For each $C\in\mathcal{C}'\setminus\mathcal{F}^r$, we have $d(A, V(C))>2.9\beta\cdot 6^{g-1}+8N_{\textnormal{mo}}^2$. \end{claim}

\begin{claimproof} Suppose there is a $C\in\mathcal{C}'\setminus\mathcal{F}^r$ violating Claim \ref{DistAEachCXMosF}. In particular, $C\in\mathcal{C}\setminus\mathcal{F}$, and one of the following holds:
\begin{enumerate}[label=\roman*)]
\itemsep-.01em
\item $d(C, C_*')\leq 2.9\beta\cdot 6^{g-1}+8N_{\textnormal{mo}}^2+2$; OR
\item For some $D\in\mathcal{F}^r\setminus\{C_*'\}$, we have $d(C, D\cup P_D)\leq 2.9\beta\cdot 6^{g-1}+8N_{\textnormal{mo}}^2+2$
\end{enumerate}

If i) holds, then,  since $V(C_*')\subseteq B_3(C_*)$, we have $d(C, C_*)\leq\frac{5m}{4}$, which is false, as $C\not\in\mathcal{F}$, so ii) holds. By definition, there is an $F\in\mathcal{F}$ with $V(D)\subseteq B_3(F)$, where $D$ is an $F$-reduction cycle. As $P_D$ has length at most $\frac{3m}{2}$, we get $d(C, F)\leq 2.9\beta\cdot 6^{g-1}+8N_{\textnormal{mo}}^2+\frac{3m}{2}+8$. Since $m\leq 2.9\beta\cdot 6^{g-1}$, this contradicts M4) applied to $\mathcal{T}$. \end{claimproof}

As $G'$ is not $L'$-colorable, it follows from the inerntness condition on $A$ that $G'\setminus A$ is not $L'_{\phi}$-colorable. Since every facial subgraph of $\mathcal{T}'$ not lying in $\mathcal{C}'$ is a triangle, and since $A$ has nonempty intersection with the vertex-set of each element of $\mathcal{F}^r$, it follows from Claim \ref{DistAEachCXMosF} that is a unique facial subgraph $X$ of $G'\setminus A$ such that \begin{equation}\tag{Eq2}\label{XDecomp}V(X)=D_1(A)\cup\bigcup_{D\in\mathcal{F}^r} V(D\setminus A)\end{equation}

Let $\mathcal{T}_A:=(\Sigma, G'\setminus A, (\mathcal{C}\setminus\mathcal{F})\cup\{X\}, L'_{\phi}, X)$. Since $G'$ is a 2-cell embedding, $G'\setminus A$ is also a 2-cell embedding.  We claim that $\mathcal{T}_A$ is a mosaic, where $X$ is an open ring with empty precolored path.

\begin{claim}\label{CheckM3ForFinalTess} For each $C\in\mathcal{C}\setminus\{F\}$, we have $d(X, C)\geq 2.9\beta\cdot 6^{g-1}+6N_{\textnormal{mo}}^2$ \end{claim}

\begin{claimproof} Suppose here is a $C\in\mathcal{C}\setminus\{F\}$ with $d(C, X)<2.9\beta\cdot 6^{g-1}+6N_{\textnormal{mo}}^2$. Thus, by Claim \ref{DistAEachCXMosF}, there is no shortest $(C,X)$-path in $G'\setminus A$ whose $X$-endpoint is in $D_1(A)$, so, by our decomposition in (\ref{XDecomp}), there is a $D\in\mathcal{F}^r$ such that $D\setminus A\neq\varnothing$ and $d(C, D\setminus A)<2.9\beta\cdot 6^{g-1}+6N_{\textnormal{mo}}^2$. Thus, by definition, there is an $F\in\mathcal{C}$ such that $D$ is an $F$-reduction cycle, where $F\neq C$ and $V(D)\subseteq B_3(F)$, and so $d(C, F)<2.9\beta\cdot 6^{g-1}+6N_{\textnormal{mo}}^2+3$. Since $\mathcal{T}$ satisfies M3), this implies that $F=C_*$ and thus, by definition of $\mathcal{F}$, we have $C\in\mathcal{F}$, which is false. \end{claimproof}

\begin{claim}\label{BoundOnFwofG'FinalStep} $\textnormal{fw}(G')>4.21\beta\cdot 6^{g-1}-6$ \end{claim}

\begin{claimproof} Suppose not. As $G'$ is a 2-cell embedding, there is a noncontractible cycle $N$ of $G'$ which is contained in at most $4.21\beta\cdot 6^{g-1}-6$ facial subgraphs of $G'$. By Claim \ref{ElFRedPairDist}, together with our distance conditions on the inner rings of $\mathcal{T}$, we get that the elements of $\mathcal{C}'$ are pairwise of distance at most $2.9\beta\cdot 6^{g-1}$ apart, so, by 2) of Proposition \ref{HighEw+Triangles=HighFw} applied to the 2-cell embedding $G'$, $N$ intersects with at most one element of $\mathcal{C}'$. Thus, there is a noncontractible cycle of $G$ which is contained in at most $4.21\beta\cdot 6^{g-1}$ facial subgraphs of $G$, contradicting Lemma \ref{MainThmResForHighRepNearRings}. \end{claimproof}

\begin{claim}\label{FWReduGrG'MinA} $\textnormal{ew}(G'\setminus A)\geq 2.1\beta\cdot 6^g$ and $\textnormal{fw}(G'\setminus A)\geq 2.1\beta\cdot 6^{g-1}$. \end{claim}

\begin{claimproof} We proceed analogously to Lemma \ref{MainThmResForHighRepNearRings}. We say that a generalized chord $R$ of $X$ is a \emph{knot} of $X$ if each of the cycles in $X\cup R$ whcih contain $R$ is noncontractible. We have $\textnormal{ew}(G'\setminus A)\geq\textnormal{ew}(G')\geq\textnormal{ew}(G)\geq 2.1\beta\cdot 6^g$, so we just need to check the face-width condition. Suppose that $\textnormal{fw}(G'\setminus A)<2.1\beta\cdot 6^{g-1}$. Since $G'\setminus A$ is a 2-cell embedding, there is a noncontractible cycle $N\subseteq G'$ contained in fewer than $2.1\beta\cdot 6^{g-1}$ facial walks of $G'\setminus A$.  Recall that, by Lemma \ref{MainThmResForHighRepNearRings}, $\textnormal{fw}(G')\geq\textnormal{fw}(G)-6\geq 4.21\beta\cdot 6^{g-1}$, so $V(N)$ has nonempty intersection with $V(X)$, as $X$ is the only facial subgraph of $G'\setminus A$ which is not also a facial subgraph of $G'$. Now, it follows from Claim \ref{CheckM3ForFinalTess}, together with 2) of Proposition \ref{HighEw+Triangles=HighFw}, that $V(N)$ has no intersection with any of the elements of $\mathcal{C}\setminus\mathcal{F}$, so there is a knot $R$ with $|E(R)|\leq 2.1\beta\cdot 6^{g-1}$. We also suppose for convenience that $R$ is of minimal length among all the knots of $X$. Note that, by our decomposition of $A$, any neighbors in $A$ of the endpoints in $R$ lie in $B_2(C_*')\cup\bigcup_{D\in\mathcal{F}^r\setminus\{C_*'\}} B_2(D\cup P_D)$, since the minimality of $|E(R)|$ implies that $R$ contains no vertices of $\bigcup_{D\in\mathcal{F}^r}\textnormal{Sh}^4(D)\setminus B_2(D)$. 

Let $x, x'$ be the endpoints of $R$. By our decomposition (\ref{XDecomp}), together with Claim \ref{ElFRedPairDist}, we get that there is a lone $D^{\dagger}\in\mathcal{F}^r\setminus\{C_*\}$ such that the endpoints of $R$ lie in $D_1(A)\cup V(D^{\dagger}\setminus A)$. It follows from Claim \ref{ElementMathPDist}, together with our decomposition of $A$, that the endpoints of $R$ lie in $B_3(C_*'\cup D^{\dagger}\cup P_{D^{\dagger}})$. Since $d(C_*', D^{\dagger})\geq $, it follows that $R$ has at least one endpoint in $B_3(P_{D^{\dagger}})$. Note that $R$ has precisely one endpoint in $B_3(P_{D^{\dagger}})$, or else, since $|E(P_{D^{\dagger}})|\leq\frac{5m}{4}$, it follows that $G'$ contains a noncontractible cycle of length at most $\frac{5m}{4}+2.1\beta\cdot 6^{g-1}+6$, and then our bound on $m$ contradicts our edge-width conditions. Thus, we suppose that $x\in B_3(P_{D^{\dagger}})$ and $x'\in B_3(C_*'\cup D^{\dagger})$. Since $P_{D^{\dagger}}$ is a shortest $(C_*', D^{\dagger})$-path in $G'$, it follows that there is a noncontractible cycle in $G'$ which is contained in at most $2|E(R)|+10$ facial subgraphs of $G'$, so $\textnormal{fw}(G')\leq 2|E(R)|+10\leq 4.2\beta\cdot 6^{g-1}+10$, contradicting Claim \ref{BoundOnFwofG'FinalStep}. \end{claimproof}

It follows from Claim \ref{FWReduGrG'MinA} that $\mathcal{T}_A$ satisfies M5). Likewise, M4) is inherited from $\mathcal{T}$. Since $X$ has an empty precolored path, $\mathcal{T}_A$ satisfies all of M0)-M2). By Claim \ref{CheckM3ForFinalTess}, $\mathcal{T}_A$ satisfies M3) as well, so $\mathcal{T}_A$ is indeed a mosaic. By minimality, $G'\setminus A$ is $L'_{\phi}$-colorable, which is false, as indicated above. This completes the proof of Theorem \ref{AllMosaicsColIntermRes1-4}. \end{proof}

\end{document}